\documentclass[letterpaper]{amsart}
\usepackage[utf8x]{inputenc}
\usepackage[T1,T2A]{fontenc}
\usepackage[english]{babel}
\usepackage{graphicx}
\usepackage{longtable}
\usepackage{wrapfig}
\usepackage{rotating}
\usepackage[normalem]{ulem}
\usepackage{amsmath}
\usepackage{amssymb}
\usepackage{capt-of}
\usepackage[bookmarks]{hyperref}
\usepackage{amsthm}
\usepackage{amsfonts}
\usepackage{mathtools}
\usepackage{geometry}
\usepackage{bm}
\usepackage[warn]{mathtext}
\usepackage[bbgreekl]{mathbbol}
\usepackage{framed}
\usepackage[dvipsnames]{xcolor}
\usepackage[shortlabels]{enumitem}
\usepackage{cancel}
\usepackage{centernot}
\allowdisplaybreaks[4]

\newtheorem{theorem}{Theorem}
\newtheorem{lemma}{Lemma}

\newtheorem{remark}{Remark}

\newtheorem{corollary}{Corollary}

\DeclarePairedDelimiter\ceil{\lceil}{\rceil}
\DeclarePairedDelimiter\floor{\lfloor}{\rfloor}

\DeclareMathOperator\var{Var}
\DeclareMathOperator\diag{diag}

\makeatletter
\def\namedlabel#1#2{\begingroup
  #2%
  \def\@currentlabel{#2}%
  \phantomsection\label{#1}\endgroup
}
\makeatother
\makeatletter
\def\subsubsection{\@startsection{subsubsection}{3}%
  \z@{.5\linespacing\@plus.7\linespacing}{.1\linespacing}%
  {\normalfont\itshape}}
\def\subsection{\@startsection{subsection}{2}
  \z@{.5\linespacing\@plus.7\linespacing}{.3\linespacing}
  {\normalfont\bfseries}}
\makeatother

\usepackage[style=ieee,backend=bibtex]{biblatex}
\author[Ievlev]{Pavel Ievlev}
\address{University of Lausanne, Lausanne, Switzerland}
\email{ievlev.pn@gmail.com}
\author[Kriukov]{Nikolai Kriukov}
\address{Korteweg-de Vries Institute, University of Amsterdam,
POSTBUS  94248, 1090 GE  Amsterdam, The Netherlands.}
\email{n.kriukov@uva.nl}
\date{}
\title{Extremes of vector-valued locally additive Gaussian fields with application to double crossing probabilities}
\hypersetup{
 pdfauthor={Pavel Ievlev, Nikolai Kriukov},
 pdftitle={Extremes of vector-valued locally additive Gaussian fields with application to double crossing probabilities},
 pdfkeywords={},
 pdfsubject={},
 pdfcreator={Emacs 29.1 (Org mode 9.7)}, 
 pdflang={English}}
\begin{document}

\maketitle
\begin{abstract}
The asymptotic analysis of high exceedance probabilities for Gaussian processes
and fields has been a blooming research area since J.~Pickands introduced
the now-standard techniques in the late 60's. The \emph{vector-valued} processes,
however, have long remained out of reach due to the lack of some key tools
including Slepian's lemma, Borell-TIS and Piterbarg inequalities. In a 2020
paper by K.~D\c{e}bicki, E.~Hashorva and L.~Wang, the authors
extended the double-sum method to a large class of vector-valued processes, both
stationary and non-stationary. In this contribution we make one step forward,
extending these results to a simple yet rich class of non-homogenous
vector-valued Gaussian \emph{fields}. As an application of our findings, we present an
exact asymptotic result for the probability that a real-valued process first
hits a high positive barrier and then a low negative barrier within a finite
time horizon.
\end{abstract}
\section{Introduction}
\label{sec:org218959c}
\label{SEC:intro}
The asymptotic analysis of high exceedance probabilities for Gaussian processes
and fields has been a blooming research area for several decades. The most
classical results due to J. Pickands~\cite{Pickands1967,Pickands1969} give the
asymptotics of
\begin{equation*}
\mathbb{P}
\left\{ \exists \, t \in [ 0, T ] \colon X ( t ) > u \right\}, \quad \text{as}
\quad u \to \infty,
\end{equation*}
when \(X ( t )\) is a centered Gaussian process on \([ 0, T ]\) with values
in \(\mathbb{R}\) and satisfying some standard assumptions. These results have since
been extended in numerous directions. Among these, little was known about
similar problems for vector-valued processes until very recently. A deep
contribution~\cite{DebickiEtAl2019} paved a way to the asymptotic analysis
of the probabilities
\begin{equation*}
\mathbb{P} \left\{
\exists \, t \in [ 0, T ] \
\forall \, i \in \{ 1, \ldots, d \} \colon
X_i ( t ) > u \, b_i
\right\},
\quad \text{as} \quad u \to \infty,
\end{equation*}
for a centered \(\mathbb{R}^d\)-valued Gaussian process \(\bm{X} ( t )\),
\(t \in [ 0, T ]\) and \(\bm{b} \in \mathbb{R}^d\) with at least one positive component.
As the authors point out, even the seemingly trivial case of centered \(\bm{X}\)
 with independent components is quite challenging
(see~\cite{VectorValuedKamilTabis2015,AzaisPham2019,Pham2019}). We want
to mention in passing the non-centered vector-valued high exceedance problems,
which were initially studied for linear transformations of \(\mathbb{R}^d\)-valued
Brownian motion in~\cite{DebickiEtAl2018a,DebickiEtAl2018} and have
recently been extended to linear transformations of stationary increments
processes with independent components satisfying Berman condition in
\cite{BisewskiKriukov2021simultaneous}. The approximations of high exceedance
probabilities of vector-valued processes appear naturally in various
applications including statistics, ruin theory and queueing theory, see e.g.,
\cite{AzaisPham2019,WorsleyFriston2000,PanBorovkov2019,BorovkovPalmowski2019,DelsingMandjesSpreij2018}
\bigskip

Another direction in which the classical theorems may be extended involves high
exceedances of Gaussian random fields \(X ( \bm{t} )\), \(\bm{t} \in E \subset
\mathbb{R}^n\).
Deep results of this type are known since at least the 70's and
some of them are presented in the well-known monograph by
Piterbarg~\cite{PiterbargBook}. See~\cite{ManifoldPaperBaiPeng}
and~\cite{PiterbargManifold2021} for some recent developments.
\bigskip

In the current contribution we prove some theorems related to the mixture
of the aforementioned generalizations, that is, to centered Gaussian random
vector fields \(\bm{X} ( \bm{t} )\), \(\bm{t} \in \mathbb{R}^n\) taking values in
\(\mathbb{R}^d\) under some simplifying assumptions. More specifically,
the vector fields we consider behave near the most likely point of high
exceedance as sums of independent vector fields, each of which depends on
one coordinate of \(\bm{t}\):
\begin{equation*}
\bm{X} ( \bm{t} ) \approx
\bm{X}_1 ( t_1 ) + \bm{X}_2 ( t_2 ) + \ldots + \bm{X}_n ( t_n ),
\qquad
\bm{t} \approx \bm{t}_{ * }.
\end{equation*}
The exact meaning of this ``\(\approx\)'' sign is described by
Assumption~\ref{A2} below.
In developing these assumptions, our aim was to find the simplest yet fecund
extension of the paper~\cite{DebickiEtAl2019} to the case of
miltidimensional parameter.
There are two recent papers we want to mention in this regard.
In~\cite{2dNonSimBrownian2020} and~\cite{Krystecki2021finiteTime}
the so-called non-simultaneous ruin probability of a pair of correlated Brownian
motions with linear trends
\begin{equation*}
\mathbb{P} \left\{
\exists \, ( t, s ) \in [ 0, T ]^2 \colon
B_1 ( t ) - \mu_1 t > u, \ B_2 ( s ) - \mu_2 s > u
\right\}
\end{equation*}
was studied in the infinite horizon case \(T = \infty\) and finite horizon case \(T < \infty\) correspondingly. Although this probability may well be rewritten as a
ruin of two dimensional vector field, our local additivity assumptions are not
met in this setup, hence these two papers remain out of reach of our
Theorem~\ref{main-theorem}. Similar results have been obtained
in~\cite{MR3624871} for a class \(\mathbb{R}^2\)-valued locally-stationary
Gaussian random fields indexed by \(\mathbb{R}^n\). See
also~\cite{cheng2023expected} for some recent developments in the case of
smooth vector-valued Gaussian fields.
\bigskip

As an application of our findings, we give an asymptotic formula for the
probability that a one-dimensional stationary process \(X ( t )\),
\(t \in [ 0, T ]\) hits two distant barriers: one above and one below its
starting point. Namely, we derive a precise approximation of
\begin{equation} \label{def:double_crossing_probability}
\mathbb{P} \left\{
\exists \, t, \, s \in [ 0, T ] \colon
X ( t ) > a u, \
X ( s ) < -b u
\right\},
\qquad a, \, b \geq 0.
\end{equation}
The probability in~\eqref{def:double_crossing_probability} can be
conveniently rewritten in the vector notation as
\begin{equation*}
  \mathbb{P} \left\{ \exists \, \bm{t} \in [ 0, T ]^2 \colon \bm{X} ( \bm{t} ) > u \bm{b} \right\},
  \quad
  \bm{X} ( \bm{t} ) = \left( X ( t_1 ), \ -X ( t_2 ) \right)^\top,
  \quad
  \bm{b} = ( a, b )^\top.
\end{equation*}
If further the correlation function
\(\rho ( t ) = \mathbb{E} \left\{ X ( 0 ) X ( t ) \right\}\)
is positive, strictly smaller than \(1\), and satisfies
\begin{equation*}
\rho ( t ) \sim 1 - \vartheta \, t^{\alpha} + o \left( t^{\alpha} \right)
\quad \text{as} \quad
t \to 0,
\end{equation*}
with some \(\vartheta > 0\) and \(\alpha \in (0, 2]\), the problem falls within the
scope of our assumptions after some minor tweaks.
In view of Theorem~\ref{double-crossing-theorem} in
Section~\ref{SEC:double_crossing} we have
\begin{equation*}
\mathbb{P} \left\{
\exists \, t, \, s \in [ 0, T ] \colon
X ( t ) > a u, \
X ( s ) < -b u
\right\}
\sim c \, u^{\max \{ 0, 4 / \alpha - 2 \}} \,
\mathbb{P} \left\{
X ( 0 ) > a u, \
X ( T ) < -bu
\right\},
\end{equation*}
where the constant \(c\) is of Pickands type if \(\alpha < 1\), of Piterbarg
type if \(\alpha = 1\) and equals \(2\) if \(\alpha > 1\).

We now recapitulate the main ingredients of our approach and emphasize a few
things which, in our opinion, are worth mentioning.

First, it is known that the high exceedance event of a vector-valued Gaussian
process \(\bm{X} ( t )\), \(t \in [ 0, T ]\) is most likely to happen
near the maximizer of the so-called generalized variance function, defined as
\begin{equation*}
\sigma_{\bm{b}}^{-2} ( t )
= \min_{\bm{x} \, \geq \, \bm{b}} \bm{x}^\top \, \Sigma^{-1} ( t ) \, \bm{x},
\end{equation*}
where \(\Sigma ( t ) = R ( t, t )\) and
\(R ( t, s ) = \mathbb{E} \left\{ \bm{X} ( t ) \, \bm{X} ( s )^\top \right\}\)
is the covariance matrix of the process \(\bm{X} ( t )\).
A similar function with \(\bm{t} \in [ 0, T ]^n\) instead of \(t \in [ 0, T ]\)
plays the same role in the case of random fields.
The asymptotic analysis of the high exceedance probability usually begins with
showing that the probability that the overshoot happens outside some small
vicinity of this maximizer is negligible.

Secondly, in our approach to the issues arising from the dimensionality of the
process, we closely follow the paper~\cite{DebickiEtAl2019}, but there are
two important differences. To explain the first, let us briefly reproduce here
Assumption \textbf{D2} of the paper. Let \(R ( t, s )\) be the covariance matrix of
a Gaussian process \(X ( t )\), \(t \in [ 0, T ]\) and denote
\(\Sigma ( t ) = R ( t, t )\). Then, Assumption \textbf{D2} demands that for all
\(t \in [ 0, T ]\) there be a continuous \(d \times d\) matrix-function
\(A ( t )\), such that
\begin{equation*}
\Sigma ( t ) = A ( t ) \, A ( t )^\top
\end{equation*}
and with some \(d \times d\) real matrix \(\Xi\) and \(\beta' > 0\) holds
\begin{equation}\label{intro:A-expansion}
A ( t ) = A ( t_0 )
-| t - t_0 |^{\beta'} \, \Xi
+o \left( | t - t_0 |^{\beta'} \right)
\end{equation}
as \(t \to t_0\), where \(t_0\) is the point which maximizes the
generalized variance \(\sigma_{\bm{b}}^2 ( t )\).
From a practical point of view, it is not always easy to compute \(A ( t )\)
starting with \(\Sigma ( t )\). This is why instead of assuming something similar
to \textbf{D2} and \textbf{D3} (the latter also requires knowing \(A ( t )\) along with its
inverse), we impose assumptions directly on the asymptotic expansion of the
covariance matrix \(R ( t, s )\) for \(t\) and \(s\) close to
\(t_{ * }\) (see Assumption~\ref{A2}).

Let us now explain the second difference. In~\cite{DebickiEtAl2019} the
authors assumed that~\eqref{intro:A-expansion} is satisfied with \(\Xi\)
and \(A ( t_0 )\) such that \(\bm{w}^\top \Xi \, A^\top ( t_0 ) \, \bm{w} > 0\),
where \(\bm{w}\) is some specific vector. As it turns out, this assumption is
rather strong. To lift it, one may consider extending the
expansion~\eqref{intro:A-expansion} to the second order, namely
\begin{equation}\label{intro:A-expansion-2}
A ( t ) = A ( t_0 )
-| t - t_0 |^{\beta'} \, \Xi
-| t - t_0 |^{\beta} \, \Theta
+o \left( | t - t_0 |^{\beta} \right),
\end{equation}
where \(\bm{w}^\top \Xi \, A^\top ( t_0 ) \, \bm{w} = 0\), but
\(\bm{w}^\top \Theta \, \bm{w} > 0\). The precise conditions under which this extension
is possible are presented by Assumptions~\ref{A2.3} to~\ref{A2.6}.

In regard to the techniques used to work with the vector-valued setting of this
contribution, we refer our reader to the introduction
of~\cite{DebickiEtAl2019} where the authors describe in detail in what
aspects and why this case is much different from the one-dimensional. Here we
mention in passing that some of the tools crucial for the one-dimensional case
are not available in the multivariate setup (such as the Slepian lemma), while
others (such as the Borell-TIS \& Piterbarg inequalities) have been successfully
extended to this case.\newline

\textbf{Brief organization of the paper}. Main results of the paper are presented in
Section~\ref{SEC:main_results} with proofs relegated to
Section~\ref{SEC:proofs}. The asymptotics of the double crossing
probabilities are presented in Section~\ref{SEC:double_crossing} with
proofs relegated to Appendix. Section~\ref{SEC:auxiliary_results} contains
several auxiliary results, most of which are taken
from~\cite{DebickiEtAl2019} and reproduced here in an adapted form and
without proofs for the reader's convenience. We conclude this section by
introducing some notation.\newline

\textbf{Subscripts}.
Throughout the rest of the paper, the subscript \(u\) on any
scalar-, vector or matrix-valued function \(f\) defined on \(\mathbb{R}^n\) or
\(\mathbb{R}^n \times \mathbb{R}^n\) means, unless specified otherwise, its rescaling by a factor of
\(u^{-2/\bm{\nu}}\), that is,
\(f_u ( \bm{t} ) = f ( u^{-2/\bm{\nu}} \bm{t} )\) or \(f_u ( \bm{t}, \bm{s} ) = f ( u^{-2/\bm{\nu}} \bm{t}, u^{-2/\bm{\nu}} \bm{s} )\), where
\(\bm{\nu}\) is defined in~\eqref{def:index-sets-2}.

\textbf{Vectors}. All vectors (and only them) are written in bold letters, for instance
\(\bm{b} = ( b_1, \ldots, b_d )^\top\), \(\bm{1} = ( 1, \ldots, 1 )^\top\) and \(\bm{0} = (
0, \ldots, 0 )^\top\). If \(\mathcal{I} \subset \{ 1, \ldots, n \}\) and \(\bm{b} \in \mathbb{R}^n\), by \(\bm{b}_{\mathcal{I}}\) we mean \(( b_i )_{i \in \mathcal{I}} \in \mathbb{R}^{| \mathcal{I} |}\) or, by notation abuse,
its extension to \(\mathbb{R}^n\) by zeroes: \(b_i = 0\) for \(i \in \mathcal{I}^c\). Unless
specified otherwise, all operations on vectors are performed componentwise. For
example, \(\bm{a} \, \bm{b}\) with \(\bm{a}, \, \bm{b} \in \mathbb{R}^n\) denotes
componentwise producs: \(( \bm{a} \bm{b} )_i = a_i b_i\). Similarly for \(\bm{a} / \bm{b}\), \(e^{\bm{a}}\), or \(\floor*{\bm{a}}\), denoting \(a_i
/ b_i\), \(e^{a_i}\) and \(\floor*{a_i}\) correspondingly. We write \(\bm{a} \geq \bm{b}\) if \(a_i \geq b_i\) for all \(i \in \{ 1, \ldots, n \}\).

\textbf{Matrices}.
If \(A = ( A_{ij} )_{1 \leq i, \, j \leq d}\) is a \(d \times d\) matrix,
we shall write \(A_{IJ}\) for the submatrix \(( A_{ij} )_{i \in I, \, j \in J}\).
If \(I = J\), we shall occasionally write \(A_I\) instead of \(A_{II}\).
\(\left\| A \right\|\) denotes any fixed norm in the space of \(d \times d\)
matrices. Our formulae shall not depend on the choice of the norm.
For \(\bm{w} \in \mathbb{R}^d\), \(\diag ( \bm{w} )\) stands for the
diagonal matrix with entries \(w_1, \ldots, w_d\) on the main diagonal.

\textbf{Asymptotic equivalence.} If \((X, d_X)\) is a metric space,
\((N, \| \bm{\cdot} \|_N)\) and \((H, \| \bm{\cdot} \|_H )\) are normed spaces,
\(f, g \colon M \to N\) and \(h \colon M \to H\), we write
``\(f = g + o ( h )\) as \(x \to x_0\)'' if for every \(\varepsilon > 0\) there exists
some \(\delta > 0\) such that
\begin{equation*}
d_X ( x, x_0 ) < \delta \implies \| f ( x ) - g ( x ) \|_N \leq \varepsilon \left\| h ( x ) \right\|_H.
\end{equation*}
In particular, this convention will be frequently used with \(X = \mathbb{R}^n\),
\(N = \mathbb{R}^{d \times d}\) the space of matrices with Frobenius norm, \(H = \mathbb{R}\) with
standard distance \(| \bm{\cdot} |\) or \(H = \mathbb{R}^{d \times d}\), again with Frobenius
norm.

\textbf{Quadratic programming problem}. Let \(\Sigma\) be a \(d \times d\) real matrix with
inverse \(\Sigma^{-1}\). If \(\bm{b} \in \mathbb{R}^d \setminus ( -\infty, 0 ]^d\), then by Lemma
\ref{lemma:QPP} the quadratic programming problem \(\Pi_{\Sigma} ( \bm{b} )\)
\begin{equation*}
\Pi_{\Sigma} ( \bm{b} ) \colon
\text{minimize} \
\bm{x}^\top \, \Sigma^{-1} \, \bm{x} \
\text{under the linear constraint} \
\bm{x} \geq \bm{b}
\end{equation*}
has a unique solution \(\widetilde{\bm{b}} \geq \bm{b}\) and there exists a
unique non-empty index set \(I \subset \{ 1, \ldots, d \}\) such that
\begin{equation*}
\widetilde{\bm{b}}_I = \bm{b}_I,
\qquad
\widetilde{\bm{b}}_J = \Sigma_{IJ} ( \Sigma_{II} )^{-1} \, \bm{b}_I
\geq \bm{b}_J,
\qquad
\bm{w}_I = ( \Sigma_{II} )^{-1} \, \bm{b}_I > \bm{0}_I,
\qquad
\bm{w}_J = \bm{0}_J,
\end{equation*}
where \(\bm{w} = \Sigma^{-1} \widetilde{\bm{b}}\), where coordinates
\(J = \{ 1, \ldots, d \} \setminus I\) are responsible for the dimension-reduction
phenomena, while coordinates belonging to \(I\) play an essential role in
the exact asymptotics.

\textbf{Other notation}.
We use lower case constants \(c_1, \, c_2, \, \ldots\) to denote generic constants
used in the proofs, whose exact values are not important and can be changed from
line to line. The labeling of the constants starts anew in every proof.
Similarly, \(\epsilon_1, \, \epsilon_2, \, \ldots\) denote error terms, that is, functions of various variables
which are small in some specific sense, always described near the point where
they are introduced. Their labeling also starts anew in every proof.
\section{Main results}
\label{sec:org314c2ec}
\label{SEC:main_results}
Let \(\bm{X} ( \bm{t} )\), \(\bm{t} \in [ \bm{0}, \bm{T} ]\), \(\bm{T} > \bm{0}\) be a
non-stationary centered Gaussian random field with continuous sample paths.
Define two matrix-valued functions by
\[
R ( \bm{t}, \bm{s} )
= \mathbb{E} \left\{
\bm{X} ( \bm{t} ) \, \bm{X} ( \bm{s} )^\top
\right\},
\qquad
\Sigma ( \bm{t} ) = R ( \bm{t}, \bm{t} )
\]
and assume that \(\Sigma ( \bm{t} )\) is non-singular. Set \(\Sigma = \Sigma ( \bm{0} )\).
It is known that the function \(\sigma_{\bm{b}}^2 ( \bm{t} )\), defined by
\begin{equation}
\label{def:gen-variance}
\sigma_{\bm{b}}^{-2} ( \bm{t} )
= \min_{\bm{x} \, \geq \, \bm{b}}
\bm{x}^\top \, \Sigma^{-1} ( \bm{t} ) \, \bm{x}
\end{equation}
and further refered to as the generalized variance, plays a similar role in the
multivariate setup to that of the usual variance in the one-dimensional case.
More precisely, the high exceedance event usually happens near the maximizer of
\(\sigma_{\bm{b}}^{-2} ( \bm{t} )\).
The asymptotics then is determined by the behaviour of \(R ( \bm{t}, \bm{s} )\)
and \(\Sigma ( \bm{t} )\) near the this maximiser.
Let \(\bm{b} ( \bm{t} )\) denote the vector
which minimizes~\eqref{def:gen-variance}.
We shall assume that:
\begin{description}[leftmargin = * ]
\item[\namedlabel{A1}{A1}] \(\sigma_{\bm{b}}^2 ( \bm{t} )\) attains its unique
maximum at \(\bm{t}_{ * } = \bm{0}\).

\item[\namedlabel{A2}{A2}] There exist
\begin{enumerate}
\item collections of real \(d \times d\) matrices
\(( A_{k, i} )_{i = 1, \ldots, n}\), \(k = 1, \ldots, 5\) and
\(( A_{6, i, j} )_{i, j = 1, \ldots, n}\)
\item vectors \(\bm{\beta}', \, \bm{\beta} \in \mathbb{R}^n_{ + }\) satisfying
\(\bm{0} < \bm{\beta}' < \bm{\beta} \leq 2 \bm{\beta}'\)
\item a vector \(\bm{\alpha} \in ( 0, 2 )^n\)
\end{enumerate}
such that
\begin{equation}\label{A2.1}\tag{A2.1}
\Sigma - R ( \bm{t}, \bm{s} )
\sim
\sum_{i = 1}^n \Big[
A_{1, i} \, t_i^{\beta_i'}
+A_{2, i} \, t_i^{\beta_i}
+A_{3, i} \, s_i^{\beta_i'}
+A_{4, i} \, s_i^{\beta_i}
+S_{\alpha_i, A_{5, i}} ( t_i - s_i )
\Big]
+\sum_{i, \, j = 1}^n A_{6, i, j} \, t_i^{\beta_i'} \, s_j^{\beta_j'},
\end{equation}
where
\begin{equation} \label{def:S}
S_{\alpha, V} ( t )
\coloneqq
\left| t \right|^{\alpha} \Big( V \mathbb{1}_{t \geq 0} + V^\top \mathbb{1}_{t < 0} \Big)
\end{equation}
and \(\sim\) means that the error \(\epsilon\) satisfies
\begin{equation}\label{A2.2}\tag{A2.2}
\epsilon
= o \left(
\sum_{i = 1}^n \Big[ t_i^{\beta_i} + s_i^{\beta_i} + | t_i - s_i |^{\alpha_i} \Big]
\right)
\quad \text{as} \quad ( \bm{t}, \bm{s} ) \downarrow ( \bm{0}, \bm{0} ).
\end{equation}
Denote
\begin{equation}\label{def:index-sets-1}
\mathcal{F} \coloneqq \left\{ i \in \{ 1, \ldots, n \} \colon 2 \beta_i' = \beta_i \right\},
\qquad
\mathcal{I} \coloneqq \left\{ i \in \{ 1, \ldots, n \} \colon \alpha_i < \beta_i \right\},
\end{equation}
and assume further that
\begin{align}
\label{A2.3}\tag{A2.3}
& A_{1, i} \, \bm{w} = \bm{0}
\quad \text{and} \quad
\bm{b} ( \bm{t} ) - \bm{b} ( \bm{0} )
= o \left( \sum_{i = 1}^n t_i^{\beta_i'} \right)
& \text{as} \ \bm{t} \to \bm{0},
\\[7pt]
\label{A2.4}\tag{A2.4}
& \xi_i \coloneqq \bm{w}^\top A_{2, i} \, \bm{w} > 0
& \text{for all} \ i \in \{ 1, \ldots, n \},
\\[7pt]
\label{A2.5}\tag{A2.5}
& \varkappa_i \coloneqq \bm{w}^\top A_{5, i} \, \bm{w} > 0
& \text{for all} \ i \in \mathcal{I}.
\end{align}
Finally, define a block matrix \(D = ( D_{i, j} )_{i, j \in \mathcal{F}}\), each block of
which is a \(d \times d\) matrix given by
\begin{equation*}
D_{i, j} \coloneqq A_{6, i, j} + A_{1, i} \, \Sigma^{-1} A_{1, j}^\top,
\quad i, \, j \in \mathcal{F},
\end{equation*}
and assume that it is positive definite:
\begin{equation}
\label{A2.6}\tag{A2.6}
D \succcurlyeq 0.
\end{equation}

\item[\namedlabel{A3}{A3}] There exist \(\bm{\gamma} \in (0, 2]^n\) and \(C > 0\), such that
for all \(\bm{t}, \bm{s}\)
\begin{equation} \label{Holder} \tag{A3}
  \mathbb{E} \left\{ \left| \bm{X} ( \bm{t} ) - \bm{X} ( \bm{s} ) \right|^2 \right\}
  \leq C \sum_{i = 1}^n | t_i - s_i |^{\gamma_i}.
\end{equation}
\end{description}

We shall also frequently use the following notation:
\begin{equation}\label{def:index-sets-2}
\bm{\nu} \coloneqq \min \{ \bm{\alpha}, \bm{\beta} \},
\quad
\mathcal{J} \coloneqq \left\{ i \in \{ 1, \ldots, n \} \colon \alpha_i = \beta_i \right\},
\quad
\mathcal{K} \coloneqq \left\{ i \in \{ 1, \ldots, n \} \colon \alpha_i > \beta_i \right\}.
\end{equation}
Note that
\begin{equation*}
\{ i \colon \nu_i = \alpha_i \} = \mathcal{I} \cup \mathcal{J}
\quad \text{and} \quad
\{ i \colon \nu_i = \beta_i \} = \mathcal{J} \cup \mathcal{K}.
\end{equation*}

\begin{remark}
It follows from~\ref{A2.1} that \(A_{3, i} = A_{1, i}^\top\),
\(A_{4, i} = A_{2, i}^\top\) and \(A_{6, i, j}^\top = A_{6, j, i}\).
Moreover, the terms with \(t_i^{\beta_i'} \, t_j^{\beta_j'}\) such that
\(i \not\in \mathcal{F}\), \(j \not\in \mathcal{F}\) or both can be subsumed into the error term.
Hence, the assumption~\ref{A2.1} may be rewritten as follows:
\begin{multline}\label{Sigma-minus-R}
\Sigma - R ( \bm{t}, \bm{s} )
=
\sum_{i = 1}^n \Big[
A_{1, i} \, t_i^{\beta_i'}
+A_{2, i} \, t_i^{\beta_i}
+A_{1, i}^\top \, s_i^{\beta_i'}
+A_{2, i}^\top \, s_i^{\beta_i}
+S_{\alpha_i, A_{5, i}} ( t_i - s_i )
\Big]
\\[7pt]
+\sum_{i, \, j \in \mathcal{F}} A_{6, i, j} \, t_i^{\beta_i / 2} \, s_j^{\beta_j / 2}
+o \left(
\sum_{i = 1}^n \Big[ t_i^{\beta_i} + s_i^{\beta_i} + | t_i - s_i |^{\alpha_i} \Big]
\right).
\end{multline}
\end{remark}

\begin{remark} \label{important-remark}
It may be instructive to compare our assumption~\ref{A2} to that
of~\cite{DebickiEtAl2019}. To this end, consider the case
\(A_{1, i} = A_{6, i, j} = 0\). The assumptions~\ref{A2.3}
and~\ref{A2.6} are fulfilled automatically and assumption~\ref{A2.1}
reads:
\begin{equation}\label{simple-case}\tag{A2.1*}
\Sigma - R ( \bm{t}, \bm{s} )
=
\sum_{i = 1}^n \Big[
A_{2, i} \, t_i^{\beta_i}
+A_{2, i}^\top \, s_i^{\beta_i}
+S_{\alpha_i, A_{5, i}} ( t_i - s_i )
\Big]
+\epsilon,
\end{equation}
with the same error order as in~\ref{A2.2}.
Assumption~\ref{simple-case} combines (2.10), (2.11) and (2.13) of the
aforementioned paper into one and extends it from processes to fields of simple
(additive) covariance structure. It also has an advantage over them: it does not
require finding \(A ( \bm{t} )\) such that
\(\Sigma ( \bm{t} ) = A ( \bm{t} ) \, A^\top ( \bm{t} )\),
which can always be done in theory, but hard to implement in practice.
Next, assumption~\ref{A2.4} extends (2.12) to the case of fields
and~\ref{A2.5} does the same to the condition \(\bm{w}^\top V \bm{w} > 0\)
of Theorem 2.4.
\smallskip

Another important difference consists in allowing the leading order of
\(\Sigma - R ( \bm{t}, \bm{s} )\) to nullify \(\bm{w}\), that is, to fail
the condition \(\bm{w}^\top A_{1, i} \, \bm{w} > 0\) and even
\(A_{1, i} \, \bm{w} \neq \bm{0}\)\footnote{Note that
\(A = 0 \implies A \, \bm{w} = \bm{0} \implies \bm{w}^\top A \, \bm{w} = 0\),
but neither implication is reversible.}. In this case we require a certain speed
of convergence of the quadratic programming problem solutions
\(\bm{b} ( \bm{t} )\) to \(\bm{b} ( \bm{0} )\), see~\ref{A2.3}.
\end{remark}

\begin{remark}
Assumption~\ref{A2.6} is somewhat mysterious, which is why we present a
few intermediary results without assuming it in the Appendix. By
Lemma~\ref{lemma:quadratic-covariance} it is equivalent to the following:
there exists a family \(( C_{i, k} )_{i, k \in \mathcal{F}}\) such that
\begin{equation}\label{A2.6-D-as-sum}
D_{i, j} = \sum_{k \in \mathcal{F}} C_{i, k} \, C_{k, j}^\top.
\end{equation}
Note that if \(A_{6, i, j} = 0\), then
\(A_{1, i} \Sigma^{-1} A_{1, j}^\top = D_i \, D_j^\top\)
with \(D_i = A_{1, i} \Sigma^{-1/2}\), and therefore the assumption is satisfied.
An example of this situation may be found in our fBm double crossing example.
Another useful example is when \(A_{6, i, j}\) it not zero, but can itself be
represented as \(C_i \, C_j^\top\) for some family \(C_i\).
\(A_{6, i, j} = C_i \, C_j^\top\), then the assumption is also satisfied.
\end{remark}
\subsection{Constants}
\label{sec:orgfaa26c6}
For \(\alpha \in (0, 2]\) and a matrix \(V\), satisfying standard assumptions, let
\(\bm{Y}_{\alpha, V} ( t )\), \(\bm{t} \in \mathbb{R}\) be a multivariate fBm with cmf
\begin{equation}\label{def:R-and-S}
R_{\alpha, V} ( t, s )
\coloneqq
S_{\alpha, V} ( t ) + S_{\alpha, V} ( -s ) - S_{\alpha, V} ( t - s ),
\end{equation}
where \(S_{\alpha, V}\) is defined by~\eqref{def:S}.

For a triplet of disjoint sets
\(\mathcal{I}, \ \mathcal{J}, \ \mathcal{K} \subset \{ 1, \ldots, n \}\),
vector
\(\bm{\nu} \in (0, 2]^{\, \mathcal{I} \, \cup \, \mathcal{J} \, \cup \mathcal{K}}\)
and two collections of matrices
\(\mathbb{V} = \left( V_i \right)_{i \in \mathcal{I} \cup \mathcal{J}}\)
and
\(\mathbb{W} = \left( W_i \right)_{i \in \mathcal{J} \cup \mathcal{K}}\)
define a multivariate additive fBm field
\(\bm{Y}_{\bm{\nu}, \mathbb{V}} ( \bm{t} )\), \(\bm{t} \in \mathbb{R}^n\)
and a non-random vector field
\(\bm{d}_{\bm{\nu}, \mathbb{W}} ( \bm{t} )\), \(\bm{t} \in \mathbb{R}^n\)
by
\[
\bm{Y}_{\bm{\nu}, \mathbb{V}} ( \bm{t} )
\coloneqq
\sum_{i \in \mathcal{I} \cup \mathcal{J}}
\Big[ \bm{Y}_{\nu_i, V_i} ( t_i ) - S_{\nu_i, V_i} ( t_i ) \, \bm{1} \Big],
\qquad
\bm{d}_{\bm{\nu}, \mathbb{W}} ( \bm{t} )
\coloneqq
\sum_{i \in \mathcal{J} \cup \mathcal{K}} | t_i |^{\nu_i} \, W_i \, \bm{1}.
\]
Consider also a family of matrices \(\mathbb{D} = ( C_{i, j} )_{i, j \in \mathcal{D}}\),
\(\mathcal{D} \subset \{ 1, \ldots, n \}\), and set
\begin{equation*}
C_k ( \bm{t} ) \coloneqq \sum_{i \in \mathcal{D}} C_{i, k} \, t_i^{\beta_i/2},
\quad
\bm{Z} ( \bm{t} ) \coloneqq \sum_{k \in \mathcal{D}} C_k ( \bm{t} ) \, \bm{\mathcal{N}}_k,
\end{equation*}
where \(\bm{\mathcal{N}}_k \sim N ( \bm{0}, I )\) are standard Gaussian vectors,
independent of each other and of the fields \(\bm{Y}_{\nu_i, V_i}\), \(i \in \mathcal{I} \cup \mathcal{J}\).

Finally, for a compact set \(E \subset \mathbb{R}^n\) define
\begin{equation*}
  H_{\bm{\nu}, \mathbb{V}, \mathbb{W}, \mathbb{D}} \left( E \right)
  \coloneqq
  \int_{\mathbb{R}^d} e^{\bm{1}^\top \bm{x}} \,
  \mathbb{P} \left\{
    \exists \, \bm{t} \in E \colon
    \bm{Y}_{\bm{\nu}, \mathbb{V}} ( \bm{t} )
    +\bm{Z} ( \bm{t} )
    -\bm{d}_{\bm{\nu}, \mathbb{W}} ( \bm{t} ) > \bm{x}
  \right\}
  \mathop{d \bm{x}}
\end{equation*}
and
\begin{equation*}
\mathcal{H}_{\bm{\nu}, \mathbb{V}, \mathbb{W}, \mathbb{D}}
\coloneqq
\lim_{\Lambda \to \infty}
\lim_{S \to \infty}
S^{-| \mathcal{I} |}
H_{\bm{\nu}, \mathbb{V}, \mathbb{W}, \mathbb{D}}
\left( \left[ \bm{0}, \bm{S}' \right] \right),
\qquad
\bm{S}' = S \bm{1}_{\mathcal{I}} + \Lambda \bm{1}_{\mathcal{I}^c}
\end{equation*}
whenever the limit exists. Define also
\begin{equation*}
H_{\bm{\nu}, \mathbb{V}, \mathbb{W}} ( E ) \coloneqq
H_{\bm{\nu}, \mathbb{V}, \mathbb{W}, \varnothing} ( E ),
\end{equation*}
and similarly for \(\mathcal{H}_{\bm{\nu}, \mathbb{V}, \mathbb{W}}\).

As we shall see, for \(\mathcal{I}, \ \mathcal{J}, \ \mathcal{K}, \, \mathcal{F}\) from~\eqref{def:index-sets-1}
and~\eqref{def:index-sets-2} and matrices from~\ref{A2.1} this limit
exists and is positive and finite, provided that~\ref{A2.4}
and~\ref{A2.5} are satisfied, see Theorem~\ref{main-theorem}.

For an \(n \times n\) matrix \(\Xi\) and vector \(\bm{\beta} > \bm{0}\), define
\begin{equation}\label{def:G}
G ( \bm{\beta}, \Xi )
\coloneqq
\int_{\mathbb{R}_+^n}
\exp \left(
-\frac{1}{2}
\sum_{i, j} \Xi_{i, j} \, t_i^{\beta_i/2} \, t_j^{\beta_j/2}
\right)
\mathop{d \bm{t}}.
\end{equation}
\subsection{Main theorem}
\label{sec:org26870a8}
\begin{theorem} \label{main-theorem}
Let \(\bm{X} ( \bm{t} )\), \(\bm{t} \in [ \bm{0}, \bm{T} ] \subset \mathbb{R}^n\) be a
centered \(\mathbb{R}^d\)-valued Gaussian random field, satisfying
Assumptions~\ref{A1}, \ref{A2} and~\ref{A3}. Then
\begin{equation} \label{main-asymptotics}
  \mathbb{P} \left\{ \exists \, \bm{t} \in [ \bm{0}, \bm{T} ]
  \colon \bm{X} ( \bm{t} ) > u \bm{b}
  \right\}
  \sim
  u^{\zeta} \,
  \mathcal{H}_{\bm{\nu}, \mathbb{V}_{\bm{w}}, \mathbb{W}_{\bm{w}}, \mathbb{D}_{\bm{w}}} \,
  G ( \bm{\beta}_{\mathcal{I}}, \Xi_{\mathcal{I}, \mathcal{I}} ) \,
  \mathbb{P} \left\{ \bm{X} ( \bm{0} ) > u \bm{b} \right\},
\end{equation}
with
\begin{gather*}
\bm{\nu} = \min \{ \bm{\alpha}, \bm{\beta} \},
\quad
\zeta = \sum_{i \in \mathcal{I}} \left( \frac{2}{\alpha_i} - \frac{2}{\beta_i} \right),
\\[7pt]
\begin{aligned}
\mathbb{V}_{\bm{w}}
& = \Big( \diag ( \bm{w} ) \, A_{5, i} \, \diag ( \bm{w} ) \Big)_{i \in \mathcal{I} \cup \mathcal{J}},
& \mathbb{W}_{\bm{w}}
& = \Big( \diag ( \bm{w} ) \, A_{2, i} \, \diag ( \bm{w} ) \Big)_{i \in \mathcal{J} \cup \mathcal{K}},
\\[7pt]
\mathbb{D}_{\bm{w}}
& = \Big( \diag ( \bm{w} ) \, C_{i, k} \Big)_{i, k \in ( \mathcal{J} \cup \mathcal{K} ) \cap \mathcal{F}},
& \Xi_{i, j} & = \bm{w}^\top \Big[
2 \, A_{2, i} \, \mathbb{1}_{i = j} + D_{i, j} \, \mathbb{1}_{i, j \in \mathcal{F}}
\Big] \bm{w},
\end{aligned}
\end{gather*}
\(G\) defined by~\eqref{def:G},
\(( C_{i, k} )_{i, k \in ( \mathcal{J} \cup \mathcal{K} ) \cap \mathcal{F}}\) any family of matrices
satisfying~\eqref{A2.6-D-as-sum}, and
\begin{equation*}
G ( \bm{\beta}_{\mathcal{I}}, \Xi_{\mathcal{I}, \mathcal{I}} ), \quad
\mathcal{H}_{\bm{\nu}, \mathbb{V}_{\bm{w}}, \mathbb{W}_{\bm{w}}, \mathbb{D}_{\bm{w}}}
\in ( 0, \infty ).
\end{equation*}
\end{theorem}

\begin{remark}\label{rem:gen-var-XI}
By Lemma~\eqref{lemma:gen-var-exp}, the matrix \(\Xi\) may be
alternatively defined by
\begin{equation*}
\sigma_{\bm{b}}^{-2} ( \bm{\tau} ) - \sigma_{\bm{b}}^{-2} ( \bm{0} )
=
\sum_{i, j = 1}^n \Xi_{i, j} \, \tau_i^{\beta_i/2} \, \tau_j^{\beta_j/2}
+o \left( \sum_{i = 1}^n \tau_i^{\beta_i} \right).
\end{equation*}
This is useful from the practical point of view, since to apply the theorem we
first have to compute \(\sigma_{\bm{b}}^{-2} ( \bm{\tau} )\).
\end{remark}

\begin{corollary} \label{main-corollary}
If the conditions of Theorem~\ref{main-theorem} are satisfied with
\(\mathcal{F} = \varnothing\) or \(D_{i, j} = 0\) for all \(i, j \in \mathcal{F}\), then
\begin{equation*}
\mathbb{P} \left\{
\exists \, \bm{t} \in [\bm{0}, \bm{T}] \colon
\bm{X} ( \bm{t} ) > u \bm{b}
\right\}
\sim
u^{\zeta} \, \mathcal{H}_{\bm{\nu}, \mathbb{V}_{\bm{w}}, \mathbb{W}_{\bm{w}}} \,
\prod_{i \in \mathcal{I}} ( 2 \, \xi_i )^{-1/\beta_i} \, \Gamma \left( \frac{1}{\beta_i} + 1 \right) \,
\mathbb{P} \{ \bm{X} ( \bm{0} ) > u \bm{b} \}.
\end{equation*}
If \(\mathcal{F} \subset \mathcal{I}\), but not necessarily \(\varnothing\), we have that
\(\mathcal{D} = \varnothing\), and therefore~\eqref{main-asymptotics} holds with
\(\mathcal{H}_{\bm{\nu}, \mathbb{V}_{\bm{w}}, \mathbb{W}_{\bm{w}}}\).
\end{corollary}

\begin{remark}
If we consider the same problem on \([ -\bm{T}_1, \bm{T}_2 ]\), where
\(\bm{T}_{1, 2} \geq \bm{0}\) satisfy \(T_{1, i} + T_{2, i} > 0\) for all \(i\)
(so that the rectangle \([ -\bm{T}_1, \bm{T}_2 ]\) be of full
dimension), we can obtain a result similar to Theorem~\ref{main-theorem}
under slightly modified assumptions. Denote
\[
\mathcal{L} \coloneqq \{ i \colon T_{1, i} > 0 \},
\qquad
\mathcal{R} \coloneqq \{ i \colon T_{2, i} > 0 \},
\]
and consider the following symmetric extension of~\ref{A2.1} to the
negative values of \(t_i\)'s:
\begin{equation*}
\Sigma - R ( \bm{t}, \bm{s} )
\sim
\sum_{i = 1}^n \Big[
A_{1, i} \, |t_i|^{\beta_i'}
+A_{2, i} \, |t_i|^{\beta_i}
+A_{3, i} \, |s_i|^{\beta_i'}
+A_{4, i} \, |s_i|^{\beta_i}
+S_{\alpha_i, A_{5, i}} ( t_i - s_i )
\Big]
+\sum_{i, \, j = 1}^n A_{6, i, j} \, |t_i|^{\beta_i'} \, |s_j|^{\beta_j'}.
\end{equation*}
With these assumptions, we have
\begin{equation*}
  \mathbb{P} \left\{
  \exists \, \bm{t} \in [ -\bm{T}_1, \bm{T}_2 ] \colon
  \bm{X} ( \bm{t} ) > u \bm{b}
  \right\}
  \sim
  u^{\zeta} \,
  \mathcal{H}_{\bm{\nu}, \mathbb{V}_{\bm{w}}, \mathbb{W}_{\bm{w}}, \mathbb{D}_{\bm{w}}} ( \mathcal{L}, \mathcal{R} ) \,
  G ( \bm{\beta}_{\mathcal{I}}, \Xi_{\mathcal{I}, \mathcal{I}} ) \,
  \mathbb{P} \left\{ \bm{X} ( \bm{0} ) > u \bm{b} \right\},
\end{equation*}
where
\begin{equation*}
\mathcal{H}_{\bm{\nu}, \mathbb{V}, \mathbb{W}, \mathbb{D}} ( \mathcal{L}, \mathcal{R} )
\coloneqq
\lim_{\Lambda \to \infty}
\lim_{S \to \infty}
S^{-| \mathcal{I} |}
H_{\bm{\nu}, \mathbb{V}, \mathbb{W}, \mathbb{D}}
\left(
\bm{S}' \, [ -\bm{1}_{\mathcal{L}}, \bm{1}_{\mathcal{R}} ]
\right)
\in ( 0, \infty ),
\quad
\bm{S}' \coloneqq S \bm{1}_{\mathcal{I}} + \Lambda \bm{1}_{\mathcal{I}^c}.
\end{equation*}
\end{remark}
\section{Double crossing probabilities: general facts}
\label{sec:org4c63d61}
\label{SEC:double_crossing}

Let \(X ( t )\), \(t \in [ 0, T ] \subset \mathbb{R}\) be a continuous centered Gaussian process,
with covariance function \(r(t, s) = \mathbb{E} \left\{ X ( t ) X ( s ) \right\}\).

We want to study the probability that in a given finite time interval the process
\(X\) hits two distant barriers: one above and one below its initial point.
Formally, we study the asymptotics of
\begin{equation*}
  \mathbb{P} \left\{
    \exists \, t, \, s \in [ 0, T ] \colon X ( t ) > a u, \ X ( s ) < -bu
  \right\},
  \qquad a, b > 0,
\end{equation*}
which we shall refer to as the double crossing probability. For our purposes, it
may be conveniently rewritten as
\begin{equation*}
  \mathbb{P} \left\{ \exists \, \bm{t} \in [ 0, T ]^2 \colon \bm{X} ( \bm{t} ) > u \bm{b} \right\},
  \quad
  \bm{X} ( \bm{t} ) = \left( X ( t_1 ), \ -X ( t_2 ) \right)^\top,
  \quad
  \bm{b} = ( a, b )^\top.
\end{equation*}
The problem is thus reduced to the study of a simple (not double) crossing of a
two-dimensional vector-field \(\bm{X} ( \bm{t} )\) over \([ 0, T ]^2\),
which is exactly the setup of our Main Theorem~\ref{main-theorem}.

Let us briefly show how Main Theorem \ref{main-theorem} may be applied in this
case. First, we should find the maximizer of the generalized variance
\(\sigma_{a, b}^2 ( \bm{t} )\) defined by
\begin{equation} \label{def:double_crossing_generalized_sigma_general}
\sigma_{a, b}^{-2} ( \bm{t} )
= \min_{\bm{x} \geq ( a, b )^\top} \bm{x}^\top \, \Sigma^{-1} ( \bm{t} ) \, \bm{x},
\end{equation}
where the matrix \(\Sigma ( \bm{t} )\) is given by
\(\Sigma ( \bm{t} ) = R ( \bm{t}, \bm{t} )\) and
\begin{equation*}
R ( \bm{t}, \bm{s} )
= \mathbb{E} \left\{ \bm{X} ( \bm{t} ) \, \bm{X} ( \bm{s} )^\top \right\}
= \begin{pmatrix}
r ( t_1, s_1 ) & -r ( t_1, s_2 ) \\
-r ( t_2, s_1 ) & r ( t_2, s_2 )
\end{pmatrix},
\qquad \bm{t} = ( t_1, t_2 )^\top, \quad \bm{s} = ( s_1, s_2 )^\top.
\end{equation*}
The inverse of \(\Sigma ( \bm{t} )\), necessary to compute the generalized variance, exists
for all \(\bm{t} \not\in \{ \bm{t} = ( t, t )^\top \colon t \in [ 0, T ] \}\) and is given by
\begin{equation*}
\Sigma^{-1} ( \bm{t} )
=
\frac{1}{r ( t_1, t_1 ) \, r ( t_2, t_2 ) - r^2 ( t_1, t_2 )}
\begin{pmatrix}
r ( t_2, t_2 ) & r ( t_1, t_2 ) \\
r ( t_1, t_2 ) & r ( t_1, t_1 )
\end{pmatrix}.
\end{equation*}
Near the diagonal \(\{ \bm{t} = ( t, t )^\top \colon t \in [ 0, T ] \}\) the matrix \(\Sigma ( \bm{t} )\)
is degenerate, so we need some additional lemma to deal with the probability that
the extreme event happens there. Intuitively, such event means that a process
has managed to hit both boundaries during a very short time interval, which seems
unlikely. The precise meaning to this is given by the next lemma.

\begin{lemma} \label{double-crossing-diagonal}
Let \(X ( t )\) be a separable centered Gaussian process with a.s. continuous
sample paths. For \(\varepsilon > 0\), define \(D_{\varepsilon} = \{ \bm{t} \in [ 0, T ]^2 \colon | t_1 - t_2 | < \varepsilon \}\).
If there exists a function \(f\) such that
\[
\mathbb{E} \left\{ \Big[ X ( t + l ) - X ( t ) \Big]^2 \right\}
<
f ( l )
\]
for all \(t \in [ 0, T ]\), \(l > 0\) and \(f ( l ) \to 0\) as \(l \to 0\), then
for any \(\delta > 0\) there exists \(\varepsilon > 0\) such that
\[
\mathbb{P} \left\{
\exists \, \bm{t} \in D_{\varepsilon} \colon
X ( t_1 ) > au, \ X ( t_2 ) < -b u
\right\}
= o \left( e^{-\delta u^2} \right).
\]
\end{lemma}

Unfortunately, even the problem of minimizing \(\sigma_{\bm{b}}^{-2} ( \bm{t} )\)
over \(\{ \bm{t} = ( t , s )^\top \colon | t - s | > \varepsilon \}\) is too
hard in its full generality. In the next two sections we study the double crossing
probabilities for two classes of processes: stationary with positive correlation
and fractional Brownian motion.
\section{Double crossing probabilities: stationary case}
\label{sec:org8932c7f}
\label{SEC:double_crossing_stationary}

Let \(X ( t )\) be a stationary process with unit variance and positive
correlation function \(r ( t, s ) = \rho ( | t - s | ) > 0\) satisfying
\begin{equation} \label{rho-stationary}
\rho ( t ) = 1 - \vartheta \, t^{\alpha} + o \left( t^{\alpha} \right)
\quad \text{as} \quad
t \to 0
\end{equation}
with some \(\vartheta > 0\) and \(\alpha \in ( 0, 2 ]\). We additionally assume
that \(\rho\) is strictly decreasing and differentiable in \(t > 0\).
\bigskip

We need to minimize
\begin{equation*}
\frac{
x_1^2 + 2 x_1 x_2 \, \rho ( | t_1 - t_2 | ) + x_2^2
}{
1 - \rho^2 ( t_1, t_2 )
}
\end{equation*}
with respect to \(\bm{x} = ( x_1, x_2 )^\top\) subject to \(\bm{x} \geq ( a, b )^\top\),
and then minimize it again, but with respect to \(\bm{t}\) sufficiently far
away from the diagonal. The unique solution of the first minimization problem in
this case can easily be shown to be \(\bm{x} = ( a, b )^\top\). Therefore, we have
\begin{equation*}
\sigma_{a, b}^{-2} ( \bm{t} )
=
\min_{\bm{x} \geq ( a, b )^\top}
\bm{x}^\top \, \Sigma^{-1} ( \bm{t} ) \, \bm{x}
= \frac{
a^2 + 2 a b \, \rho ( | t_1 - t_2 | ) + b^2
}{
1 - \rho^2 ( | t_1 - t_2 | )
}.
\end{equation*}
To solve the second, we note by rewriting \(\sigma_{a, b}^{-2} ( \bm{t} )\) as
\[
\sigma_{a, b}^{-2} ( \bm{t} )
= \frac{( a + b )^2}{1 - \rho^2 ( | t_1 - t_2 | )}
-\frac{2 a b}{1 + \rho ( | t_1 - t_2 | )}
\]
that \(\sigma_{a, b}^{-2} ( \bm{t} )\) attains its minimum at the same point
as \(\rho ( | t_1 - t_2 | )\). Since \(\rho\) is decreasing, we have
two minimizing points \(\bm{t}_{ * , 1 } = ( 0, T )^\top\) and
\(\bm{t}_{ * , 2 } = ( T, 0 )^\top\). We can easily find the values of the
generalized variance at these points to be equal:
\[
\sigma_{a, b}^{-2} ( \bm{t}_{ * , 1 } )
= \sigma_{a, b}^{-2} ( \bm{t}_{ * , 2 } )
= \frac{a^2 + 2 a b \rho ( T ) + b^2}{1 - \rho^2 ( T )}.
\]

By Lemma~\ref{double-crossing-diagonal} and Piterbarg inequality, we can
easily show that for any \(\varepsilon > 0\) we have
\begin{multline*}
\mathbb{P} \left\{
\exists \, \bm{t} \in [ \bm{0}, \bm{T} ] \colon
\bm{X} ( \bm{t} ) > u \, \bm{b}
\right\}
\\[7pt]
\sim
\mathbb{P} \left\{
\exists \, \bm{t} \in [ 0, \varepsilon ] \times [ T - \varepsilon, T ] \colon
\bm{X} ( \bm{t} ) > u \, \bm{b}
\right\}
+\mathbb{P} \left\{
\exists \, \bm{t} \in [ T - \varepsilon, T ] \times [ 0, \varepsilon ] \colon
\bm{X} ( \bm{t} ) > u \, \bm{b}
\right\}
\\[7pt]
= \mathbb{P} \left\{
\exists \, \bm{t} \in [ 0, \varepsilon ]^2 \colon
\bm{X}_1 ( \bm{t} ) > u \, \bm{b}
\right\}
+\mathbb{P} \left\{
\exists \, \bm{t} \in [ 0, \varepsilon ]^2 \colon
\bm{X}_2 ( \bm{t} ) > u \, \bm{b}
\right\},
\end{multline*}
where in the last line we performed appropriate time changes in both
probabilities by introducing
\(\bm{X}_1 ( \bm{t} ) = ( X ( t_1 ), -X ( T - t_2 ) )^\top\)
and
\(\bm{X}_2 ( \bm{t} ) = ( X ( T - t_1 ), -X ( t_2 ) )^\top\).
\bigskip

In order to apply Main Theorem~\ref{main-theorem} to
\begin{equation*}
\mathbb{P} \left\{
\exists \, \bm{t} \in [ 0, \varepsilon ]^2 \colon
\bm{X}_1 ( \bm{t} ) > u \, \bm{b}
\right\}
\quad \text{and} \quad
\mathbb{P} \left\{
\exists \, \bm{t} \in [ 0, \varepsilon ]^2 \colon
\bm{X}_2 ( \bm{t} ) > u \, \bm{b}
\right\},
\end{equation*}
we need to derive asymptotic expansions of the corresponding covariances.
This is done in the following lemma.
\begin{lemma} \label{double-crossing-expansions}
The random field \(\bm{X}_1 ( \bm{t} )\) satisfies the
assumptions~\ref{A1} to~\ref{A3} of Theorem~\ref{main-theorem}
with
\begin{equation*}
\alpha_1 = \alpha_2 = \alpha, \quad \beta_1 = \beta_2 = 1, \quad \mathcal{F} = \varnothing
\end{equation*}
and
\begin{equation*}
A_{2, 1} = A_{2, 2}^\top =
-\rho' ( T )
\begin{pmatrix}
0 & 1 \\
0 & 0
\end{pmatrix},
\quad
A_{5, 1} =
\vartheta
\begin{pmatrix}
1 & 0 \\
0 & 0
\end{pmatrix},
\quad
A_{5, 2} =
\vartheta
\begin{pmatrix}
0 & 0 \\
0 & 1
\end{pmatrix}.
\end{equation*}
Moreover,
\begin{equation*}
\bm{w}
= \Sigma^{-1} ( \bm{0} ) \, ( a, b )^\top
= \frac{1}{1 - \rho^2 ( T )}
\begin{pmatrix}
1 & \rho ( T ) \\
\rho ( T ) & 1
\end{pmatrix}
\begin{pmatrix}
a \\ b
\end{pmatrix}
= \frac{1}{1 - \rho^2 ( T )}
\begin{pmatrix}
a + b \rho ( T ) \\
b + a \rho ( T )
\end{pmatrix}
\end{equation*}
Assumption~\ref{A2.4} is satisfied with \(\xi_i\) given by
\begin{equation*}
\xi_1 = \bm{w}^\top \, A_{2, 1} \, \bm{w}
= \xi_2 = \bm{w}^\top \, A_{2, 2} \, \bm{w}
= \frac{
-\rho' ( T ) ( a + b \rho ( T ) )
( b + a \rho ( T ) )
}{( 1 - \rho^2 ( T ) )^2}
> 0,
\end{equation*}
and~\ref{A2.5} with \(\varkappa_i\) given by
\begin{equation*}
\varkappa_1 = \bm{w}^\top \, A_{5, 1} \, \bm{w}
= \varkappa_2 ( \bm{w} ) = \bm{w}^\top \, A_{5, 2} \, \bm{w}
= \frac{C ( b + a \rho ( T ) )^2}{( 1 - \rho^2 ( T ) )^2}
> 0.
\end{equation*}
\end{lemma}

Using Lemma~\ref{double-crossing-diagonal}, and noting that
\(\mathcal{F} = \varnothing\), we may apply the first assertion of
Corollary~\ref{main-corollary} instead of Theorem~\ref{main-theorem},
and obtain the following theorem on the asymptotics of double crossing
probabilities.
\begin{theorem} \label{double-crossing-theorem}
Let \(X ( t )\), \(t \in [0, T]\) be a centered a.s. continuous stationary Gaussian process with
unit variance and positive strictly decreasing and differentiable correlation
function \(\rho ( t ) > 0\) which satisfies~\eqref{rho-stationary}
with some \(\vartheta > 0\) and \(\alpha \in ( 0, 2 ]\). Define
\begin{equation*}
p ( u ) \coloneqq \mathbb{P} \left\{ X ( 0 ) > a u, \ X ( T ) < -bu \right\}.
\end{equation*}
Then
\begin{description}[leftmargin = * ]
\item [Pickands case] If \(\alpha < 1\),
\begin{equation*}
\mathbb{P} \left\{
\exists \, t, \, s \in [ 0, T ] \colon
X ( t ) > a u, \
X ( s ) < -b u
\right\}
\sim
C \,
\mathcal{H}^2 \,
u^{4 / \alpha - 2} \, p ( u ),
\end{equation*}
where
\begin{equation*}
\mathcal{H}
= \lim_{S \to \infty} \frac{1}{S} \,
\mathbb{E} \left\{
\sup_{t \in [ 0, S ]}
e^{B_H ( t ) - t^{2H} / 2}
\right\},
\quad
C =
\frac{2^{1 + 2/\alpha} \, \theta^{2/\alpha}}{( \rho' ( T ) )^2}
\left(
\frac{
( 1 - \rho^2 ( T ) )^2
}{
( a + b \rho ( T ) ) ( b + a \rho ( T ) )
}
\right)^{2 - 2/\alpha}.
\end{equation*}

\item [Piterbarg case] If \(\alpha = 1\),
\begin{equation*}
\mathbb{P} \left\{
\exists \, t, s \in [ 0, T ] \colon
X ( t ) > a u, \
X ( s ) < -b u
\right\}
\sim
2 \ \widetilde{\mathcal{H}}_1 \, \widetilde{\mathcal{H}}_2 \,
p ( u ),
\end{equation*}
where
\begin{equation*}
\widetilde{\mathcal{H}}_k
= \lim_{S \to \infty}
\mathbb{E} \left\{
  \sup_{t \in [ 0, S ]}
  e^{B ( t ) - ( 1 + \lambda_k ) t / 2}
\right\},
\qquad
\lambda_1 = \frac{-\rho' ( T ) \, w_2}{2 \, w_1 \, \vartheta}, \quad
\lambda_2 = \frac{-\rho' ( T ) \, w_1}{2 \, w_2 \, \vartheta}
> 0.
\end{equation*}

\item [Talagrand case] If \(\alpha > 1\),
\begin{equation*}
\mathbb{P} \left\{
\exists \, t, \, s \in [ 0, T ] \colon
X ( t ) > a u, \
X ( s ) < -b u
\right\}
\sim
2 \, p ( u ).
\end{equation*}
\end{description}
\end{theorem}
\section{Double crossing probabilities: fBm case}
\label{sec:orgf91a9db}
\label{SEC:double_crossing_fBm}

In this section we study the double crossing
probability~\eqref{def:double_crossing_probability} in case when
\(X ( t )\) is a fractional Brownian motion \(B_H ( t )\), that is, a Gaussian
process associated to the following covariance function:
\begin{equation*}
r ( t, s )
= \frac{1}{2} \Big(
t^{2H} + s^{2H} - | t - s |^{2H}
\Big).
\end{equation*}
As explained in Section~\ref{SEC:double_crossing}, we should first
minimize \(\sigma_{\bm{b}}^{-2} ( \bm{t} )\), defined
in~\eqref{def:double_crossing_generalized_sigma_general}, in \(\bm{x}\)
subject to \(\bm{x} \geq ( a, b )^\top\), and then minimize it again, but with
respect to \(\bm{t}\) sufficiently far away from the diagonal. Minimization in
\(\bm{x}\) yields \(\bm{x} = ( a, b )^\top\), and therefore we have
\begin{equation*}
\sigma_{a,b}^{-2} ( \bm{t} )
= \min_{\bm{x} \geq ( a, b )^\top}
\bm{x}^\top \, \Sigma^{-1} ( \bm{t} ) \, \bm{x}
= \frac{
a^2 \, t_1^{2H} + 2 a b \, r ( t_1, t_2 ) + b^2 \, t_2^{2H}
}{
(t_1 t_2 )^{2H} - r^2 ( t_1, t_2 )
}.
\end{equation*}
For the second minimization, we have the following lemma.

\begin{lemma} \label{lemma:double-crossing-minimizer-fBm}
There exists \(t_{ * } \in ( 0, T )\), such that the function
\(\sigma_{a, b}^{-2} ( \bm{t} )\) defined on
\(\{ \bm{t} \in [ 0, T ]^2 \colon | t_1 - t_2 | > \varepsilon \}\), \(\varepsilon > 0\)
attains,
\begin{description}[leftmargin = *, before={\renewcommand\makelabel[1]{\bfseries ##1,}}]
\item [if \(a < b\)] its unique minimum at point
\(\bm{t}_{ * , 1 } = ( T, t_{ * } )^\top\)
\item [if \(a > b\)] its unique minimum at point
\(\bm{t}_{ * , 2 } = ( t_{ * }, T )^\top\)
\item [if \(a = b\)] its minimum at exactly two points
\(\bm{t}_{ *, 1 } = ( T, t_{ * } )^\top\) and \(\bm{t}_{ *, 2 } = ( t_{ * }, T )^\top\)

Moreover, we have
\[
\sigma_{\bm{b}}^{-2} ( \bm{t}_{ * , 1 } )
-\sigma_{\bm{b}}^{-2} ( \bm{t}_{ * , 1 } - \bm{\tau} )
\sim
-\kappa_1 \, \tau_1
-\kappa_2 \, \tau_2^2
\quad \text{and} \quad
\sigma_{\bm{b}}^{-2} ( \bm{t}_{ * , 2 } )
-\sigma_{\bm{b}}^{-2} ( \bm{t}_{ * , 2 } - \bm{\tau} )
\sim
-\kappa_2 \, \tau_1^2
-\kappa_1 \, \tau_2
\]
with
\[
\kappa_1 \coloneqq
-\frac{\partial \sigma_{\bm{b}}^{-2}}{\partial t_1} ( \bm{t}_{ * , 1 } )
= -\frac{\partial \sigma_{\bm{b}}^{-2}}{\partial t_2} ( \bm{t}_{ * , 2 } )
> 0,
\qquad
\kappa_2 \coloneqq
\frac{\partial^2 \sigma_{\bm{b}}^{-2}}{\partial t_2^2} ( \bm{t}_{ * , 1 } )
= \frac{\partial^2 \sigma_{\bm{b}}^{-2}}{\partial t_1^2} ( \bm{t}_{ * , 2 } ) > 0.
\]
\end{description}
\end{lemma}

By Lemma~\ref{double-crossing-diagonal} and Piterbarg
inequality~\eqref{Piterbarg-inequality}, we can easily show that for any
\(\varepsilon > 0\) we have
\begin{equation}\label{fBm-eps-vicinity}
\begin{aligned}
&
\mathbb{P} \left\{
\exists \, \bm{t} \in [ \bm{0}, \bm{T} ] \colon
\bm{X} ( \bm{t} ) > u \, \bm{b}
\right\}
\\[7pt]
& \hspace{50pt} \sim
\mathbb{P} \left\{
\exists \, \bm{t} \in
[ T - \varepsilon, T ]
\times [ t_{ * } - \varepsilon, t_{ * } + \varepsilon ]
\colon
\bm{X} ( \bm{t} ) > u \, \bm{b}
\right\}
\\[7pt]
& \hspace{60pt} +
\mathbb{P} \left\{
\exists \, \bm{t} \in
[ t_{ * } - \varepsilon, t_{ * } + \varepsilon ]
\times [ T - \varepsilon, T ]
\colon
\bm{X} ( \bm{t} ) > u \, \bm{b}
\right\}
\\[7pt]
& \hspace{50pt}
= \mathbb{P} \left\{
\exists \, \bm{t} \in [ 0, \varepsilon ] \times [ -\varepsilon, \varepsilon ] \colon
\bm{X}_1 ( \bm{t} ) > u \, \bm{b}
\right\}
\\[7pt]
& \hspace{60pt}
+\mathbb{P} \left\{
\exists \, \bm{t} \in [ -\varepsilon, \varepsilon ] \times [ 0, \varepsilon ] \colon
\bm{X}_2 ( \bm{t} ) > u \, \bm{b}
\right\},
\end{aligned}
\end{equation}
where in the last line we performed appropriate time changes in both
probabilities by introducing
\(\bm{X}_1 ( \bm{t} )
=
( X ( T - t_1 ), -X ( t_2 - t_{ * } ) )^\top\)
and
\(\bm{X}_2 ( \bm{t} )
=
( X ( t_1 - t_{ * } ), -X ( T - t_2 ) )^\top\).

Next, assume that \(a > b\). Then we can use the Piterbarg
inequality~\eqref{Piterbarg-inequality} again to show that
\begin{equation}\label{fBm-a-bigger-than-b}
\mathbb{P} \left\{
\exists \, \bm{t} \in [ -\varepsilon, \varepsilon ] \times [ 0, \varepsilon ] \colon
\bm{X}_2 ( \bm{t} ) > u \, \bm{b}
\right\}
= o \big(
\mathbb{P} \left\{
\exists \, \bm{t} \in [ 0, \varepsilon ] \times [ -\varepsilon, \varepsilon ] \colon
\bm{X}_1 ( \bm{t} ) > u \, \bm{b}
\right\}
\big).
\end{equation}
Intuitively this means that the process is less likely to first hit a higher
barrier and then hit the lower than the other way around. If \(a = b\), the
two probabilities are equal, which is clear from the symmetry.
Combining~\eqref{fBm-eps-vicinity} and~\eqref{fBm-a-bigger-than-b},
we obtain that if \(a \geq b\), then
\begin{equation*}
\mathbb{P} \left\{
\exists \, \bm{t} \in [ \bm{0}, \bm{T} ] \colon
\bm{X} ( \bm{t} ) > u \bm{b}
\right\}
\sim
\left( 1 + \mathbb{1}_{a = b} \right) \,
\mathbb{P} \left\{
\exists \, \bm{t} \in [0, \varepsilon] \times [-\varepsilon, \varepsilon] \colon
\bm{X}_1 ( \bm{t} ) > u \bm{b}
\right\}.
\end{equation*}

In order to apply Theorem~\ref{main-theorem} to
\(\mathbb{P} \left\{
\exists \, \bm{t} \in [ 0, \varepsilon ]^2 \colon
\bm{X}_1 ( \bm{t} ) > u \, \bm{b}
\right\}\),
we need to derive asymptotic expansions of the corresponding covariance matrix
\(R ( \bm{t}, \bm{s} )\) as \(\bm{t}\) and \(\bm{s}\) tend to zero.
This is done in the following lemma.

\begin{lemma} \label{double-crossing-expansions-fBm}
The random field \(\bm{X}_1 ( \bm{t} )\) satisfies the assumptions~\ref{A1}
to~\ref{A3} of Theorem~\ref{main-theorem} with
\begin{equation*}
\alpha_1 = \alpha_2 = 2H, \quad
\beta_1 = 1, \quad
\beta_2 = 2, \quad
\beta_2' = 1, \quad
\mathcal{F} = \{ 2 \},
\end{equation*}
and
\begin{gather*}
\begin{aligned}
A_{2, 1}
& =
H
\begin{pmatrix}
-T^{2H-1} & T^{2H-1} - | T - t_{ * } |^{2H-1} \\
0 & 0
\end{pmatrix},
\\[7pt]
A_{1, 2}
& =
H
\begin{pmatrix}
0 & 0 \\
t_{ * }^{2H-1} + | T - t_{ * } |^{2H-1} & -t_{ * }^{2H-1}
\end{pmatrix},
\\[7pt]
A_{2, 2}
& =
H \left( H - \frac{1}{2} \right)
\begin{pmatrix}
0 & 0 \\
t_{ * }^{2H-2} + | T - t_{ * } |^{2H-2} & -t_{ * }^{2H-2} \\
\end{pmatrix},
\\[7pt]
A_{5, 1}
& =
\frac{1}{2}
\begin{pmatrix}
1 & 0 \\
0 & 0
\end{pmatrix},
\quad
A_{5, 2} =
\frac{1}{2}
\begin{pmatrix}
0 & 0 \\
0 & 1
\end{pmatrix},
\quad
A_{6, 2, 2} = 0.
\end{aligned}
\end{gather*}
Moreover,
\begin{equation*}
\bm{w} ( \bm{t} ) = \Sigma^{-1} ( \bm{t} ) \, \bm{b} =
\frac{1}{t_1^{2H} t_2^{2H} - r^2 ( t_1, t_2 )}
\begin{pmatrix}
t_2^{2H} \, a + r ( t_1, t_2 ) \, b
\\[7pt]
r ( t_1, t_2 ) \, a + t_1^{2H} \, b
\end{pmatrix},
\end{equation*}
and
\begin{equation*}
A_{1, 1} \, \bm{w} \neq \bm{0}, \quad
A_{1, 2} \, \bm{w} ( \bm{t} ) \sim \bm{0}, \quad
\bm{w}^\top A_{5, i} \, \bm{w} > 0, \quad
\bm{w}^\top A_{2, i} \, \bm{w} > 0, \quad i = 1, 2.
\end{equation*}
\end{lemma}

Note that \(2 \in \mathcal{I}\). Since \(2\) is also the only element of \(\mathcal{F}\), it
follows that \(\mathcal{F} \subset \mathcal{I}\), and we can use the second assertion of
Corollary~\ref{main-corollary} instead of
Theorem~\ref{main-theorem}. Applying it with the data from
Lemmata~\ref{lemma:double-crossing-minimizer-fBm}
and~\ref{double-crossing-expansions-fBm}, we obtain the following result.

\begin{theorem} \label{double-crossing-theorem-fBm}
Let \(a \geq b\). Set
\begin{equation*}
p ( u ) \coloneqq
\left( 1 + \mathbb{1}_{a = b} \right)
\mathbb{P} \left\{ B_H ( T ) > au, \ B_H ( t_{ * } ) < -bu \right\}.
\end{equation*}
Then, with \(\kappa_1\) and \(\kappa_2\) from
Lemma~\ref{lemma:double-crossing-minimizer-fBm}, we have the following results.
\begin{description}[leftmargin = * , before={\renewcommand\makelabel[1]{\bfseries ##1,}} ]
\item [If \(H < 1/2\)]
\begin{equation*}
\mathbb{P} \left\{
\exists \, t, \, s \in [ 0, T ] \colon
B_H ( t ) > a u, \
B_H ( s ) < -b u
\right\}
\sim
\frac{9 \, \pi \, w_1^{1/H} \, w_2^{1/H} \, \mathcal{H}^2}{2 \, \kappa_1^{1/2} \, \kappa_2^{1/2}} \,
\, u^{2/H - 3} \,
p ( u ),
\end{equation*}
where
\begin{equation*}
\mathcal{H}
= \lim_{S \to \infty}
\frac{1}{S} \,
\mathbb{E} \left\{
\sup_{t \in [ 0, S ]} e^{B_H ( t ) - t^{2H} / 2}
\right\}.
\end{equation*}

\item [If \(H = 1/2\)]
\begin{equation*}
\mathbb{P} \left\{
\exists \, t, \, s \in [ 0, T ] \colon
B_H ( t ) > a u, \
B_H ( s ) < -b u
\right\}
\sim
\frac{3 \sqrt{\pi} \, w_2^{1/H} \, \mathcal{H} \, \widetilde{\mathcal{H}}}{\kappa_2^{1/2}}
\, u \, p ( u )
\end{equation*}
where
\begin{equation*}
\widetilde{\mathcal{H}}
= \lim_{\Lambda \to \infty}
\mathbb{E} \left\{
\sup_{t \in [ 0, \Lambda ]}
e^{B ( t ) - \lambda t}
\right\}
\end{equation*}
with\footnote{It can be shown that \(\lambda > 1/2\). Moreover, it may be represented as
\(\lambda = 1/2 + \kappa_1 / w_1\).}
\begin{equation*}
\lambda =
\frac{1}{2}
+H T^{2H - 1} w_1 - H \Big[ T^{2H - 1} + ( T - t_{ * } )^{2H - 1} \Big] w_2
\end{equation*}

\item [If \(H > 1/2\)]
\begin{equation*}
\mathbb{P} \left\{
\exists \, t, \, s \in [ 0, T ] \colon
B_H ( t ) > a u, \
B_H ( s ) < -b u
\right\}
\sim
\frac{3 \sqrt{\pi} \, w_2^{1/H} \, \mathcal{H}}{\kappa_2^{1/2}} \,
u^{1/H - 1} \, p ( u ).
\end{equation*}
\end{description}
\end{theorem}

\begin{remark}
Note that \(\kappa_2 \neq 2 \, \bm{w}^\top A_{2, 2} \, \bm{w}\), as it were
in~\cite[Remark 2.3, (2.15)]{DebickiEtAl2019}. Using
Lemma~\ref{lemma:gen-var-exp}, we can show that
\(\kappa_2 = 2 \, \bm{w}^\top A_{2, 2} \, \bm{w} + \bm{w}^\top A_{1, 2} \, \Sigma^{-1} A_{1, 2}^\top \, \bm{w}\).
\end{remark}
\section{Auxiliary results}
\label{sec:orgdcf2870}
\label{SEC:auxiliary_results}
This Section consists of known results, taken from~\cite{DebickiEtAl2019}
and reproduced here for the reader's convenience.
\subsection{Quadratic programming problem}
\label{sec:orged02637}
For a given non-singular \(d \times d\) real matrix \(\Sigma\) we consider
the quadratic programming problem
\begin{equation} \label{QPP}
  \Pi_\Sigma ( \bm{b} ) \colon
  \text{minimize} \
  \bm{x}^\top \, \Sigma^{-1} \, \bm{x} \
  \text{under the linear constraint} \
  \bm{x} \geq \bm{b}.
\end{equation}
Below \(J = \{ 1, \ldots, d \} \setminus I\) can be empty; the claim in
\eqref{eq:bJ} is formulated under the assumption that \(J\) is non-empty.

\begin{lemma} \label{lemma:QPP}
Let \(d \geq 2\) and \(\Sigma\) a \(d \times d\) symmetric positive
definite matrix with inverse \(\Sigma^{-1}\). If \(\bm{b} \in \mathbb{R}^d
\setminus (-\infty, ]^d\), then \(\Pi_{\Sigma} ( \bm{b} )\) has a unique
solution \(\widetilde{\bm{b}}\) and there exists a unique non-empty index set
\(I \subset \{ 1, \ldots, d \}\) with \(m \leq d\) elements such that
\begin{align}
& \widetilde{\bm{b}}_I
= \bm{b}_I \neq \bm{0}_I
\\[7pt]
& \widetilde{\bm{b}}_J
= \Sigma_{JI} ( \Sigma_{II} )^{-1} \bm{b}_I \geq \bm{b}_J,
\qquad ( \Sigma_{II} )^{-1} \bm{b}_I > \bm{0}_I,
\label{eq:bJ}
\\[7pt]
& \min_{\bm{x} \, \geq \, \bm{b}}
\bm{x}^\top \, \Sigma^{-1} \, \bm{x}
= \widetilde{\bm{b}}^\top \, \Sigma^{-1} \, \widetilde{\bm{b}}
= \bm{b}_I^\top \, ( \Sigma_{II} )^{-1} \, \bm{b}_I
> 0,
\\
& \max_{\bm{z} \in [ 0, \infty )^d \colon \bm{z}^\top \bm{b} > 0}
\frac{( \bm{z}^\top \bm{b} )^2}{\bm{z}^\top \, \Sigma \, \bm{z}}
= \frac{( \bm{w}^\top \bm{b} )^2}{\bm{w}^\top \, \Sigma \, \bm{w}}
= \min_{\bm{x} \, \geq \, \bm{b}} \bm{x}^\top \, \Sigma^{-1} \, \bm{x},
\end{align}
with \(\bm{w} = \Sigma^{-1} \, \widetilde{\bm{b}}\) satisfying
\(\bm{w}_I = ( \Sigma_{II} )^{-1} \bm{b}_I > \bm{0}_I\),
\(\bm{w}_J = \bm{0}_J\).
\end{lemma}

Denote the solution map of the quadratic programming problem \eqref{QPP} by
\(\mathcal{P} \colon \Sigma^{-1} \mapsto \widetilde{\bm{b}}\)
with \(\widetilde{\bm{b}}\) the unique solution to \(\Pi_{\Sigma} ( \bm{b} )\).
The next result is a special case of~\cite[Theorem 3.1]{MR528899}.

\begin{lemma}
\(\mathcal{P}\) is Lipshitz continuous on compact subset of the space of real
\(d \times d\) symmetric positive definite matrices.
\end{lemma}

We will also need the following supplementary lemma on quadratic optimization.
\begin{lemma}\label{lemma:QPP-supp}
Let \(E\) be a compact subset of \(\mathbb{R}^n\) and
\(( \Sigma ( \bm{t} ) )_{\bm{t} \in E}\) be a uniformly positive definite family
of symmetric \(d \times d\) matrices such that the map \(\bm{t} \mapsto \Sigma ( \bm{t} )\) is
continuous.
Denote by \(I ( \bm{t} )\), \(K ( \bm{t} )\) and \(L ( \bm{t} )\)
the three index sets of \(\Pi_{\Sigma ( \bm{t} )} ( \bm{b} )\),
introduced in Lemma \ref{lemma:QPP}.
Then the following assertions hold:
\begin{enumerate}
\item There exist three finite disjoint locally closed covers \(( A_V )_{V \in 2^d}\),
\(( B_V )_{V \in 2^d}\) and \(( C_{U, V} )_{U, V \in 2^d}\) of \(E\),
such that
\begin{equation*}
I ( \bm{t} ) = \sum_{U \in 2^d} U^c \, \mathbb{1}_{A_U} ( \bm{t} ),
\quad
L ( \bm{t} ) = \sum_{V \in 2^d} V^c \, \mathbb{1}_{B_V} ( \bm{t} ),
\quad
K ( \bm{t} ) = \sum_{U, \, V \in 2^d} U \cap V \, \mathbb{1}_{C_{U, V}} ( \bm{t} ).
\end{equation*}
\item The maps \(\bm{t} \mapsto I ( \bm{t} )\) and \(\bm{t} \mapsto L ( \bm{t} )\) are lower hemicontinuous.
Moreover, for all \(\bm{t} \in E\) there exists
\(\varepsilon ( \bm{t} ) > 0\) such that for all \(\bm{s}\)
such that \(| \bm{t} - \bm{s} | < \varepsilon ( \bm{t} )\) holds \(I ( \bm{t} ) \subset I ( \bm{s} )\).
\end{enumerate}
\end{lemma}

\begin{remark}
Since \(I ( \bm{t} ) \cup K ( \bm{t} ) \cup L ( \bm{t} ) = \{ 1, \ldots, d \}\),
it follows that the set-valued map \(\bm{t} \mapsto K ( \bm{t} )\) is upper hemicontinuous.
\end{remark}

\begin{remark}
Note that if the upper hemicontinuity property holds uniformly in \(\bm{t}\),
that is if \(\varepsilon ( \bm{t} )\) can be taken independent on \(\bm{t}\),
then \(\bm{t} \mapsto J ( \bm{t} )\) is constant.
\end{remark}
\subsection{Borell-TIS and Piterbarg inequalities}
\label{sec:orgfa72e5e}
\begin{lemma}
Let \(\bm{Z} ( \bm{t} )\), \(\bm{t} \in E \subset \mathbb{R}^k\) be a separable centered
\(d\)-dimensional vector-valued Gaussian random field having components with
a.s. continuous paths. Assume that
\(\Sigma ( \bm{t} )
= \mathbb{E} \left\{
\bm{Z} ( \bm{t} ) \, \bm{Z} ( \bm{t} )^\top
\right\}\)
is non-singular for all \(\bm{t} \in E\).
Let \(\bm{b} \in \mathbb{R}^d \setminus (-\infty, 0]^d\) and define \(\sigma_{\bm{b}}^2 ( \bm{t} )\) as
in~\eqref{def:gen-variance}.
If \(\sigma_{\bm{b}}^2 = \sup_{t \in E} \sigma_{\bm{b}^2} ( \bm{t} ) \in ( 0, \infty )\),
then there exists some positive constant \(\mu\) such that for all \(u > \mu\)
\begin{equation}
\label{Borell-TIS}
\mathbb{P} \left\{
\exists \, \bm{t} \in E \colon
\bm{Z} ( \bm{t} ) > u \bm{b}
\right\}
\leq
\exp \left(
-\frac{( u - \mu )^2}{2 \sigma_{\bm{b}}^2}
\right).
\end{equation}
If further for some \(C \in ( 0, \infty )\) and \(\bm{\gamma} \in ( 0, 2 ]^k\)
\begin{equation}
\label{eq:29}
\sum_{1 \leq i \leq k}
\mathbb{E} \left\{
\left( Z_i ( \bm{t} ) - Z_i ( \bm{s} ) \right)^2
\right\}
\leq
C \sum_{m = 1}^k | t_m - s_m |^{\gamma_m}
\end{equation}
and
\begin{equation}
\label{eq:30}
\left\|
\Sigma^{-1} ( \bm{t} ) - \Sigma^{-1} ( \bm{s} )
\right\|_F
\leq
C \sum_{m = 1}^k | t_m - s_m |^{\gamma_m}
\end{equation}
hold for all \(\bm{t}, \, \bm{s} \in E\), then for all \(u\) positive
\begin{equation}
\label{Piterbarg-inequality}
\mathbb{P} \left\{
\exists \, \bm{t} \in E \colon
\bm{Z} ( \bm{t} ) > u \bm{b}
\right\}
\leq
C_{ * } \operatorname{mes} ( E ) \,
u^{2d/\gamma - 1}
\exp \left(
-\frac{u^2}{2 \sigma_{\bm{b}}^2}
\right),
\end{equation}
where \(C_{ * }\) is some positive constant not depending on \(u\). In
particular, if \(\sigma_{\bm{b}}^2 ( \bm{t} )\), \(\bm{t} \in E\) is
continuous and achieves its unique maximum at some fixed point \(\bm{t}_{ * }
\in E\), then \eqref{Piterbarg-inequality} is still valid if \eqref{eq:29} and
\eqref{eq:30} are assumed to hold only for all \(\bm{t}, \, \bm{s} \in E\) in an open neighborhood of \(\bm{t}_{ * }\).
\end{lemma}
\subsection{Local Pickands lemma}
\label{sec:orga166594}
Let us introduce the following assumptions.

\begin{description}[leftmargin = * ]
\item[\namedlabel{B1}{B1}] For all large \(u\) and all \(\bm{\tau} \in Q_u\), the
matrix \(\Sigma_{u, \bm{\tau}} = R_{u, \bm{\tau}} ( \bm{0}, \bm{0} )\) is positive
definite and
\begin{equation}
\label{eq:32}
\lim_{u \to \infty}
\sup_{\bm{\tau} \in Q_u}
u \left\| \Sigma - \Sigma_{u, \bm{\tau}} \right\|_F
= 0
\end{equation}
holds for some positive definite matrix \(\Sigma\).

\item[\namedlabel{B2}{B2}] There exists a continuous \(\mathbb{R}^d\)-valued function
\(\bm{d} ( \bm{t} )\), \(\bm{t} \in E\) and a continuous matrix-valued
function \(K ( \bm{t}, \bm{s} )\), \(( \bm{t}, \bm{s} ) \in E \times E\), such that
\begin{equation}
\label{eq:31}
\lim_{u \to \infty}
\sup_{\bm{\tau} \in Q_u, \ \bm{t} \in E}
u \, \left\| \Sigma_{u, \bm{\tau}} - R_{u, \bm{\tau}} ( \bm{t}, \bm{0} ) \right\|_F = 0,
\end{equation}
\begin{equation}
\label{eq:2}
\lim_{u \to \infty} \sup_{\bm{\tau} \in Q_u, \ \bm{t} \in \mathbb{E}}
\left| u^2 \left[
I - R_{u, \bm{\tau}} \, \Sigma_{u, \bm{\tau}}^{-1}
\right] \widetilde{\bm{b}}
-\bm{d} ( \bm{t} )
\right|
= 0
\end{equation}
and
\begin{equation}
\label{eq:33}
\lim_{u \to \infty}
\sup_{\bm{\tau} \in Q_u}
\sup_{\bm{t}, \, \bm{s} \in E}
\left\|
u^2 \, \Big[
R_{u, \bm{\tau}} ( \bm{t}, \bm{s} )
-R_{u, \bm{\tau}} ( \bm{t}, \bm{0} ) \,
\Sigma_{u, \bm{\tau}}^{-1} \,
R_{u, \bm{\tau}} ( \bm{0}, \bm{s} )
\Big]
-K ( \bm{t}, \bm{s} )
\right\|_F
= 0.
\end{equation}

\item[\namedlabel{B3}{B3}] There exist positive constants \(C\) and \(\bm{\gamma} \in ( 0, 2 ]^k\) such that for any
\(\bm{t}, \, \bm{s} \in E\)
\begin{equation}
\label{eq:34}
\sup_{\bm{\tau} \in Q_u}
u^2 \, \mathbb{E} \left\{
\left|
\bm{X}_{u, \bm{\tau}} ( \bm{t} )
-\bm{X}_{u, \bm{\tau}} ( \bm{s} )
\right|^2
\right\}
\leq
C \sum_{m = 1}^k | t_m - s_m |^{\gamma_m}.
\end{equation}
\end{description}

For \(\bm{Y} ( \bm{t} )\), \(\bm{t} \in E\) a centered \(\mathbb{R}^d\)-valued Gaussian random field with a.s. continuous sample paths with cmf \(K
( \bm{s}, \bm{t} )\), \(( \bm{s}, \bm{t} ) \in E \times E\) and an \(\mathbb{R}^d\)-valued function \(\bm{d}\) define below
\begin{equation}
\label{eq:35}
H_{\bm{Y}, \bm{d}} ( E )
= \int_{\mathbb{R}^d}
e^{\bm{1}^\top \bm{x}}
\mathbb{P} \left\{
\exists \, \bm{t} \in E \colon
\bm{Y} ( \bm{t} ) - \bm{d} ( \bm{t} )
> \bm{x}
\right\}
\mathop{d \bm{x}}.
\end{equation}

\begin{lemma} \label{lemma:local-Pickands}
Suppose that \(\bm{X}_{u, \bm{\tau}} ( \bm{t} )\), \(\bm{t} \in E\), \(u >
0\), \(\bm{\tau} \in Q_u\) satisfy \ref{B1}, \ref{B2} and \ref{B3}.
Let \(\bm{w} = \Sigma^{-1} \, \widetilde{\bm{b}}\), where
\(\widetilde{\bm{b}}\) is the unique solution of \(\Pi_{\Sigma} ( \bm{b} )\).
If \(\bm{Y}_{\bm{w}} ( \bm{t} )\), \(\bm{t} \in E\) has
cmf \(R ( \bm{t}, \bm{s} ) = \diag ( \bm{w} ) \, K ( \bm{t}, \bm{s} ) \, \diag ( \bm{w} )\) and
\(\bm{d}_{\bm{w}} ( \bm{t} ) = \diag ( \bm{w} ) \, \bm{d} ( \bm{t} )\), then we have
\begin{equation}
\label{eq:36}
\lim_{u \to \infty}
\sup_{\bm{\tau} \in Q_u}
\left|
\frac{
\mathbb{P} \left\{
\exists \, \bm{t} \in E \colon
\bm{X}_{u, \bm{\tau}} ( \bm{t} ) > u \bm{b}
\right\}
}{
\mathbb{P} \left\{
\bm{X}_{u, \bm{\tau}} ( \bm{0} ) > u \bm{b}
\right\}
}
-H_{\bm{Y}_{\bm{w}}, \bm{d}_{\bm{w}}} ( E )
\right|
= 0.
\end{equation}
\end{lemma}

\begin{remark}
If we suppose stronger assumptions on \(\Sigma_{u, \bm{\tau}}\), for instance
\[
\lim_{u \to \infty}
\sup_{\bm{\tau} \in Q_u}
\left\|
u^2 \, \Big[ \Sigma - \Sigma_{u, \bm{\tau}} \Big]
-\Xi
\right\|_F
= 0,
\]
then as \(u \to \infty\)
\[
\mathbb{P} \left\{
\bm{X}_{u, \bm{\tau}} ( \bm{0} )
> u \bm{b}
\right\}
\sim
e^{-\bm{w}^\top \, \Xi \, \bm{w} / 2}
\mathbb{P} \left\{
\bm{\mathcal{N}} > u \bm{b}
\right\},
\]
where \(\bm{\mathcal{N}}\) is a centered Gaussian vector with covariance matrix \(\Sigma\).
\end{remark}
\subsection{Integral estimate}
\label{sec:orgccff620}
\begin{lemma} \label{lemma:integral_estimate}
If a family of Hölder continuous random fields
\(\bm{\chi}_{\bm{x}} ( \bm{t} )\), \(\bm{t} \in [ \bm{0}, \bm{\Lambda} ]\)
measurable in \(\bm{x} \in \mathbb{R}^d\) satisfies
\begin{equation*}
  \sup_{F \subset \{ 1, \ldots, d \}}
  \sup_{\bm{t} \in [ \bm{0}, \bm{\Lambda} ]}
  \bm{w}_F^\top \,
  \mathbb{E} \left\{ \bm{\chi}_{\bm{x}, F} ( \bm{t} ) \right\}
  \leq
  G + \varepsilon \sum_{j = 1}^d | x_j |
\end{equation*}
and
\begin{equation*}
  \sup_{F \subset \{ 1, \ldots, d \}} \sup_{\bm{t} \in [ \bm{0}, \bm{\Lambda} ]}
  \var \left\{ \bm{w}_F^\top \, \bm{\chi}_{\bm{x}, F} ( \bm{t} ) \right\}
  \leq \sigma^2
\end{equation*}
with some constants \(\bm{w} > \bm{0}\), \(\sigma^2 > 0\),
\(G \in \mathbb{R}\) and small enough \(\varepsilon > 0\), then there exist
constants \(C, \, c > 0\) such that the following inequality holds:
\begin{equation*}
  \int_{\mathbb{R}^d}
  e^{\bm{w}^\top \bm{x}} \,
  \mathbb{P} \left\{
    \exists \, \bm{t} \in [ \bm{0}, \bm{\Lambda} ] \colon
    \bm{\chi}_{\bm{x}} ( \bm{t} ) > \bm{x}
  \right\}
  \mathop{d \bm{x}}
  \leq
  C e^{c ( G + \sigma^2 )}.
\end{equation*}
\end{lemma}
\section{Proof of the main theorem}
\label{sec:org41d5d4c}
\label{SEC:proofs}
\subsection{Log-layer bound}
\label{sec:org1e3f6a5}
\begin{lemma}[Log-layer bound] \label{lemma:log-layer-bound}
Suppose \(\bm{X} ( \bm{t} )\) satisfy Assumptions~\ref{A1} to~\ref{A3}.
Then there exist positive constants \(c\), \(u_0\) and \(\bm{\Lambda}_0\) such
that for \(\bm{\Lambda} \geq \bm{\Lambda}_0\) and \(u \geq u_0\)
\begin{equation*}
  \mathbb{P} \left\{
    \exists \, \bm{t} \in [ \bm{0}, \bm{\delta}_u ] \setminus
    u^{-2 / \bm{\beta}} [ \bm{0}, \bm{\Lambda} ] \colon
    \bm{X} ( \bm{t} ) > u \bm{b}
  \right\}
  \leq
  c \,
  \mathbb{P} \left\{
    \bm{X} ( \bm{0} ) > u \bm{b}
  \right\}
  \exp \left(
    -\frac{1}{8}
    \sum_{i = 1}^n
    \xi_i \, \Lambda_i^{\beta_i}
  \right),
\end{equation*}
where \(\xi_i = \bm{w}^\top A_{2, i} \, \bm{w} > 0\) by~\ref{A2.4}.
\end{lemma}

\begin{proof}
For simplicity assume that \(I = \{ 1, \ldots, d \}\), hence \(\widetilde{\bm{b}} = \bm{b}\).
The idea of the proof is to split the log-layer
\([ \bm{0}, u^{-2/\bm{\beta}} \ln^{2/\bm{\beta}} u ] \setminus u^{-2 / \bm{\beta}} [ \bm{0}, \bm{\Lambda} ]\)
into tiny pieces
\begin{equation*}\label{LL:1}
\bm{\Lambda} u^{-2/\bm{\nu}} [ \bm{k}, \bm{k} + \bm{1} ],
\qquad
\bm{k} \in Q_u = \bigcup_{\mathcal{L} \neq \varnothing} Q_u ( \mathcal{L} ),
\end{equation*}
where the union is taken over non-empty subsets \(\mathcal{L}\) of \(\{ 1, \ldots, n \}\)
and
\begin{equation*}
Q_u ( \mathcal{L} )
=
\left\{
\bm{k} \in \mathbb{Z}_+^n \colon
k_i \geq u^{2/\nu_i - 2/\beta_i} / \Lambda_i, \quad i \in \mathcal{L}
\right\}
\cap
\left[ \bm{0}, u^{2/\bm{\nu} - 2/\bm{\beta}} / \bm{\Lambda} \right].
\end{equation*}
Next, derive a suitable \textit{uniform} in \(\bm{k} \in Q_u\) bound for the
Pickands intervals' probabilities
\begin{equation}\label{LL:2}
  \mathbb{P} \left\{
    \exists \, \bm{t} \in \bm{\Lambda} u^{-2/\bm{\nu}} [ \bm{k}, \bm{k} + \bm{1} ] \colon
    \bm{X} ( \bm{t} ) > u \bm{b}
  \right\},
\end{equation}
and then sum them up to obtain an upper bound on the layer's probability. To
this end, let us define a family of random fields
\begin{equation*}\label{LL:3}
\bm{X}_{u, \bm{k}} ( \bm{t} )
=
\bm{X} \left(
u^{-2 / \bm{\nu}}
\left( \bm{\Lambda} \bm{k} + \bm{t} \right)
\right),
\qquad
\bm{k} \in Q_u, \quad
\bm{t} \in [ \bm{0}, \bm{\Lambda} ],
\end{equation*}
and denote the corresonding covariance and variance matrices by
\begin{equation*}\label{LL:4}
R_{u, \bm{k}} ( \bm{t}, \bm{s} )
=
R \left(
u^{-2 / \bm{\nu}}
\left( \bm{\Lambda} \bm{k} + \bm{t} \right),
u^{-2 / \bm{\nu}}
\left( \bm{\Lambda} \bm{k} + \bm{s} \right)
\right),
\qquad
\Sigma_{u, \bm{k}}
= \Sigma \left(
\bm{\Lambda} u^{-2 / \bm{\nu}} \bm{k}
\right).
\end{equation*}

Next, apply the law of total probability
\begin{equation}\label{LL:5}
\mathbb{P} \left\{
\exists \, \bm{t} \in \bm{\Lambda} u^{-2 / \bm{\nu}}
\left[ \bm{k}, \bm{k} + \bm{1} \right] \colon
\bm{X} ( \bm{t} )
>
u \bm{b}
\right\}
=
u^{-d}
\int_{\mathbb{R}^d}
\mathbb{P} \left\{
\exists \, \bm{t} \in [ \bm{0}, \bm{\Lambda} ] \colon
\bm{\chi}_{u, \bm{k}} ( \bm{t} )
>
\bm{x}
\right\}
\varphi_{\Sigma_{u, \bm{k}} } \left( u \bm{b} - \frac{\bm{x}}{u} \right)
\mathop{d \bm{x}},
\end{equation}
where \(\bm{\chi}_{u, \bm{k}} ( \bm{t} )\) denotes the conditional process
\(u \left( \bm{X}_{u, \bm{k}} ( \bm{t} ) - u \bm{b} \right) + \bm{x}\)
given
\(\bm{X}_{u, \bm{k}} ( \bm{0} ) = u \bm{b} - u^{-1} \bm{x}\).

First, bound \(\varphi_{\Sigma_{u, \bm{k}}}\) using~\eqref{prefactor-inequality}
\begin{equation*}\label{LL:6}
\ln \frac{
\varphi_{\Sigma_{u, \bm{k}}} \left( u \bm{b} - \frac{\bm{x}}{u} \right)
}{
\varphi_{\Sigma} ( u \bm{b} )
}
\leq
-\frac{1}{2} \, u^2
\, \bm{b}^\top \,
\Big[
\Sigma_{u, \bm{k}}^{-1}
-\Sigma^{-1}
\Big]
\, \bm{b}
+\bm{b}^\top \Sigma_{u, k}^{-1} \, \bm{x}.
\end{equation*}
Plugging this into~\eqref{LL:5} and noting that
\(u^{-d} \, \varphi_{\Sigma} ( u \bm{b} )
\sim \mathbb{P} \left\{ \bm{X} ( \bm{0} ) > u \bm{b} \right\}\),
we obtain the following bound:
\begin{equation}\label{LL:AB-bound}
\mathbb{P} \left\{
\exists \, \bm{t} \in \bm{\Lambda} u^{-2 / \bm{\nu}}
\left[ \bm{k}, \bm{k} + 1 \right] \colon
\bm{X} ( \bm{t} )
>
u \bm{b}
\right\}
\leq
A B \,
\mathbb{P} \left\{ \bm{X} ( \bm{0} ) > u \bm{b} \right\},
\end{equation}
where
\begin{equation*}\label{LL:8}
\begin{gathered}
A
\coloneqq
\exp \left(
-\frac{1}{2} \, u^2 \, \bm{b}^\top
\left[
\Sigma_{u, \bm{k}}^{-1}
-\Sigma^{-1}
\right]
\bm{b}
\right),
\\[7pt]
B
\coloneqq
\int_{\mathbb{R}^d}
e^{
\bm{b}^\top
\Sigma_{u, \bm{k}}^{-1} \,
\bm{x}
} \,
\mathbb{P} \left\{
\exists \, \bm{t} \in [ \bm{0}, \bm{\Lambda} ] \colon
\bm{\chi}_{u, \bm{k}} ( \bm{t} )
>
\bm{x}
\right\}
\mathop{d \bm{x}}.
\end{gathered}
\end{equation*}
At this point, we split the proof in four parts: bounding \(A\),
bounding \(B\), comparing the bounds and summing them in \(\bm{k}\).
\smallskip

\textbf{Bounding \(A\).}
By~\eqref{exp-prefactor-1}, we have
\begin{equation}\label{LL:exp-prefactor-1}
\bm{b}^\top \Big[ \Sigma^{-1} ( \bm{\tau} ) - \Sigma^{-1} \Big] \bm{b}
= \sum_{i, j = 1}^n
\Xi_{i,j}
\, \tau_i^{\beta_i/2} \, \tau_j^{\beta_j/2}
+o \left( \sum_{i = 1}^n \tau_i^{\beta_i} \right),
\end{equation}
where
\begin{equation*}
\Xi_{i, j} = \bm{w}^\top \widetilde{D}_{i, j} \, \bm{w},
\qquad
\widetilde{D}_{i, j}
= 2 \, A_{2, i} \, \mathbb{1}_{i = j}
+\Big[ A_{6, i, j} + A_{1, i} \Sigma^{-1} A_{1, j}^\top \Big] \mathbb{1}_{i, j \in \mathcal{F}},
\end{equation*}
Using~\eqref{XI-elliptical}, we can bound~\eqref{LL:exp-prefactor-1}
as follows:
\begin{equation*}
\bm{b}^\top \Big[ \Sigma^{-1} ( \bm{\tau} ) - \Sigma^{-1} \Big] \bm{b}
= \sum_{i, j = 1}^n
\Xi_{i,j}
\, \tau_i^{\beta_i/2} \, \tau_j^{\beta_j/2}
+o \left( \sum_{i = 1}^n \tau_i^{\beta_i} \right)
\geq
\frac{3}{2} \sum_{i = 1}^n \bm{w}^\top A_{2, i} \, \bm{w} \, \tau_i^{\beta_i}
\end{equation*}
for \(\bm{\tau}\) close enough to \(\bm{0}\).
Plugging \(\bm{\tau} = u^{-2/\bm{\nu}} \bm{\Lambda} \bm{k}\), we obtain that for all
large enough \(u\) holds
\begin{equation}\label{LL:A}
-\frac{1}{2} \, u^2 \, \bm{b}^\top
\Big[ \Sigma^{-1}_{u, \bm{k}} - \Sigma^{-1} \Big] \bm{b}
\leq
-\frac{3}{4}
\, \sum_{i = 1}^n \bm{w}^\top A_{2, i} \, \bm{w} \,
\, u^{2 - 2 \beta_i / \nu_i}
\, ( \Lambda_i k_i )^{\beta_i}.
\end{equation}

\textbf{Bounding \(B\).}
As a next step, we derive a bound for the integral
\begin{equation} \label{LL:B}
  \int_{\mathbb{R}^d}
  e^{
    \bm{b} ^\top
    \Sigma_{u, \bm{k}}^{-1}
    \bm{x}
  } \,
  \mathbb{P} \left\{
    \exists \, \bm{t} \in [ \bm{0}, \bm{\Lambda} ] \colon
    \bm{\chi}_{u, \bm{k}} ( \bm{t} )
    >
    \bm{x}
  \right\}
  \mathop{d \bm{x}}
  \leq c_1 e^{c_2 ( G + \sigma^2 )},
\end{equation}
using Lemma~\ref{lemma:integral_estimate}. Here \(G \in \mathbb{R}\) and
\(\sigma^2 > 0\) are any such numbers that for all small enough \(\varepsilon\) hold
\begin{equation}\label{eq:log-layer-G-and-sigma}
\begin{gathered}
\sup_{F \subset \{ 1, \ldots, d \}}
\sup_{\bm{t} \in [ \bm{0}, \bm{\Lambda} ]}
\bm{w}_F^\top \,
\mathbb{E} \left\{ \bm{\chi}_{u, \bm{k}, F} ( \bm{t} ) \right\}
\leq
G + \varepsilon \sum_{j = 1}^d | x_j |,
\\[7pt]
\sup_{F \subset \{ 1, \ldots, d \}}
\sup_{\bm{t} \in [ \bm{0}, \bm{\Lambda} ]}
\var \left\{ \bm{w}_F^\top \, \bm{\chi}_{u, \bm{k}, F} ( \bm{t} ) \right\}
\leq \sigma^2.
\end{gathered}
\end{equation}
For our current needs they also must be uniform in \(\bm{k}\). The following
estimate for \(G + \sigma^2\) is proven in the Appendix
(see Section~\ref{LL:calculations})
\begin{equation}
\label{LL:G-plus-sigma2}
G + \sigma^2 =
c_5 \sum_{i = 1}^n \Bigg[
u^{2 - 2 \beta_i / \nu_i} \Lambda_i^{\beta_i}
\Big[
( k_i \vee 1 )^{\beta_i - 1}
+\varepsilon^{-1} \, ( k_i \vee 1 )^{\beta_i - 2}
+\varepsilon \left( k_i^{\beta_i} + 1 \right)
\Big]
+u^{2 - 2 \alpha_i / \nu_i} \Lambda_i^{\alpha_i}
\Bigg].
\end{equation}

\textbf{Comparing the bounds.}
Combining the bound~\eqref{LL:A} for \(A\), the bound~\eqref{LL:B}
for \(B\) with the \(G + \sigma^2\) estimated in~\eqref{LL:G-plus-sigma2},
and plugging all this into the \(AB\) bound~\eqref{LL:AB-bound},
we arrive at the following inequality:
\begin{multline}
\label{LL:lhs}
\ln
\frac{
\mathbb{P} \left\{
\exists \, \bm{t} \in \bm{\Lambda} u^{-2 / \bm{\nu}}
\left[ \bm{k}, \bm{k} + \bm{1} \right] \colon
\bm{X} ( \bm{t} )
>
u \bm{b}
\right\}
}{
c_1 \, \mathbb{P} \left\{ \bm{X} ( \bm{0} ) > u \bm{b} \right\}
}
\\[7pt]
\leq
-\frac{3}{4}
\, \sum_{i = 1}^n \bm{w}^\top A_{2, i} \, \bm{w}
\, u^{2 - 2 \beta_i / \nu_i}
\, ( \Lambda_i k_i )^{\beta_i}
+c_6 \sum_{i = 1}^n \Bigg[
u^{2 - 2 \beta_i / \nu_i} \Lambda_i^{\beta_i}
\Big[
( k_i \vee 1 )^{\beta_i - 1}
\\[7pt]
+\varepsilon^{-1} \, ( k_i \vee 1 )^{\beta_i - 2}
+\varepsilon \left( k_i^{\beta_i} + 1 \right)
\Big]
+u^{2 - 2 \alpha_i / \nu_i} \Lambda_i^{\alpha_i}
\Bigg].
\end{multline}
Setting \(\varepsilon = \bm{w}^\top A_{2, i} \, \bm{w} / 4 \, c_6\), we find that the
\(i\)-th term is at most
\begin{equation}\label{LL:master-bound}
u^{2-2\beta_i/\nu_i} \, \Lambda_i^{\beta_i}
\left[
-\frac{\bm{w}^\top A_{2, i} \, \bm{w}}{2} \, k_i^{\beta_i}
+c_6 \Big[
( k_i \vee 1 )^{\beta_i - 1}
+\varepsilon^{-1} ( k_i \vee 1 )^{\beta_i - 2}
+\varepsilon
\Big]
\right]
+c_6 \, u^{2-2\alpha_i/\nu_i} \, \Lambda_i^{\alpha_i}.
\end{equation}
Note that if \(k_i = 0\), which is only possible if \(i \in \mathcal{L}^c\),
the~\eqref{LL:master-bound} reads:
\begin{equation*}
c_7 \, u^{2-2\beta_i/\nu_i} \, \Lambda_i^{\beta_i}
+c_6 \, u^{2-2\alpha_i/\nu_i} \, \Lambda_i^{\alpha_i}
\end{equation*}
If \(k_i \geq 1\), then \(k_i \vee 1 = k_i\), and~\eqref{LL:master-bound} is
at most
\begin{equation}\label{LL:bound1}
-\frac{\bm{w}^\top A_{2, i} \, \bm{w}}{4} \, k_i^{\beta_i}
u^{2-2\beta_i/\nu_i} \, \Lambda_i^{\beta_i}
+c_6 \, u^{2-2\alpha_i/\nu_i} \, \Lambda_i^{\alpha_i}
\end{equation}
under the following conditions:
\begin{enumerate}
\item if \(i \in \mathcal{L} \cap \mathcal{I}\), \(k_i \geq u^{2/\alpha_i - 2/\beta_i} / \Lambda_i \to \infty\), and we can
choose \(u_0\) such that~\eqref{LL:bound1} holds for all
\(k_i \geq 1\) and \(u \geq u_0\)
\item if \(i \in \mathcal{L}^c \cup \mathcal{I}^c\), there exists \(k_{0, i}\), independent of \(u\),
such that~\eqref{LL:bound1} holds for \(k_i \geq k_{0, i}\).
\end{enumerate}

\textbf{Combined upper bound for \(\bm{k} \geq \bm{k}_0\).}
Denote
\begin{equation*}
k_{0, i} = k_{0, i} \vee 1
\quad \text{for} \quad i \in \mathcal{I}^c
\quad \text{and} \quad
k_{0, i} = u^{2/\nu_i-2/\beta_i}
\quad \text{for} \quad i \in \mathcal{L} \cap \mathcal{I}.
\end{equation*}
We have shown that there exits \(u_0\) and such that if \(u \geq u_0\)
and \(\bm{k} \geq \bm{k}_0\), then the left-hand side of~\eqref{LL:lhs} is
at most
\begin{equation*}
-\frac{1}{4} \, \sum_{i \in \mathcal{L} \cup \mathcal{I}}^n \Big[
u^{2 - 2 \beta_i/\nu_i} \, \Lambda_i^{\beta_i} \,
\bm{w}^\top A_{2, i} \, \bm{w} \, k_i^{\beta_i}
+c_6 \, u^{2 - 2 \alpha_i / \nu_i} \, \Lambda_i^{\alpha_i}
\Big]
+c_8.
\end{equation*}
Hence,
\begin{equation}
\label{LL:k-geq-k0-bound}
\frac{
\mathbb{P} \left\{
\exists \, \bm{t} \in \bm{\Lambda} u^{-2 / \bm{\nu}}
\left[ \bm{k}, \bm{k} + \bm{1} \right] \colon
\bm{X} ( \bm{t} )
>
u \bm{b}
\right\}
}{
c_9 \, \mathbb{P} \left\{ \bm{X} ( \bm{0} ) > u \bm{b} \right\}
}
\\[7pt]
\leq
\prod_{i = 1}^n
\exp \left(
-\frac{\bm{w}^\top A_{2, i} \bm{w}}{4} \,
u^{2-2\beta_i/\nu_i} \,
( \Lambda_i k_i )^{\beta_i}
+c_3 u^{2-2\alpha_i/\nu_i} \, \Lambda_i^{\alpha_i}
\right).
\end{equation}
Let us now show how to cover \(\bm{k}\)'s with \(1 \leq k_i \leq k_{0, i}\),
\(i \in \mathcal{I}^c\). Assume, for the sake of simplicity, that \(\bm{k}\) is such that
there is exactly one \(i\) such that \(1 \leq k_i \leq k_{0, i}\). The general case
can be addressed in a similar way. Note that so far we did not assume anything
about \(\Lambda_i\) except positivity. We want to exploit this fact. To this end,
set \(\widetilde{\Lambda}_i \coloneqq \Lambda_i / k_{0,i}\) and \(\bm{x}' = ( x_j )_{j \neq i}\).
By~\eqref{LL:k-geq-k0-bound} we have that
\begin{equation*}
\mathbb{P} \left\{
\exists \, \bm{t} \in \bm{\Lambda} u^{-2 / \bm{\nu}}
\left[ \bm{k}, \bm{k} + \bm{1} \right] \colon
\bm{X} ( \bm{t} )
>
u \bm{b}
\right\}
\\[7pt]
\leq
\sum_{j = k_{0, i} k_i}^{k_{0, i} ( k_i + 1 ) - 1}
\mathbb{P} \left\{
\begin{aligned}
\displaystyle
& \exists \, t_i \in \widetilde{\Lambda}_i \, u^{-2/\nu_i} [ j, j + 1 ]
\\
\displaystyle
& \exists \, \bm{t}' \in \bm{\Lambda}' u^{-2/\nu_i'} [ \bm{k}', \bm{k}' + \bm{1}' ]
\end{aligned}
\quad \colon \quad
\bm{X} ( \bm{t} ) > u \bm{b}
\right\}
\end{equation*}
Applying~\eqref{LL:k-geq-k0-bound} to the summands, we find that the same
bound~\eqref{LL:k-geq-k0-bound} holds true in this case, but with a
different constant.

\textbf{Bound improvement for \(k_i = 0\).}
Allowing \(k_i = 0\) for some \(i \in \mathcal{L}^c\), we obtain
\begin{multline*}
\frac{
\mathbb{P} \left\{
\exists \, \bm{t} \in \bm{\Lambda} u^{-2 / \bm{\nu}}
\left[ \bm{k}, \bm{k} + \bm{1} \right] \colon
\bm{X} ( \bm{t} )
>
u \bm{b}
\right\}
}{
c_9 \, \mathbb{P} \left\{ \bm{X} ( \bm{0} ) > u \bm{b} \right\}
}
\\[7pt]
\leq
\prod_{i = 1}^n
\exp \left(
-\frac{\bm{w}^\top A_{2, i} \bm{w}}{4} \,
u^{2-2\beta_i/\nu_i} \,
( \Lambda_i k_i )^{\beta_i}
+c_3 u^{2-2\alpha_i/\nu_i} \, \Lambda_i^{\alpha_i}
\right)
\\[7pt]
\times
\prod_{j \in \mathcal{L}^c \colon k_j = 0}
\exp \left( c_6 u^{2-2\beta_j/\nu_j} \Lambda_j^{\beta_j} \right).
\end{multline*}
Note that if \(j \in \mathcal{L}^c \cap \mathcal{I}\), then the corresponding factor is bounded by a
constant, since \(u^{2-2\beta_j/\nu_j} \to 0\).
If \(i \in \mathcal{I}^c\), set \(\widetilde{\Lambda}_i = 1\) and
apply the same trick as above. That is, slice in the \(i\)-th direction and
sum back:
\begin{equation*}
\mathbb{P} \left\{
\exists \, \bm{t} \in \bm{\Lambda} u^{-2 / \bm{\nu}}
\left[ \bm{k}, \bm{k} + \bm{1} \right] \colon
\bm{X} ( \bm{t} )
>
u \bm{b}
\right\}
\\[7pt]
\leq
\sum_{j = 0}^{\Lambda_j + 1}
\mathbb{P} \left\{
\begin{aligned}
\displaystyle
& \exists \, t_i \in \, u^{-2/\nu_i} [ j, j + 1 ]
\\
\displaystyle
& \exists \, \bm{t}' \in \bm{\Lambda}' u^{-2/\nu_i'} [ \bm{k}', \bm{k}' + \bm{1}' ]
\end{aligned}
\quad \colon \quad
\bm{X} ( \bm{t} ) > u \bm{b}
\right\}
\end{equation*}
Since \(\widetilde{\Lambda}_i = 1\), the product in \(j \in \mathcal{L}^c \colon k_j = 0\)
becomes a constant. Other factors remain the same, except of the \(i\)-th,
which gives
\begin{equation*}
\sum_{j = 0}^{\Lambda_i}
\exp \left(
-\frac{\bm{w}^\top A_{2, i} \, \bm{w}}{4} \, j^{\beta_i}
\right)
\leq
c_{10} \, \exp \left(
-\frac{\bm{w}^\top A_{2, i} \, \bm{w}}{4}
\right)
= c_{11}.
\end{equation*}

\textbf{Summing up the bounds in \(\bm{k}\).}
Summing~\eqref{LL:k-geq-k0-bound} in \(\bm{k}\), we obtain
\begin{multline*}
  \sum_{\bm{k} \in Q_u ( \mathcal{L} )}
  \frac{
    \mathbb{P} \left\{
      \exists \, \bm{t} \in \bm{\Lambda} u^{-2 / \bm{\nu}}
      [\bm{k}, \bm{k} + \bm{1}] \colon
      \bm{X} ( \bm{t} )
      >
      u \bm{b}
    \right\}
  }{
    c_{12} \,
    \mathbb{P} \left\{
      \bm{X} ( \bm{0} ) > u \bm{b}
    \right\}
  }
  \\[7pt]
  \leq
  \prod_{i \in \mathcal{I}}
  \sum_{k_i \neq 0}
  \exp \left(
  -\frac{\bm{w}^\top A_{2, i} \, \bm{w}}{4} \,
  u^{2-2\beta_i/\nu_i} \,
  ( \Lambda_i k_i )^{\beta_i}
  +c_3 \, \Lambda_i^{\alpha_i}
  \right)
  \\[7pt]
  \times
  \prod_{j \in \mathcal{I}^c}
  \sum_{k_j \neq 0}
  \exp \left(
  -\frac{\bm{w}^\top A_{2, j} \, \bm{w}}{4} \,
  ( \Lambda_j k_j )^{\beta_j}
  +c_3 u^{2-2\alpha_j/\beta_j} \, \Lambda_j^{\alpha_j}
  \right).
\end{multline*}
If \(i \in \mathcal{I}\), then
\begin{multline*}
\sum_{k_i \geq u^{2/\alpha_i - 2/\beta_i}}
\exp \left(
-\frac{1}{4} \,
u^{2-2\beta_i/\nu_i} \,
\bm{w}^\top A_{2, i} \, \bm{w} \, ( \Lambda_i k_i )^{\beta_i}
+c_3 \, \Lambda_i^{\alpha_i}
\right)
\\[7pt]
\leq
c_5
\exp \left(
-\frac{\bm{w}^\top A_{2, i} \, \bm{w}}{4} \,
\Lambda_i^{\beta_i}
+c_3 \, \Lambda_i^{\alpha_i}
\right)
\leq
c_5
\exp \left(
-\frac{1}{8} \,
\bm{w}^\top A_{2, i} \, \bm{w} \, \Lambda_i^{\beta_i}
\right),
\end{multline*}
where the last inequality is true for \(\Lambda_i \geq \Lambda_{0, i}\) with
\(\Lambda_{0, i}^{\beta_i - \alpha_i} \, \bm{w}^\top A_{2, i} \, \bm{w} \geq 8 \, c_3\).
Same upper bound works for \(i \in \mathcal{I}^c\).
\end{proof}

\begin{remark}
Note that \(\bm{w}^\top A_{2, i} \, \bm{w} > 0\) has to be satisfied for all
\(i\)'s, because otherwise one of the sums in \(k_i\) (as, for example, the
last sum of the proof) may be infinite.
\end{remark}
\subsection{Double sum bound}
\label{sec:org03d5a2e}
Define for \(\bm{\tau}, \bm{\lambda}, \bm{S} \in \mathbb{R}_{ + }^n\) the double events' probabilities by
\begin{equation} \label{def:double-event}
  P_{\bm{b}} ( \bm{\tau}, \bm{\lambda}, \bm{S} )
  \coloneqq \mathbb{P} \left\{
    \begin{aligned}
    & \exists \, \bm{t} \in u^{-2 / \bm{\nu}} [ \bm{\tau}, \bm{\tau} + \bm{S} ] \colon
    & \bm{X} ( \bm{t} ) > u \bm{b},
    \\[3pt]
    & \exists \, \bm{s} \in u^{-2 / \bm{\nu}} [ \bm{\lambda}, \bm{\lambda} + \bm{S} ] \colon
    & \bm{X} ( \bm{s} ) > u \bm{b}
    \end{aligned}
  \right\}.
\end{equation}

\begin{lemma}[Double sum bound] \label{lemma:double-sum}
Let double events' displacements \(\bm{\tau}, \, \bm{\lambda} \in [ \bm{0}, \bm{N}_u ]\) be such that
\begin{description}[leftmargin = * , before={\renewcommand\makelabel[1]{\bfseries ##1.}}]
\item [1] There is no offset in Piterbarg and Talagrand type coordinates:
\(\bm{\tau}_{\mathcal{J} \cup \mathcal{K}} = \bm{\lambda}_{\mathcal{J} \cup \mathcal{K}} = \bm{0}_{\mathcal{J} \cup \mathcal{K}}\).
\item [2] There exists a (possibly empty) subset of Pickands type coordinates
\(\mathcal{I}' \subset \mathcal{I}\) in which there is no offset:
\(\bm{\tau}_{\mathcal{I}'} = \bm{\lambda}_{\mathcal{I}'}\).
\item [3] The offset in the remaining Pickands type coordinates
\(\mathcal{I}_2 \coloneqq \mathcal{I} \setminus \mathcal{I}'\)
is strictly greater than \(\bm{S}\):
\(\bm{\lambda}_{\mathcal{I}_2} - \bm{\tau}_{\mathcal{I}_2} > \bm{S}_{\mathcal{I}_2}\).
\end{description}
Then, there exists \(u_0 \geq 0\) and \(\bm{S}_0 > \bm{0}\) such that for all
\(u \geq u_0\) and \(\bm{S} \geq \bm{S}_0\) holds
\begin{equation*}
  \frac{
    P_{\bm{b}} ( \bm{\tau}, \bm{\lambda}, \bm{S} )
  }{
    H_u ( \bm{\tau} ) \,
    \mathbb{P} \left\{ \bm{X} ( \bm{0} ) > u \bm{b} \right\}
  }
  \leq
  C
  \prod_{i \in \mathcal{J} \cup \mathcal{K}} e^{c S_i^{\beta_i}}
  \prod_{i \in \mathcal{I}'} S_i
  \prod_{i \in \mathcal{I}_2}
  \left[
  \frac{S_i}{\lambda_i - \tau_i - S_i}
  \right]^2
  \exp \left(
  -\frac{\varkappa_i}{64} \,
  ( \lambda_i - \tau_i - S_i )^{\alpha_i}
  \right)
\end{equation*}
with some constants \(C, \ c > 0\) and
\begin{equation*}
H_u ( \bm{\tau} ) =
\exp \left(
-\frac{1}{4} \sum_{i \in \mathcal{I}_2} \xi_i \, \tau_i^{\beta_i} u^{2 - 2 \beta_i / \nu_i}
\right).
\end{equation*}
\end{lemma}

\begin{remark}
We want to stress the fact that the conditions of the lemma demand that there be
no Pickands type coordinates \(i \in \mathcal{I}\) with offsets smaller than
\(S_i\), except those in which the offset is zero. This is not a coincidence,
since the adjacent intervals are to be dealt with differently (see proof of
Theorem~\ref{main-theorem} for details).
Note also that if \(\mathcal{I}' = \mathcal{I}\), the assertion of the lemma is trivial.
\end{remark}

\begin{proof}
To begin with, note that if \(\bm{X} ( \bm{t} )\) and \(\bm{X} ( \bm{s} )\)
exceed \(u \bm{b}\), then their sum exceeds \(2 u \bm{b}\):
\begin{multline*}
  P_{\bm{b}} ( \bm{\tau}, \bm{\lambda}, \bm{S} )
  \leq
  \mathbb{P} \left\{
    \exists \, \bm{t} \in u^{-2 / \bm{\nu}} [ \bm{\tau}, \bm{\tau} + \bm{S} ], \
    \exists \, \bm{s} \in u^{-2 / \bm{\nu}} [ \bm{\lambda}, \bm{\lambda} + \bm{S} ] \colon
    \bm{X} ( \bm{t} ) + \bm{X} ( \bm{s} ) > 2 u \bm{b}
  \right\}
  \\[7pt] =
  \mathbb{P} \left\{
    \exists \, \bm{t}, \bm{s} \in [ \bm{0}, \bm{S} ] \colon
    \bm{X}_{u, \bm{\tau}, \bm{\lambda}} ( \bm{t}, \bm{s} ) > u \bm{b}
  \right\}
\end{multline*}
where
\begin{equation*}
  \bm{X}_{u, \bm{\tau}, \bm{\lambda}} ( \bm{t}, \bm{s} )
  = \frac{1}{2} \left(
    \bm{X} ( u^{-2/\bm{\nu}} \bm{\tau} + u^{-2 / \bm{\nu}} \bm{t} )
    +\bm{X} ( u^{-2/\bm{\nu}} \bm{\lambda} + u^{-2 / \bm{\nu}} \bm{s} )
  \right).
\end{equation*}
Henceforth, we shall seek a bound of the latter probability. To this end, we
employ an idea analogous to that of the proof of
Lemma~\ref{lemma:log-layer-bound}: first, apply the law of total
probability
\begin{equation} \label{eq:13}
  \mathbb{P} \left\{
    \exists \, \bm{t}, \, \bm{s} \in [ \bm{0}, \bm{S} ] \colon
    \bm{X}_{u, \bm{\tau}, \bm{\lambda}} ( \bm{t}, \bm{s} ) > u \bm{b}
  \right\}
  = u^{-d} \int_{\mathbb{R}^d}
  \mathbb{P} \left\{
    \exists \, \bm{t}, \bm{s} \in [ \bm{0}, \bm{S} ] \colon
    \bm{\chi}_{u, \bm{\tau}, \bm{\lambda}} ( \bm{t}, \bm{s} )
    > \bm{x}
  \right\}
  \varphi_{\, \Sigma_u ( \bm{\tau}, \bm{\lambda} ) } \left( u \bm{b} - \frac{\bm{x}}{u} \right)
  \mathop{d \bm{x}},
\end{equation}
where the conditional random field is defined by
\begin{equation*}
\bm{\chi}_{u, \bm{\tau}, \bm{\lambda}} ( \bm{t}, \bm{s} )
=
u \left(
\bm{X}_{u, \bm{\tau}, \bm{\lambda}} ( \bm{t}, \bm{s} )
-\bm{X}_{u, \bm{\tau}, \bm{\lambda}} ( \bm{0}, \bm{0} )
\ \middle| \
\bm{X}_{u, \bm{\tau}, \bm{\lambda}} ( \bm{0}, \bm{0} )
= u \bm{b} - \frac{\bm{x}}{u}
\right)
\end{equation*}
and \(\Sigma_u ( \bm{\tau}, \bm{\lambda} )\) is the variance matrix:
\begin{equation*}
\Sigma_u ( \bm{\tau}, \bm{\lambda} )
= \mathbb{E} \left\{
\bm{X}_{u, \bm{\tau}, \bm{\lambda}} ( \bm{0}, \bm{0} ) \,
\bm{X}_{u, \bm{\tau}, \bm{\lambda}} ( \bm{0}, \bm{0} )^\top
\right\}.
\end{equation*}
Next, bounding the exponential prefactor similarly
to~\eqref{prefactor-inequality} as follows:
\begin{equation*}
\ln \frac{
\varphi_{ \Sigma_u ( \bm{\tau}, \bm{\lambda} ) } ( u \bm{b} - u^{-1} \bm{x} )
}{
\varphi_{ \Sigma } ( u \bm{b} )
}
\leq
-\frac{1}{2} \, u^2 \,
\bm{b}^\top \Big[ \Sigma_u^{-1} ( \bm{\tau}, \bm{\lambda} ) - \Sigma^{-1} \Big] \, \bm{b}
+\bm{b}^\top \Sigma_u^{-1} ( \bm{\tau}, \bm{\lambda} ) \, \bm{x},
\end{equation*}
and using
\begin{equation*}
\mathbb{P} \{ \bm{X} ( \bm{0} ) > u \bm{b} \}
\sim
u^{-d} \, \varphi_{\Sigma} ( u \bm{b} ),
\end{equation*}
we obtain
\begin{equation} \label{DS:AB-bound}
  \mathbb{P} \left\{
    \exists \, \bm{t}, \, \bm{s} \in [ \bm{0}, \bm{S} ] \colon
    \bm{X}_{u, \bm{\tau}, \bm{\lambda}} ( \bm{t}, \bm{s} ) > u \bm{b}
  \right\}
  \leq
  A \, B \,
  \mathbb{P} \{ \bm{X} ( \bm{0} ) > u \bm{b} \},
\end{equation}
where
\begin{gather}
  \label{double-sum-A}
  A \coloneqq
  \exp \left(
  -\frac{1}{2}
  \, u^2 \, \bm{b}^\top \Big[ \Sigma_u^{-1} ( \bm{\tau}, \bm{\lambda} ) - \Sigma^{-1} \Big] \, \bm{b}
  \right),
  \\[7pt]
  \label{double-sum-B}
  B \coloneqq
  \int_{\mathbb{R}^d}
  \exp \left( \bm{b}^\top \Sigma_u^{-1} ( \bm{\tau}, \bm{\lambda} ) \, \bm{x} \right)
  \mathbb{P} \left\{
    \exists \, \bm{t}, \bm{s} \in [ \bm{0}, \bm{S} ] \colon
    \bm{\chi}_{u, \bm{\tau}, \bm{\lambda}} ( \bm{t}, \bm{s} )
    > \bm{x}
  \right\}
  \mathop{d \bm{x}}.
\end{gather}
At this point we split the proof in three parts: bounding \(A\), bounding
\(B\) and comparing the bounds.
\bigskip

\textbf{Bounding \(A\).}
By~\eqref{DS-prefactor}, we have
\begin{multline*}
\bm{b}^\top \Big[ \Sigma^{-1} ( \bm{\tau}, \bm{\lambda} ) - \Sigma^{-1} \Big] \bm{b}
\sim
\sum_{i = 1}^n
\left[
\bm{w}^\top A_{2, i} \, \bm{w} \, \Big[ \tau_i^{\beta_i} + \lambda_i^{\beta_i} \Big]
+\frac{\bm{w}^\top A_{5, i} \, \bm{w}}{2} \,
\left| \lambda_i - \tau_i \right|^{\alpha_i}
\right]
\\[7pt]
+\frac{1}{4} \,
\sum_{i, j \in \mathcal{F}}
\Xi_{i, j}
\Big[
\tau_i^{\beta_i/2} \, \lambda_j^{\beta_j/2}
+\lambda_i^{\beta_i/2} \, \tau_j^{\beta_j/2}
+\tau_i^{\beta_i/2} \, \tau_j^{\beta_j/2}
+\lambda_i^{\beta_i/2} \, \lambda_j^{\beta_j/2}
\Big]
\end{multline*}
with the error of order
\begin{equation*}
o \left( \sum_{i = 1}^n \Big[
\tau_i^{\beta_i} + \lambda_i^{\beta_i} + | \lambda_i - \tau_i |^{\alpha_i}
\Big] \right).
\end{equation*}
Using~\eqref{XI-elliptical}, we find that
\begin{equation*}
\bm{b}^\top \Big[ \Sigma^{-1} ( \bm{\tau}, \bm{\lambda} ) - \Sigma^{-1} \Big] \bm{b}
\geq
\frac{1}{2}
\sum_{i = 1}^n
\left[
\bm{w}^\top A_{2, i} \, \bm{w} \, \Big[ \tau_i^{\beta_i} + \lambda_i^{\beta_i} \Big]
+\frac{\bm{w}^\top A_{5, i} \, \bm{w}}{2} \,
\left| \lambda_i - \tau_i \right|^{\alpha_i}
\right]
\end{equation*}
for \(\bm{\tau}\) and \(\bm{\lambda}\) sufficiently close to \(\bm{0}\). Hence,
\begin{multline}\label{DS-A-bound}
-\frac{1}{2} \, u^2 \,
\bm{b}^\top \Big[ \Sigma_u^{-1} ( \bm{\tau}, \bm{\lambda} ) - \Sigma^{-1} \Big] \bm{b}
\leq
-\frac{1}{4}
\sum_{i = 1}^n
\Bigg[
u^{2 - 2 \beta_i / \nu_i} \,
\bm{w}^\top A_{2, i} \, \bm{w} \, \Big[ \tau_i^{\beta_i} + \lambda_i^{\beta_i} \Big]
\\[7pt]
+u^{2 - 2 \alpha_i / \nu_i} \,
\frac{\bm{w}^\top A_{5, i} \, \bm{w}}{2} \,
\left| \lambda_i - \tau_i \right|^{\alpha_i}
\Bigg]
\end{multline}

\textbf{Bounding \(B\).}
As in the proof of Lemma~\ref{lemma:log-layer-bound}, our next step
consists in deriving a bound for the integral
\begin{equation} \label{DS:B}
  \int_{\mathbb{R}^d}
  e^{
    \bm{b} ^\top
    \Sigma_u^{-1} ( \bm{\tau}, \bm{\lambda} ) \,
    \bm{x}
  } \,
  \mathbb{P} \left\{
    \exists \, \bm{t}, \bm{s} \in [ \bm{0}, \bm{S} ] \colon
    \bm{\chi}_{u, \bm{\tau}, \bm{\lambda}} ( \bm{t}, \bm{s} )
    >
    \bm{x}
  \right\}
  \mathop{d \bm{x}}
  \leq c_1 e^{c_2 ( G + \sigma^2 )},
\end{equation}
using Lemma~\ref{lemma:integral_estimate}. Here \(G \in \mathbb{R}\) and
\(\sigma^2 > 0\) are any such numbers that for all small enough \(\varepsilon\)
the following two inequalities hold:
\begin{equation}\label{DS:G-and-sigma}
\begin{gathered}
\sup_{F \subset \{ 1, \ldots, d \}}
\sup_{\bm{t} \in [ \bm{0}, \bm{\Lambda} ]}
\bm{w}_F^\top \,
\mathbb{E} \left\{ \bm{\chi}_{u, \bm{\tau}, \bm{\lambda}, F} ( \bm{t}, \bm{s} ) \right\}
\leq
G + \varepsilon \sum_{j = 1}^d | x_j |,
\\[7pt]
\sup_{F \subset \{ 1, \ldots, d \}}
\sup_{\bm{t} \in [ \bm{0}, \bm{\Lambda} ]}
\var \left\{ \bm{w}_F^\top \, \bm{\chi}_{u, \bm{\tau}, \bm{\lambda}, F} ( \bm{t} ) \right\}
\leq \sigma^2.
\end{gathered}
\end{equation}
For our current needs they also must be uniform in \(\bm{k}\). The following
estimate for \(G + \sigma^2\) is proven in the Appendix
(see Section~\ref{DS:calculations}):
\begin{multline}
\label{DS:G-plus-sigma2}
G + \sigma^2
=
c_1
\sum_{i = 1}^n \Bigg[
u^{2 - 2 \alpha_i / \nu_i} \,
\Big[
S_i^{\alpha_i}
+S_i \, \left( ( \lambda_i - \tau_i ) \vee S_i \right)^{\alpha_i - 1}
+\varepsilon ( \lambda_i - \tau_i )^{\alpha_i}
\Big]
\\[7pt]
+u^{2-2\beta_i/\nu_i} \,
\Big[
S_i \, ( \lambda_i \vee S_i )^{\beta_i - 1}
+\varepsilon \, ( \lambda_i \vee S_i )^{\beta_i}
+\varepsilon^{-1} \, S_j^2 \, ( \lambda_j \vee S_j )^{\beta_j - 2}
\Big]
\Bigg].
\end{multline}

\textbf{Comparing the bounds.}
Combining~\eqref{DS-A-bound} for \(A\), the bound~\eqref{DS:B} for
\(B\) with \(G + \sigma^2\) given by~\eqref{DS:G-plus-sigma2}, and plugging
all this into the \(AB\) bound~\eqref{DS:AB-bound}, we arrive at the
following inequality:
\begin{align*}
&
\ln
\frac{
P_{\bm{b}} ( \bm{\tau}, \bm{\lambda}, \bm{S} )
}{
\mathbb{P} \{ \bm{X} ( \bm{0} ) > u \bm{b} \}
}
\\[7pt]
& \hspace{30pt}
\leq
-\frac{1}{4} \sum_{i = 1}^n
\left[
u^{2 - 2 \beta_i / \nu_i} \,
\bm{w}^\top A_{2, i} \, \bm{w} \,
\Big[ \tau_i^{\beta_i} + \lambda_i^{\beta_i} \Big]
+u^{2 - 2 \alpha_i / \nu_i} \,
\frac{\bm{w}^\top A_{5, i} \, \bm{w}}{4} \, ( \lambda_i - \tau_i )^{\alpha_i}
\right]
\\[7pt]
& \hspace{30pt} +
c_1
\sum_{i = 1}^n \Bigg[
u^{2 - 2 \alpha_i / \nu_i} \,
\Big[
S_i^{\alpha_i}
+S_i \, \left( ( \lambda_i - \tau_i ) \vee S_i \right)^{\alpha_i - 1}
+\varepsilon ( \lambda_i - \tau_i )^{\alpha_i}
\Big]
\\[7pt]
+u^{2-2\beta_i/\nu_i} \,
\Big[
S_i \, ( \lambda_i \vee S_i )^{\beta_i - 1}
+\varepsilon \, ( \lambda_i \vee S_i )^{\beta_i}
+\varepsilon^{-1} \, S_j^2 \, ( \lambda_j \vee S_j )^{\beta_j - 2}
\Big]
\Bigg]
\span \omit.
\end{align*}
Let us exclude the terms
\begin{equation*}
\ln H_u ( \bm{\tau} ) \coloneqq
-\frac{1}{4} \sum_{i=0}^n \bm{w}^\top A_{2, i} \, \bm{w} \,
u^{2 - 2 \beta_i / \nu_i} \, \tau_i^{\beta_i}
\end{equation*}
from further considerations, since it will be useful for us as it is. Note that
\(H_u\) will appear without alterations in conclusion of the lemma. Setting
\begin{equation*}
\varepsilon =
\min \left\{
\frac{\bm{w}^\top A_{5, i} \, \bm{w}}{32 \, c_1},
\frac{\bm{w}^\top A_{2, i} \, \bm{w}}{8 \, c_1}
\right\},
\end{equation*}
we find that that the \(i\)-th term is at most
\begin{multline}\label{DS:master-bound}
u^{2 - 2 \beta_i / \nu_i} \left[
-\frac{\bm{w}^\top A_{2, i} \, \bm{w}}{8} \, \lambda_i^{\beta_i}
+c_1
\Big[
S_i \, ( \lambda_i \vee S_i )^{\beta_i - 1}
+\varepsilon^{-1} \, S_j^2 \, ( \lambda_j \vee S_j )^{\beta_j - 2}
\Big]
\right]
\\[7pt]
+u^{2 - 2 \alpha_i / \nu_i} \,
\left[
-\frac{\bm{w}^\top A_{5, i} \, \bm{w}}{32} \, ( \lambda_i - \tau_i )^{\alpha_i}
+c_1
\Big[
S_i^{\alpha_i}
+S_i \, \left( ( \lambda_i - \tau_i ) \vee S_i \right)^{\alpha_i - 1}
\Big]
\right].
\end{multline}

\textbf{Case \(\tau_i = \lambda_i = 0\).}
By assumptions of the lemma, this happens if and only if \(i \in \mathcal{K} \cup \mathcal{J}\). The
right-hand side of~\eqref{DS:master-bound} reads:
\begin{equation*}
c_2 \, S_i^{\beta_i} + c_3 \, u^{2 - 2 \alpha_i / \nu_i} \, S_i^{\alpha_i}
\end{equation*}
with some new \(c_2, \, c_3 > 0\). This bound can be further simplified if we
note that for \(i \in \mathcal{K}\) the second term tends to zero as \(u \to \infty\) for a
fixed \(S_i\). Hence, this contribution is at most
\begin{equation*}
c_2 \, S_i^{\beta_i} + c_4.
\end{equation*}

\textbf{Case \(\lambda_i = \tau_i \neq 0\).}
By assumptions of the lemma, this happens if and only if \(i \in \mathcal{I}' \subsetneq \mathcal{I}\). Note
that this set may be empty. The right-hand side of~\eqref{DS:master-bound}
reads:
\begin{equation*}
u^{2 - 2 \beta_i / \nu_i}
\left[
-\frac{\bm{w}^\top A_{2, i} \, \bm{w}}{4} \, \lambda_i^{\beta_i}
+c_1
\Big[
S_i \, ( \lambda_i \vee S_i )^{\beta_i - 1}
+\varepsilon^{-1} \, S_j^2 \, ( \lambda_j \vee S_j )^{\beta_j - 2}
\Big]
\right]
+c_4 \, S_i^{\alpha_i}.
\end{equation*}
There exists \(N_i\) such that for all \(\lambda_i \geq N_i \, S_i\) this
contribution may be bounded from above by
\begin{equation*}
-\frac{\bm{w}^\top A_{2, i} \, \bm{w}}{8} \,
u^{2 - 2 \beta_i / \nu_i} \,
\lambda_i^{\beta_i}
+c_4 \, S_i^{\alpha_i}.
\end{equation*}

\textbf{Case \(\lambda_i - \tau_i > S_i\).}
By assumptions of the lemma, this condition holds for all \(i \in \mathcal{I} \setminus \mathcal{I}'\), and
this set is non-empty. The right-hand side of~\eqref{DS:master-bound}
reads:
\begin{equation*}
u^{2 - 2 \beta_i / \nu_i}
\left[
-\frac{\bm{w}^\top A_{2, i} \, \bm{w}}{4} \, \lambda_i^{\beta_i}
+c_1
\Big[
S_i \, \lambda_i^{\beta_i - 1}
+\varepsilon^{-1} \, S_j^2 \, \lambda_j^{\beta_j - 2}
\Big]
\right]
+\left[
-\frac{\bm{w}^\top A_{5, i} \, \bm{w}}{32} \, ( \lambda_i - \tau_i )^{\alpha_i}
+c_5 S_i \, ( \lambda_i - \tau_i )^{\alpha_i - 1}
\right].
\end{equation*}
There exists \(N_i\) such that for all \(\lambda_i - \tau_i \geq N_i S_i\) this
contribution is at most
\begin{equation*}
-\frac{\bm{w}^\top A_{2, i} \, \bm{w}}{8} \,
u^{2 - 2 \beta_i / \nu_i} \,
\lambda_i^{\beta_i}
-\frac{\bm{w}^\top A_{5, i} \, \bm{w}}{64} \, ( \lambda_i - \tau_i )^{\alpha_i}.
\end{equation*}

\textbf{Combined bound.}
We have obtained the following inequality: if
\begin{equation}\label{DS:bound-cond}
\lambda_i \geq N_i \, S_i
\quad \text{for} \quad
i \in \mathcal{I}',
\quad \text{and} \quad
\lambda_i - \tau_i \geq N_i \, S_i
\quad \text{for} \quad
i \in \mathcal{I} \setminus \mathcal{I}',
\end{equation}
then
\begin{multline}\label{DS:combined1}
\ln
\frac{
P_{\bm{b}} ( \bm{\tau}, \bm{\lambda}, \bm{S} )
}{
c_7 \, H_u ( \bm{\tau} ) \,
\mathbb{P} \{ \bm{X} ( \bm{0} ) > u \bm{b} \}
}
\leq
-\frac{1}{8} \sum_{i \in \mathcal{I}}
u^{2 - 2 \beta_i / \nu_i} \,
\bm{w}^\top A_{2, i} \, \bm{w}
\, \lambda_i^{\beta_i}
\\[7pt]
-\frac{1}{64} \sum_{i \in \mathcal{I} \setminus \mathcal{I}'}
\bm{w}^\top A_{5, i} \, \bm{w} \, ( \lambda_i - \tau_i )^{\alpha_i}
+c_6 \sum_{i \in \mathcal{I}'} S_i^{\alpha_i}
+c_6 \sum_{i \in \mathcal{J} \cup \mathcal{K}} S_i^{\beta_i}.
\end{multline}
Next, we want to lift conditions~\eqref{DS:bound-cond}.

\textbf{Lifting the condition \(\lambda_i - \tau_i \geq N_i \, S_i\).}
Assume that there is exactly one \(i \in \mathcal{I} \setminus \mathcal{I}'\) such that
\(\lambda_i - \tau_i < N_i \, S_i\). The general case can be addressed in the same way.
Define \(\Delta \coloneqq ( \lambda_i - \tau_i - S_i ) / N_i\), and slice
\(P_{\bm{b}}\) in the \(i\)-th direction using this new scale:
\begin{equation}\label{DS:sum-bound}
P_{\bm{b}} ( \bm{\tau}, \bm{\lambda}, \bm{S} )
\leq
\sum_{j = \tau_i / \Delta}^{( \tau_i + S_i ) / \Delta}
\sum_{k = \lambda_i / \Delta}^{( \lambda_i + S_i ) / \Delta}
P_{\bm{b}} ( k, j, \Delta )
\leq
S_i^2 \, \Delta^{-2} \,
P_{\bm{b}} \left(
j^{ * } \Delta, \, k_{ * } \, \Delta, \, \Delta
\right),
\end{equation}
where \(j^{ * } = ( \tau_i + S_i ) / \Delta\), \(k_{ * } = \lambda_i / \Delta\), and
\begin{equation*}
P_{\bm{b}} ( \tau, \lambda, \Delta )
=
\mathbb{P} \left\{
\begin{aligned}
& \exists \, t_i \in u^{-2/\nu_i} [ \tau, \tau + \Delta ],
&
\\[3pt]
& \exists \, \bm{t}' \in u^{-2/\bm{\nu}'}  [ \bm{\tau}', \bm{\tau}' + \bm{S}' ] \colon
& \bm{X} ( \bm{t} ) > u \bm{b},
\\[3pt]
& \exists \, s_i \in u^{-2/\nu_i} [ \lambda, \lambda + \Delta ],
&
\\[3pt]
& \exists \, \bm{s}' \in u^{-2/\bm{\nu}'} [ \bm{\lambda}', \bm{\lambda}' + \bm{S}' ] \colon
& \bm{X} ( \bm{s} ) > u \bm{b}
\end{aligned}
\right\},
\qquad
\bm{x}' = ( x_j )_{j \neq i}.
\end{equation*}
Note that
\begin{equation*}
k_{ * } \, \Delta - j^{ * } \Delta
= \lambda_i - \tau_i - S_i
= N_i \Delta,
\end{equation*}
which means that we can apply~\eqref{DS:combined1}. Therefore,
\begin{multline}\label{DS:combined2}
\ln
\frac{
P_{\bm{b}} ( \bm{\tau}, \bm{\lambda}, \bm{S} )
}{
c_7 \, H_u ( \bm{\tau} ) \,
\mathbb{P} \{ \bm{X} ( \bm{0} ) > u \bm{b} \} \,
\prod_{i \in \mathcal{I} \setminus \mathcal{I}'}
S_i^2 \, \Delta^{-2}
}
\leq
-\frac{1}{8} \sum_{i \in \mathcal{I}}
u^{2 - 2 \beta_i / \nu_i} \,
\bm{w}^\top A_{2, i} \, \bm{w} \,
 \lambda_i^{\beta_i}
\\[7pt]
-\frac{1}{64} \sum_{i \in \mathcal{I} \setminus \mathcal{I}'} \bm{w}^\top A_{5, i} \,
\bm{w} \, ( \lambda_i - \tau_i - S_i )^{\alpha_i}
+c_6 \sum_{i \in \mathcal{I}'} S_i^{\alpha_i}
+c_6 \sum_{i \in \mathcal{J} \cup \mathcal{K}} S_i^{\beta_i}.
\end{multline}
which is now valid without the second condition of~\eqref{DS:bound-cond}.

\textbf{Lifting the condition \(\lambda_i \geq N_i \, S_i\).}
Assume now that there is exactly one \(i \in \mathcal{I}'\) such that
\(\lambda_i < N_i \, S_i\). Recall that for \(i \in \mathcal{I}'\) holds \(\lambda_i = \tau_i\).
Using the same approach, take \(\Delta \coloneqq ( \lambda_i \wedge 1 ) / N_i\). Then,
\begin{equation}\label{DS:sum-bound2}
P_{\bm{b}} ( \bm{\tau}, \bm{\lambda}, \bm{S} )
\leq
\sum_{k \neq j}
P_{\bm{b}} ( k \Delta, j \Delta, \Delta )
+\sum_{k = j} P_{\bm{b}} ( j \Delta, j \Delta, \Delta ).
\end{equation}
where the sums are taken over
\begin{equation*}
\left\{
( k, j ) \in \mathbb{Z}^2_{ + } \colon
\lambda_i / \Delta \leq k, \, j \leq ( \lambda_i + S_i ) / \Delta
\right\}.
\end{equation*}
Note that all the results up to this point were valid without any assumptions on
\(\bm{S} > \bm{0}\). Now, assuming that \(S_i\) is large enough, we can use
the bounds proven above to show that
\begin{equation*}
\sum_{| k - j | \geq 2}
P_{\bm{b}} ( k \Delta, j \Delta, \Delta )
\leq c_8 \sum_{k = j} P_{\bm{b}} ( j \Delta, j \Delta, \Delta ).
\end{equation*}
It can also be shown that
\begin{equation*}
\sum_{| k - j | = 1} P_{\bm{b}} ( k \Delta, j \Delta, \Delta )
\leq c_9 \sum_{k = j} P_{\bm{b}} ( j \Delta, j \Delta, \Delta ).
\end{equation*}
It is therefore enough to bound the second sum of~\eqref{DS:sum-bound2}.
We have
\begin{equation}\label{DS:sum-bound3}
P_{\bm{b}} ( \bm{\tau}, \bm{\lambda}, \bm{S} )
\leq
c_9 \, S_i \, \Delta^{-1} \, P_{\bm{b}} ( j_{ * } \Delta, j_{ * } \Delta, \Delta ).
\end{equation}
Since
\(j_{ * } \Delta = \lambda_i \geq \lambda_i \wedge 1 = N_i \, \Delta\),
the bound~\eqref{DS:combined2} is applicable. The \(i\)-th term of the
right-hand side is
\begin{equation*}
-\frac{\bm{w}^\top A_{2, i} \, \bm{w}}{8} \,
u^{2 - 2 \beta_i / \nu_i} \, ( j_{ * } \, \Delta )^{\beta_i}
+c_6 \, \Delta_i^{\alpha_i}.
\end{equation*}
Note that \(j_{ * } = \lambda_i / \Delta\), so the first term did not change, whereas the
second is now bounded by a constant: \(c_6 ( \lambda_i \wedge 1 )^{\alpha_i} \leq c_6\). We have
thus shown that the following bound
\begin{multline}\label{DS:combined3}
\ln
\frac{
P_{\bm{b}} ( \bm{\tau}, \bm{\lambda}, \bm{S} )
}{
c_{10} \, H_u ( \bm{\tau} ) \,
\mathbb{P} \{ \bm{X} ( \bm{0} ) > u \bm{b} \} \,
\prod_{i \in \mathcal{I} \setminus \mathcal{I}'}
S_i^2 \, \Delta^{-2}
\prod_{i \in \mathcal{I}} S_i \, \Delta^{-1}
}
\\[7pt]
\leq
-\frac{1}{8} \sum_{i \in \mathcal{I}}
u^{2 - 2 \beta_i / \nu_i} \,
\bm{w}^\top A_{2, i} \, \bm{w} \,
 \lambda_i^{\beta_i}
-\frac{1}{64} \sum_{i \in \mathcal{I} \setminus \mathcal{I}'} \bm{w}^\top A_{5, i} \,
\bm{w} \, ( \lambda_i - \tau_i - S_i )^{\alpha_i}
+c_6 \sum_{i \in \mathcal{J} \cup \mathcal{K}} S_i^{\beta_i}
\end{multline}
is valid without conditions~\eqref{DS:bound-cond}. This concludes the
proof.
\end{proof}
\subsection{Positivity of constant}
\label{sec:org6e5f7bf}
\begin{lemma} \label{lemma:single-point-event}
For all \(\bm{t} \in \mathbb{R}^n\) holds
\begin{equation*}
\int_{\mathbb{R}^d} e^{\bm{1}^\top \bm{x}} \,
\mathbb{P} \left\{
\bm{Y}_{\bm{\nu}, \mathbb{V}} ( \bm{t} )
+\bm{Z} ( \bm{t} )
-\bm{d}_{\bm{\nu}, \mathbb{W}} ( \bm{t} ) > \bm{x}
\right\}
\mathop{d \bm{x}}
=
\exp \left(
-\sum_{i \in \mathcal{J} \cup \mathcal{K}} | t_i |^{\nu_i} \bm{1}^\top W_i \, \bm{1}
+\bm{1}^\top R_{\bm{Z}} ( \bm{t}, \bm{t} ) \, \bm{1}
\right).
\end{equation*}
\end{lemma}

\begin{proof}
Set
\begin{equation} \label{tld}
\widetilde{\bm{Y}}_{\bm{\nu}, \mathbb{V}} ( \bm{t} )
= \sum_{i \in \mathcal{I} \cup \mathcal{J}} \bm{Y}_{\nu_i, V_i} ( t_i ),
\qquad
\widetilde{\bm{d}}_{\bm{\nu}, \mathbb{W}} ( \bm{t} )
= \sum_{i \in \mathcal{I} \cup \mathcal{J}} S_{\nu_i, V_i} ( t_i ) \bm{1}
+\sum_{j \in \mathcal{J} \cup \mathcal{K}} | t_j |^{\nu_i} W_i \bm{1}
\end{equation}
and note that
\[
\bm{Y}_{\bm{\nu}, \mathbb{V}} ( \bm{t} )
-\bm{d}_{\bm{\nu}, \mathbb{W}} ( \bm{t} )
=
\widetilde{\bm{Y}}_{\bm{\nu}, \mathbb{V}} ( \bm{t} )
-\widetilde{\bm{d}}_{\bm{\nu}, \mathbb{W}} ( \bm{t} ).
\]
The claim follows from
\begin{align*}
&
\int_{\mathbb{R}^d} e^{\bm{1}^\top \bm{x}}
\mathbb{P} \left\{
  \bm{Y}_{\bm{\nu}, \mathbb{V}} ( \bm{t} )
  +\bm{Z} ( \bm{t} )
  -\bm{d}_{\bm{\nu}, \mathbb{W}} ( \bm{t} ) > \bm{x}
\right\}
\mathop{d \bm{x}}
\\[7pt]
& \hspace{50pt} =
\mathbb{E} \left\{
\int_{\mathbb{R}^d} e^{\bm{1}^\top \bm{x}}
\mathbb{1} \left\{
\bm{Y}_{\bm{\nu}, \mathbb{V}} ( \bm{t} )
+\bm{Z} ( \bm{t} )
-\bm{d}_{\bm{\nu}, \mathbb{W}} ( \bm{t} ) > \bm{x}
\right\}
\mathop{d \bm{x}}
\right\}
\\[7pt]
& \hspace{50pt} =
\exp \left(
-\bm{1}^\top \widetilde{\bm{d}}_{\bm{\nu}, \mathbb{W}} ( \bm{t} )
\right)
\mathbb{E} \left\{
\exp \left(
\bm{1}^\top \widetilde{\bm{Y}}_{\bm{\nu}, \mathbb{V}} ( \bm{t} )
\right)
\right\}
\,
\mathbb{E} \left\{
\exp \left( \bm{1}^\top \bm{Z} ( \bm{t} ) \right)
\right\}
\\[7pt]
& \hspace{50pt} =
\exp \left(
-\sum_{i \in \mathcal{J} \cup \mathcal{K}} | t_i |^{\nu_i} \bm{1}^\top W_i \, \bm{1}
\right)
\exp \left(
-\sum_{i \in \mathcal{I} \cup \mathcal{J}} \bm{1}^\top S_{\nu_i, V_i} ( t_i ) \, \bm{1}
\right)
\\[7pt]
& \hspace{50pt}
\times
\exp \left(
\frac{1}{2}
\sum_{i \in \mathcal{I} \cup \mathcal{J}} \bm{1}^\top R_{\nu_i, V_i} ( t_i, t_i ) \, \bm{1}
\right)
\exp \left(
\bm{1}^\top R_{\bm{Z}} ( \bm{t}, \bm{t} ) \, \bm{1}
\right)
\end{align*}
along with \(R_{\alpha, V} ( t, t ) = 2 S_{\alpha, V} ( t )\).
\end{proof}

\begin{lemma} \label{lemma:double-point-event}
For all \(\bm{t}, \ \bm{s} \in \mathbb{R}^n\) holds
\begin{multline*}
  \int_{\mathbb{R}^d} e^{\bm{1}^\top \bm{x}} \,
  \mathbb{P} \left\{
    \begin{aligned}
      \bm{Y}_{\bm{\nu}, \mathbb{V}} ( \bm{t} )
      +\bm{Z} ( \bm{t} )
      -\bm{d}_{\bm{\nu}, \mathbb{W}} ( \bm{t} ) > \bm{x} \\
      \bm{Y}_{\bm{\nu}, \mathbb{V}} ( \bm{s} )
      +\bm{Z} ( \bm{s} )
      -\bm{d}_{\bm{\nu}, \mathbb{W}} ( \bm{s} ) > \bm{x}
    \end{aligned}
  \right\}
  \mathop{d \bm{x}}
  \leq \\[7pt] \leq
  \exp \left(
    -\frac{1}{2} \sum_{i \in \mathcal{J} \cup \mathcal{K}}
  \Big[ | t_i |^{\nu_i} + | s_i |^{\nu_i} \Big] \bm{1}^\top W_i \, \bm{1}
  \right)
  \exp \left(
    -\frac{1}{4} \sum_{i \in \mathcal{I} \cup \mathcal{J}}
    \bm{1}^\top S_{\nu_i, V_i} ( t_i - s_i ) \bm{1}
  \right)
  \\[7pt]
  \times
  \exp \left(
    \frac{1}{2} \, \bm{1}^\top
    R_{\bm{Z}} ( \bm{t}, \bm{t} )
    +R_{\bm{Z}} ( \bm{t}, \bm{s} )
    +R_{\bm{Z}} ( \bm{s}, \bm{t} )
    +R_{\bm{Z}} ( \bm{s}, \bm{s} )
    \bm{1}
  \right).
\end{multline*}
\end{lemma}

\begin{proof}
Set \(\widetilde{\bm{Y}}_{\bm{\nu}, \mathbb{V}}\) and
\(\widetilde{\bm{d}}_{\bm{\nu}, \mathbb{W}}\) as in~\eqref{tld}. By
\begin{equation*}
  \mathbb{P} \left\{
    \begin{aligned}
      \bm{Y}_{\bm{\nu}, \mathbb{V}} ( \bm{t} )
      +\bm{Z} ( \bm{t} )
      -\bm{d}_{\bm{\nu}, \mathbb{W}} ( \bm{t} ) > \bm{x} \\
      \bm{Y}_{\bm{\nu}, \mathbb{V}} ( \bm{s} )
      +\bm{Z} ( \bm{s} )
      -\bm{d}_{\bm{\nu}, \mathbb{W}} ( \bm{s} ) > \bm{x}
    \end{aligned}
  \right\}
  \leq
  \mathbb{P} \left\{
  \begin{multlined}
  \bm{Y}_{\bm{\nu}, \mathbb{V}} ( \bm{t} )
  -\bm{d}_{\bm{\nu}, \mathbb{W}} ( \bm{t} )
  +\bm{Z} ( \bm{t} )
  \\
  +\bm{Y}_{\bm{\nu}, \mathbb{V}} ( \bm{s} )
  -\bm{d}_{\bm{\nu}, \mathbb{W}} ( \bm{s} )
  +\bm{Z} ( \bm{s} )
  > 2 \bm{x}
  \end{multlined}
  \right\},
\end{equation*}
we have
\begin{multline*}
  \int_{\mathbb{R}^d} e^{\bm{1}^\top \bm{x}} \,
  \mathbb{P} \left\{
    \begin{aligned}
      \bm{Y}_{\bm{\nu}, \mathbb{V}} ( \bm{t} )
      +\bm{Z} ( \bm{t} )
      -\bm{d}_{\bm{\nu}, \mathbb{W}} ( \bm{t} ) > \bm{x} \\
      \bm{Y}_{\bm{\nu}, \mathbb{V}} ( \bm{s} )
      +\bm{Z} ( \bm{s} )
      -\bm{d}_{\bm{\nu}, \mathbb{W}} ( \bm{s} ) > \bm{x}
    \end{aligned}
  \right\}
  \mathop{d \bm{x}}
  \leq
  \exp \left(
    -\frac{1}{2} \, \bm{1}^\top \Big[
      \widetilde{\bm{d}}_{\bm{\nu}, \mathbb{W}} ( \bm{t} )
      +\widetilde{\bm{d}}_{\bm{\nu}, \mathbb{W}} ( \bm{s} )
    \Big]
  \right)
  \\[7pt]
  \times
  \mathbb{E} \left\{
    \exp \left(
      \frac{1}{2} \, \bm{1}^\top \Big[
        \widetilde{\bm{Y}}_{\bm{\nu}, \mathbb{V}} ( \bm{t} )
        +\widetilde{\bm{Y}}_{\bm{\nu}, \mathbb{V}} ( \bm{s} )
      \Big]
    \right)
  \right\}
  \,
  \mathbb{E} \left\{
    \exp \left(
      \frac{1}{2} \, \bm{1}^\top
      \left( \bm{Z} ( \bm{t} ) + \bm{Z} ( \bm{s} ) \right) \,
      \bm{\mathcal{N}}
    \right)
  \right\}.
\end{multline*}
First, compute the last expectation:
\begin{equation*}
\mathbb{E} \left\{ \exp \left(
\frac{1}{2} \,
\bm{1}^\top \left( \bm{Z} ( \bm{t} ) + \bm{Z} ( \bm{s} ) \right)
\right) \right\}
=
\exp \left(
\frac{1}{2} \,
\bm{1}^\top
\Big[
R_{\bm{Z}} ( \bm{t}, \bm{t} )
+R_{\bm{Z}} ( \bm{t}, \bm{s} )
+R_{\bm{Z}} ( \bm{s}, \bm{t} )
+R_{\bm{Z}} ( \bm{s}, \bm{s} )
\Big]
\bm{1}
\right).
\end{equation*}
Since
\begin{align*}
&
\mathbb{E} \left\{
  \exp \left(
    \frac{1}{2} \bm{1}^\top \Big[
      \widetilde{\bm{Y}}_{\bm{\nu}, \mathbb{V}} ( \bm{t} )
      +\widetilde{\bm{Y}}_{\bm{\nu}, \mathbb{V}} ( \bm{s} )
    \Big]
  \right)
\right\}
\\[7pt]
& \hspace{30pt}
=
\exp \left(
  \frac{1}{8}
  \bm{1}^\top
  \mathbb{E} \left\{
    \Big[
      \widetilde{\bm{Y}}_{\bm{\nu}, \mathbb{V}} ( \bm{t} )
      +\widetilde{\bm{Y}}_{\bm{\nu}, \mathbb{V}} ( \bm{s} )
    \Big]
    \Big[
      \widetilde{\bm{Y}}_{\bm{\nu}, \mathbb{V}} ( \bm{t} )
      +\widetilde{\bm{Y}}_{\bm{\nu}, \mathbb{V}} ( \bm{s} )
    \Big]^\top
  \right\}
  \bm{1}
\right)
\\[7pt]
& \hspace{30pt} =
\exp \left(
  \frac{1}{8}
  \sum_{i \in \mathcal{I} \cup \mathcal{J}}
  \bm{1}^\top
  \Big[
    R_{\nu_i, V_i} ( t_i, t_i )
    +R_{\nu_i, V_i} ( t_i, s_i )
    +R_{\nu_i, V_i} ( s_i, t_i )
    +R_{\nu_i, V_i} ( s_i, s_i )
  \Big]
  \bm{1}
\right)
\\[7pt]
& \hspace{30pt} =
\exp \left(
  \frac{1}{4} \sum_{i \in \mathcal{I} \cup \mathcal{J}}
  \bm{1}^\top
  \Big[
    S_{\nu_i, V_i} ( t_i )
    +R_{\nu_i, V_i} ( t_i, s_i )
    +S_{\nu_i, V_i} ( s_i )
  \Big]
  \bm{1}
\right)
\end{align*}
and
\[
\bm{1}^\top \Big[
  \widetilde{\bm{d}}_{\bm{\nu}, \mathbb{W}} ( \bm{t} )
  +\widetilde{\bm{d}}_{\bm{\nu}, \mathbb{W}} ( \bm{s} )
\Big]
=
\sum_{i \in \mathcal{I} \cup \mathcal{J}} \bm{1}^\top
\Big[ S_{\nu_i, V_i} ( t_i ) + S_{\nu_i, V_i} ( s_i ) \Big] \bm{1}
+\sum_{i \in \mathcal{J} \cup \mathcal{K}}
\Big[ | t_i |^{\nu_i} + | t_i |^{\nu_i} \Big] \bm{1}^\top W_i \, \bm{1},
\]
we have
\begin{align*}
  &
  \exp \left(
    -\frac{1}{2} \bm{1}^\top \Big[
      \widetilde{\bm{d}}_{\bm{\nu}, \mathbb{W}} ( \bm{t} )
      +\widetilde{\bm{d}}_{\bm{\nu}, \mathbb{W}} ( \bm{s} )
    \Big]
  \right)
  \mathbb{E} \left\{
    \exp \left(
      \frac{1}{2} \bm{1}^\top \Big[
        \widetilde{\bm{Y}}_{\bm{\nu}, \mathbb{V}} ( \bm{t} )
        +\widetilde{\bm{Y}}_{\bm{\nu}, \mathbb{V}} ( \bm{s} )
      \Big]
    \right)
  \right\}
  \\[7pt]
  & \hspace{30pt} =
  \exp \left(
    -\frac{1}{2} \sum_{i \in \mathcal{J} \cup \mathcal{K}}
    \Big[ | t_i |^{\nu_i} + | s_i |^{\nu_i} \Big] \bm{1}^\top W_i \bm{1}
  \right)
  \\[7pt]
  & \hspace{40pt} \times
  \exp \left(
    -\frac{1}{4} \sum_{i \in \mathcal{I} \cup \mathcal{J}} \bm{1}^\top \Big[
      S_{\nu_i, V_i} ( t_i ) + S_{\nu_i, V_i} ( s_i ) - R_{\nu_i, V_i} ( t_i, s_i )
    \Big] \bm{1}
  \right)
\end{align*}
and the claim follows by
\(S_{\alpha, V} ( t ) + S_{\alpha, V} ( s ) - R_{\alpha, V} ( t, s ) = S_{\alpha, V} ( t - s )\).
\end{proof}

\begin{lemma}[Lower bound for the constant] \label{lemma:const-lower-bound}
For \(\bm{S}_{\mathcal{J} \cup \mathcal{K}} / 2 < \bm{\delta}_{\mathcal{J} \cup \mathcal{K}} < \bm{S}_{\mathcal{J} \cup \mathcal{K}}\) and
all \(\bm{\delta}_{\mathcal{I}} > \bm{0}\), holds
\begin{equation} \label{eq:19}
  H_{\bm{\nu}, \mathbb{V}, \mathbb{W}} \left( [ \bm{0}, \bm{S} ] \right)
  \prod_{i \in \mathcal{I}} \frac{1}{S_i}
  \geq
  \prod_{i \in \mathcal{I}} \frac{1}{\delta_i} \left[
    1 - \sum_{\mathcal{I}_0 \subsetneq \mathcal{I}} \prod_{i \in \mathcal{I} \setminus \mathcal{I}_0} \frac{A_i}{\delta_i}
  \right]
\end{equation}
with
\[
  A_i
  = 2 \left( \frac{4}{\bm{1}^\top V_i \bm{1}} \right)^{1/\nu_i}
  \Gamma \left( \frac{1}{\nu_i} + 1 \right)
\]
if \(\bm{1}^\top V_i \, \bm{1} > 0\) for all \(i \in \mathcal{I}\).
\end{lemma}

\begin{proof}
For any \(\bm{\delta} > \bm{0}\) we have
\begin{multline*}
  H_{\bm{\nu}, \mathbb{V}, \mathbb{W}}
  \left( [ \bm{0}, \bm{S} ] \right)
  \geq
  \int_{\mathbb{R}^{d}} e^{\bm{1}^\top \bm{x}} \,
  \mathbb{P} \left\{
  \exists \, \bm{t} \in [ \bm{0}, \bm{S} ] \cap \bm{\delta} \mathbb{Z}_{+}^n \colon
  \bm{Y}_{\bm{\nu}, \mathbb{V}} ( \bm{t} )
  +\bm{Z} ( \bm{t} )
  -\bm{d}_{\bm{\nu}, \mathbb{W}} ( \bm{t} ) > \bm{x}
  \right\}
  \mathop{d \bm{x}}
  \\[7pt] \geq
  \sum_{\bm{k} \leq \bm{N}_{\bm{\delta}}}
  \int_{\mathbb{R}^{d}} e^{\bm{1}^\top \bm{x}} \,
  \mathbb{P} \left\{
  \bm{Y}_{\bm{\nu}, \mathbb{V}} ( \bm{\delta} \bm{k} )
  +\bm{Z} ( \bm{\delta} \bm{k} )
  -\bm{d}_{\bm{\nu}, \mathbb{W}} ( \bm{\delta} \bm{k} ) > \bm{x}
  \right\}
  \mathop{d \bm{x}}
  \\[7pt]
  -\sum_{
    \substack{
      \bm{k}, \, \bm{l} \, \leq \, \bm{N}_{\bm{\delta}}, \\[3pt]
      \bm{k} \, \neq \, \bm{l}
    }
  }
  \int_{\mathbb{R}^{d}} e^{\bm{1}^\top \bm{x}} \,
  \mathbb{P} \left\{
    \begin{aligned}
      \bm{Y}_{\bm{\nu}, \mathbb{V}} ( \bm{\delta} \bm{k} )
      +\bm{Z} ( \bm{\delta} \bm{k} )
      -\bm{d}_{\bm{\nu}, \mathbb{W}} ( \bm{\delta} \bm{k} ) > \bm{x} \\
      \bm{Y}_{\bm{\nu}, \mathbb{V}} ( \bm{\delta} \bm{l} )
      +\bm{Z} ( \bm{\delta} \bm{l} )
      -\bm{d}_{\bm{\nu}, \mathbb{W}} ( \bm{\delta} \bm{l} ) > \bm{x}
    \end{aligned}
  \right\}
  \mathop{d \bm{x}},
\end{multline*}
where
\(\bm{N}_{\bm{\delta}} = \floor*{\bm{S} / \bm{\delta}}\).
Take
\(\bm{S}_{\mathcal{J} \cup \mathcal{K}} / 2 < \bm{\delta}_{\mathcal{J} \cup \mathcal{K}} < \bm{S}_{\mathcal{J} \cup \mathcal{K}}\).
Then \(\bm{N}_{\bm{\delta}, \mathcal{J} \cup \mathcal{K}} = \bm{1}_{\mathcal{J} \cup \mathcal{K}}\), and therefore
\(\bm{k}_{\mathcal{I} \cup \mathcal{K}} = \bm{0}_{\mathcal{J} \cup \mathcal{K}}\).
It follows by definition of \(\bm{Z}\) that
\(\bm{Z} ( \bm{\delta} \bm{k} ) = \bm{0}\).
Apply Lemma~\ref{lemma:single-point-event} and
Lemma~\ref{lemma:double-point-event}.
\begin{equation} \label{eq:15}
\begin{aligned}
  H_{\bm{\nu}, \mathbb{V}, \mathbb{W}} \left( [ \bm{0}, \bm{S} ] \right)
  & \geq
  \sum_{\bm{k}_{\mathcal{I}} \leq \bm{N}_{\bm{\delta}, \mathcal{I}}} 1
  -\sum_{
    \substack{
      \bm{k}_{\mathcal{I}}, \, \bm{l}_{\mathcal{I}} \, \leq \, \bm{N}_{\bm{\delta}, \mathcal{I}}, \\[3pt]
      \bm{k}_{\mathcal{I}} \, \neq \, \bm{l}_{\mathcal{I}}
    }
  }
  \exp \left(
    -\frac{1}{4} \sum_{i \in \mathcal{I}} \bm{1}^\top S_{\nu_i, V_i} ( \delta_i k_i - \delta_i l_i ) \bm{1}
  \right)
  \\
  & = \prod_{i \in \mathcal{I}} \frac{S_i}{\delta_i} - \mathbb{\Sigma}_2 ( \bm{\delta} ).
\end{aligned}
\end{equation}
To bound the double sum, use
\(S_{\alpha, V} = | t |^{\alpha} \left( V 1_{t \geq 0} + V^\top 1_{t < 0} \right)\)
reindex the sum as follows: let \(\mathcal{I}_0 ( \bm{k} )\) denote those indices, in which
\(k_i = l_i\). This set cannot be equal to the entire \(\mathcal{I}\), becasuse in this case
\(\bm{k}_{\mathcal{I}} = \bm{l}_{\mathcal{I}}\) and such pairs are excluded from the sum, but it
is empty if \(\min_{i \in \mathcal{I}} | k_i - l_i | \geq 1\).
\begin{align*}
  \mathbb{\Sigma}_2 ( \bm{\delta} )
  & =
  \sum_{\mathcal{I}_0 \subsetneq \mathcal{I}}
  \prod_{i \in \mathcal{I} \setminus \mathcal{I}_0}
  \sum_{k = 1}^{N_{\bm{\delta}, i}}
  \sum_{l = k + 1}^{N_{\bm{\delta}, i}}
  2 \exp \left(
    -\frac{1}{4} \sum_{i \in \mathcal{I}} \delta_i^{\nu_i} | k_i - l_i |^{\nu_i} \bm{1}^\top V_i \bm{1}
  \right)
  \prod_{i \in \mathcal{I}_0} \sum_{k = 1}^{N_{\bm{\delta}, i}} 1
  \\[7pt]
  & \leq
    \sum_{\mathcal{I}_0 \subsetneq \mathcal{I}}
    \prod_{i \in \mathcal{I} \setminus \mathcal{I}_0}
    \frac{2 S_i}{\delta_i^2}
    \left( \frac{4}{\bm{1}^\top V_i \bm{1}} \right)^{1/\nu_i}
    \int_0^{\infty} e^{-x^{\nu_i}} \mathop{dx}
  \prod_{i \in \mathcal{I}_0} \sum_{k = 1}^{N_{\bm{\delta}, i}} 1
  \\[7pt]
  & =
    \sum_{\mathcal{I}_0 \subsetneq \mathcal{I}}
    \prod_{i \in \mathcal{I} \setminus \mathcal{I}_0}
    \frac{2 S_i}{\delta_i^2}
    \left( \frac{4}{\bm{1}^\top V_i \bm{1}} \right)^{1/\nu_i}
  \Gamma \left( \frac{1}{\nu_i} + 1 \right)
  \prod_{i \in \mathcal{I}_0} \frac{S_i}{\delta_i}
  \\[7pt]
  & =
  \prod_{j \in \mathcal{I}} S_i
  \sum_{\mathcal{I}_0 \subsetneq \mathcal{I}}
  \prod_{i \in \mathcal{I} \setminus \mathcal{I}_0}
  \frac{2}{\delta_i^2}
  \left( \frac{4}{\bm{1}^\top V_i \bm{1}} \right)^{1/\nu_i}
  \Gamma \left( \frac{1}{\nu_i} + 1 \right)
  \prod_{i \in \mathcal{I}_0} \frac{1}{\delta_i}
\end{align*}
Combining the above together, we obtain
\begin{equation} \label{eq:18}
  H_{\bm{\nu}, \mathbb{V}, \mathbb{W}} \left( [ \bm{0}, \bm{S} ] \right)
  \prod_{i \in \mathcal{I}} \frac{1}{S_i}
  \geq \prod_{i \in \mathcal{I}} \frac{1}{\delta_i}
  -\sum_{\mathcal{I}_0 \subsetneq \mathcal{I}} \prod_{i \in \mathcal{I} \setminus \mathcal{I}_0} \frac{A_i}{\delta_i^2} \prod_{i \in \mathcal{I}_0} \frac{1}{\delta_i}
\end{equation}
where
\[
  A_i =
  2 \left( \frac{4}{\bm{1}^\top V_i \bm{1}} \right)^{1/\nu_i}
  \Gamma \left( \frac{1}{\nu_i} + 1 \right).
\]
This concludes the proof.
\end{proof}
\subsection{Main theorem proof}
\label{sec:org6546d83}
\begin{proof}
By Lemma~\ref{lemma:gen-var-exp}, the generalized variance satisfies
\begin{equation*}
\sigma_{\bm{b}}^{-2} ( \bm{\tau} ) - \sigma_{\bm{b}}^{-2} ( \bm{0} )
=
\sum_{i, j = 1}^n \Xi_{i, j} \, \tau_i^{\beta_i/2} \, \tau_j^{\beta_j/2}
+o \left( \sum_{i = 1}^n \tau_i^{\beta_i} \right),
\end{equation*}
with \(\Xi\) defined in~\eqref{XI-alternative}.
By~\eqref{XI-elliptical} and~\ref{A2.4}, there exists a constant
\(c > 0\) such that
\begin{equation*}
\sigma_{\bm{b}}^{-2} ( \bm{\tau} ) - \sigma_{\bm{b}}^{-2} ( \bm{0} )
\geq c \sum_{i = 1}^n \tau_i^{\beta_i},
\end{equation*}
and therefore we obtain
\begin{equation*}
\mathbb{P} \left\{
\exists \, \bm{t} \in [ \bm{0}, \bm{T} ] \colon
\bm{X} ( \bm{t} ) > u \bm{b}
\right\}
\sim
\mathbb{P} \left\{
\exists \, \bm{t} \in [ \bm{0}, \bm{\delta}_u ] \colon
\bm{X} ( \bm{t} ) > u \bm{b}
\right\}
\end{equation*}
with \(\bm{\delta}_u = u^{-2/\bm{\beta}} \ln^{2/\bm{\beta}} u\) by using the Piterbarg
inequality~\ref{Piterbarg-inequality}.
\subsubsection{Upper bound}
\label{sec:org316ba81}
Take \(\bm{\Lambda} > \bm{0}\) and use the log-layer bound
(Lemma~\ref{lemma:log-layer-bound}) to obtain
\begin{multline} \label{main-theorem:upper}
\mathbb{P} \left\{
\exists \, \bm{t} \in [ \bm{0}, \bm{\delta}_u ] \colon
\bm{X} ( \bm{t} ) > u \bm{b}
\right\}
\leq
\mathbb{P} \left\{
\exists \, \bm{t} \in [ \bm{0}, u^{-2/\bm{\beta}} \bm{\Lambda} ] \colon
\bm{X} ( \bm{t} ) > u \bm{b}
\right\}
\\[7pt]
+C \, \mathbb{P} \left\{ \bm{X} ( \bm{0} ) > u \bm{b} \right\}
\exp \left( -\frac{1}{8} \sum_{i = 1}^n \xi_i \, \Lambda_i^{\beta_i} \right).
\end{multline}
In order to bound the first term from above, let us split the cube
\(u^{-2 / \bm{\beta}} [ \bm{0}, \bm{\Lambda} ]\) into parts of ``size''
\(u^{- 2 / \bm{\nu}} [ \bm{0}, \bm{S} ]\) with
\(\bm{S} > \bm{0}\) such that \(\bm{S}_{\mathcal{I}^c} = \bm{\Lambda}_{\mathcal{I}^c}\),
and note that
\begin{equation} \label{def:single-sum}
  \mathbb{P} \left\{
    \exists \, \bm{t} \in u^{-2 / \bm{\nu}} [ \bm{0}, \bm{\Lambda} ] \colon
    \bm{X} ( \bm{t} ) > u \bm{b}
  \right\}
  \leq
  \sum_{\bm{k} \leq \bm{N}_u}
  \mathbb{P} \left\{
    \exists \, \bm{t} \in u^{-2 / \bm{\nu}} \bm{S} [ \bm{k}, \bm{k} + 1 ] \colon
    \bm{X} ( \bm{t} ) > u \bm{b}
  \right\}
  \eqqcolon
  \mathbb{\Sigma}_1 ( u, \bm{\Lambda}, \bm{S} ),
\end{equation}
where
\(\bm{N}_u = \ceil*{u^{2 / \bm{\nu} - 2 / \bm{\beta}} \bm{\Lambda} / \bm{S}}\).
By Lemma~\ref{lemma:limiting-process} and the uniform local Pickands
Lemma~\ref{lemma:local-Pickands}, we have
\begin{align*}
  &
  \mathbb{P} \left\{
    \exists \, \bm{t} \in u^{-2 / \bm{\nu}} \bm{S}
    \left[ \bm{k}, \bm{k} + \bm{1} \right] \colon
    \bm{X} ( \bm{t} ) > u \bm{b}
  \right\}
  \\[7pt]
  & \hspace{30pt} \sim
    H_{\bm{\nu}, \mathbb{V}_{\bm{w}}, \mathbb{W}_{\bm{w}}}
    \left( [\bm{0}, \bm{S}] \right)
    \mathbb{P} \left\{
    \bm{X}_{u, \bm{k}} ( \bm{0} )
    > u \bm{b}
    \right\}
  \\[7pt]
  & \hspace{30pt} \sim
    H_{\bm{\nu}, \mathbb{V}_{\bm{w}}, \mathbb{W}_{\bm{w}}}
    \left( [\bm{0}, \bm{S}] \right)
    \mathbb{P} \left\{
    \bm{X} ( \bm{0} ) > u \bm{b}
    \right\}
    \exp \left(
    -\frac{u^2}{2} \,
      \bm{b}^\top
      \Big[
        \Sigma_{u, \bm{k}}^{-1} - \Sigma^{-1}
      \Big] \, \bm{b}
    \right),
\end{align*}
and therefore by~\eqref{exp-prefactor-1}
\begin{multline*}
\mathbb{P} \left\{
\exists \, \bm{t} \in u^{-2 / \bm{\nu}} \bm{S}
\left[ \bm{k}, \bm{k} + \bm{1} \right] \colon
\bm{X} ( \bm{t} ) > u \bm{b}
\right\}
\sim
H_{\bm{\nu}, \mathbb{V}_{\bm{w}}, \mathbb{W}_{\bm{w}}}
\left( [\bm{0}, \bm{S}] \right)
\mathbb{P} \left\{
\bm{X} ( \bm{0} ) > u \bm{b}
\right\}
\\[7pt]
\times
\exp \left(
-\frac{1}{2} \,
\sum_{i, j \in \mathcal{I}}
\Xi_{i, j}
\left( S_i \, k_i \, u^{-\zeta_i} \right)^{\beta_i/2}
\left( S_j \, k_j \, u^{-\zeta_j} \right)^{\beta_j/2}
\right)
\end{multline*}
with \(\Xi\) defined in~\eqref{XI-alternative} and
\(\bm{\zeta} = 2 / \bm{\beta} - 2 / \bm{\nu}\).
Using the following formula
\begin{equation*}
\sum_{\bm{k} \leq \bm{N}_u}
\exp \left(
-\frac{1}{2} \,
\sum_{i, j \in \mathcal{I}}
\Xi_{i, j}
\left( S_i \, k_i \, u^{-\zeta_i} \right)^{\beta_i/2}
\left( S_j \, k_j \, u^{-\zeta_j} \right)^{\beta_j/2}
\right)
\sim
\prod_{i \in \mathcal{I}} \frac{u^{\zeta_i}}{S_i} \, G ( \bm{\beta}, \Xi, \bm{\Lambda} ),
\end{equation*}
where
\begin{equation*}
G ( \bm{\beta}, \Xi, \bm{\Lambda} ) \coloneqq
\int_{\bm{0}_{\mathcal{I}}}^{\bm{\Lambda}_{\mathcal{I}}}
\exp \left( -\frac{1}{2} \sum_{i, j \in \mathcal{I}} \Xi_{i, j} \,
t_i^{\beta_i/2} \, t_j^{\beta_j} \right)
\mathop{d \bm{t}},
\quad
G ( \bm{\beta}, \Xi ) \coloneqq \lim_{\bm{\Lambda} \to \infty} G ( \bm{\beta}, \Xi ) < \infty,
\end{equation*}
which may be proven by Riemann sum argument, we obtain the following estimate
for the single sum:
\begin{equation*}
  \frac{
    \mathbb{\Sigma}_1 ( u, \bm{\Lambda}, \bm{S} )
  }{
    \mathbb{P} \left\{ \bm{X} ( \bm{0} ) > u \bm{b} \right\}
  }
  \sim
  H_{\bm{\nu}, \mathbb{V}_{\bm{w}}, \mathbb{V}_{\bm{w}}} \left( [ \bm{0}, \bm{S} ] \right)
  \prod_{i \in \mathcal{I}}
  \frac{u^{\zeta_i}}{S_i} \, G ( \bm{\beta}, \Xi, \bm{\Lambda} ).
\end{equation*}
Let us prove that the following limit
\begin{equation} \label{constant-finiteness}
  \mathcal{H}_{\bm{\nu}, \mathbb{V}_{\bm{w}}, \mathbb{W}_{\bm{w}}}
  =
  \lim_{\bm{\Lambda}_{\mathcal{I}^c} \to \infty}
  \lim_{\bm{S}_{\mathcal{I}} \to \infty}
  H_{\bm{\nu}, \mathbb{V}_{\bm{w}}, \mathbb{W}_{\bm{w}}} \left(
  [ \bm{0}, \bm{S} ]
  \right)
  \prod_{i \in \mathcal{I}} \frac{1}{S_i},
\end{equation}
exists and is finite. Since
\(\bm{S} \mapsto
H_{\bm{\nu}, \mathbb{V}_{\bm{w}}, \mathbb{W}_{\bm{w}}}
\left( [ \bm{0}, \bm{S} ] \right)\)
is subadditive in each \(S_i\), the limit
\begin{equation*}
\lim_{\bm{S}_{\mathcal{I}} \to \infty}
H_{\bm{\nu}, \mathbb{V}_{\bm{w}}, \mathbb{W}_{\bm{w}}}
\left( [ \bm{0}, \bm{S} ] \right)
\prod_{i \in \mathcal{I}} \frac{1}{S_i}
\end{equation*}
exists and is an increasing function of each \(S_i > 0\), \(i \in \mathcal{I}^c\).
To prove that the limit~\eqref{constant-finiteness} exists and is finite,
it would suffice to show that
\(H_{\bm{\nu}, \mathbb{V}_{\bm{w}}, \mathbb{W}_{\bm{w}}}
\left( [ \bm{0}, \bm{S}] \right) / \prod_{i \in \mathcal{I}} S_i\)
is uniformly bounded in \(\bm{S}_{\mathcal{I}^c}\).
To this end, fix \(\bm{\Lambda}_0 < \bm{\Lambda}\) and note that for every \(\bm{S}\)
and \(\bm{S}_0\) such that \(\bm{S}_{\mathcal{I}^c} = \bm{\Lambda}_{\mathcal{I}^c}\) and
\(\bm{S}_{0, \mathcal{I}^c} = \bm{\Lambda}_{0, \mathcal{I}}\), holds
\begin{multline*}
\mathbb{\Sigma}_1' ( u, \bm{\Lambda}, \bm{S} )
\leq
\mathbb{P} \left\{
\exists \, \bm{t} \in u^{-2 / \bm{\nu}} [ \bm{0}, \bm{\Lambda}] \colon
\bm{X} ( \bm{t} ) > u \bm{b}
\right\}
\\[7pt]
\leq
\mathbb{P} \left\{
\exists \, \bm{t} \in u^{-2 / \bm{\nu}} [ \bm{0}, \bm{\Lambda}_0] \colon
\bm{X} ( \bm{t} ) > u \bm{b}
\right\}
+\mathbb{P} \left\{
\exists \, \bm{t} \in u^{-2 / \bm{\nu}} [ \bm{\Lambda}_0, \bm{\Lambda} ] \colon
\bm{X} ( \bm{t} ) > u \bm{b}
\right\}
\\[7pt]
\leq
\mathbb{\Sigma}_1 ( u, \bm{\Lambda}_0, \bm{S}_0 )
+C \,
\mathbb{P} \left\{ \bm{X} ( \bm{0} ) > u \bm{b} \right\}
\exp \left( -\frac{1}{8} \, \sum_{i = 1}^n \xi_i \, \Lambda_{0, i}^{\beta_i} \right),
\end{multline*}
where \(\mathbb{\Sigma}_1' ( u, \bm{\Lambda}, \bm{S} )\) is the same single
sum~\eqref{def:single-sum}, but with \(\bm{N}_u - 1\) instead of
\(\bm{N}_u\) in the limit of summation.
It is easy to see that a computation analogous to what we did above for
\(\mathbb{\Sigma}_1 ( u, \bm{\Lambda}, \bm{S} )\)
gives the same estimate for \(\mathbb{\Sigma}_1' ( u, \bm{\Lambda}, \bm{S} )\):
\begin{equation} \label{eq:S1-lower}
\frac{
\mathbb{\Sigma}_1' ( u, \bm{\Lambda}, \bm{S} )
}{
\mathbb{P} \left\{ \bm{X} ( \bm{0} ) > u \bm{b} \right\}
}
\sim
H_{\bm{\nu}, \mathbb{V}_{\bm{w}}, \mathbb{W}_{\bm{w}}}
\left( [ \bm{0}, \bm{S} ] \right)
\prod_{i \in \mathcal{I}}
\frac{u^{\zeta_i}}{S_i} \,
G ( \bm{\beta}, \bm{\Xi} ).
\end{equation}
Hence,
\begin{align*}
&
H_{\bm{\nu}, \mathbb{V}_{\bm{w}}, \mathbb{W}_{\bm{w}}}
\left( [ \bm{0}, \bm{S} ] \right)
\prod_{i \in \mathcal{I}} \frac{1}{S_i} \, G ( \bm{\beta}, \Xi, \bm{\Lambda} )
= \lim_{u \to \infty}
\frac{
\mathbb{\Sigma}_1' ( u, \bm{\Lambda}, \bm{S} )
}{
\mathbb{P} \left\{ \bm{X} ( \bm{0} ) > u \bm{b} \right\}
\prod_{i \in \mathcal{I}} u^{\zeta_i}
}
\\[7pt]
& \hspace{30pt} \leq
\lim_{u \to \infty}
\frac{
\mathbb{\Sigma}_1 ( u, \bm{\Lambda}_0, \bm{S}_0 )
}{
\mathbb{P} \left\{ \bm{X} ( \bm{0} ) > u \bm{b} \right\}
\prod_{i \in \mathcal{I}} u^{\zeta_i}
}
+C \exp \left(
-\frac{1}{8} \sum_{i = 1}^n \xi_i \, \Lambda_i^{\beta_i}
\right) 1_{\mathcal{I} = \varnothing}
\\[7pt]
& \hspace{30pt} \leq
H_{\bm{\nu}, \mathbb{V}_{\bm{w}}, \mathbb{W}_{\bm{w}}}
\left( [ \bm{0}, \bm{S}_0 ] \right)
\prod_{i \in \mathcal{I}} \frac{1}{S_{0, i}} \,
G ( \bm{\beta}, \Xi, \bm{\Lambda}_0 )
+C \exp \left(
-\frac{1}{8} \sum_{i = 1}^n \xi_i \, \Lambda_{0, i}^{\beta_i}
\right)
1_{\mathcal{I} = \varnothing}
\eqqcolon c,
\end{align*}
Since for \(\Lambda_i > 1\) holds
\(G ( \bm{\beta}, \bm{\Xi}, \bm{\Lambda} ) > G ( \bm{\beta}, \Xi, \bm{1} ) = c_1 > 0\),
we have the uniform bound
\begin{equation*}
  H_{\bm{\nu}, \mathbb{V}_{\bm{w}}, \mathbb{W}_{\bm{w}}}
  \left( [ \bm{0}, \bm{S} ] \right)
  \prod_{i \in \mathcal{I}} \frac{1}{S_i}
  \leq
  \frac{c}{c_1},
\end{equation*}
and thus the claim is proved.
Using~\eqref{main-theorem:upper}, letting \(\bm{S} \to \infty\) and then
\(\bm{\Lambda} \to \infty\), we obtain
\begin{equation*}
  \limsup_{u \to \infty}
  \frac{
    \mathbb{P} \left\{
      \exists \, \bm{t} \in u^{-2/\beta} [ \bm{0}, \bm{\Lambda} ] \colon
      \bm{X} ( \bm{t} ) > u \bm{b}
    \right\}
  }{
    \mathbb{P} \left\{ \bm{X} ( \bm{0} ) > u \bm{b} \right\}
    \prod_{i \in \mathcal{I}} u^{\zeta_i}
  }
\leq
\mathcal{H}_{\bm{\nu}, \mathbb{V}_{\bm{w}}, \mathbb{W}_{\bm{w}}} \,
G ( \bm{\beta}, \Xi ).
\end{equation*}
\subsubsection{Lower bound}
\label{sec:org1716004}
By Bonferroni inequality,
\begin{equation} \label{main-theorem:lower}
\mathbb{P} \left\{
\exists \, \bm{t} \in u^{-2 / \bm{\nu}} [ \bm{0}, \bm{\Lambda} ] \colon
\bm{X} ( \bm{t} ) > u \bm{b}
\right\}
\geq
\mathbb{\Sigma}_1' ( u, \bm{\Lambda}, \bm{S} ) - \mathbb{\Sigma}_2 ( u, \bm{\Lambda}, \bm{S} ),
\end{equation}
where \(\mathbb{\Sigma}_1' ( u, \bm{\Lambda}, \bm{S} )\) is the same single
sum~\eqref{def:single-sum}, but with \(\bm{N}_u - \bm{1}\) instead of
\(\bm{N}_u\) in the limit of summation, and
\begin{equation*}
  \mathbb{\Sigma}_2 ( u, \bm{\Lambda}, \bm{S} )
  =
  \sum_{\bm{k} \, \neq \, \bm{l} \, \leq \, \bm{N}_u}
  P_{\bm{b}} ( \bm{S} \bm{k}, \bm{S} \bm{l}, \bm{S} )
\end{equation*}
and it only remains to bound \(\mathbb{\Sigma}_2 ( u, \bm{\Lambda}, \bm{S} )\). Note that
if \(\bm{k} \leq \bm{N}_u\), then in particular \(\bm{k}_{\mathcal{I}^c} = \bm{0}\).

First, rewrite it as \(\mathbb{\Sigma}_2 = \mathbb{\Sigma}_{2, 1} + \mathbb{\Sigma}_{2, 2}\),
where the double sums \(\mathbb{\Sigma}_{2, 1}\) and \(\mathbb{\Sigma}_{2, 2}\) are
taken over
\begin{equation*}
\left\{
    \substack{
      \bm{k}_{\mathcal{I}} \, \neq \, \bm{l}_{\mathcal{I}} \, \leq \, \bm{N}_{u, \mathcal{I}}, \\[3pt]
      \min_{i \in \mathcal{I}} | k_i - l_i | \leq 1 \\[3pt]
      \bm{k}_{\mathcal{I}^c} = \bm{l}_{\mathcal{I}^c} = \bm{0}
    }
\right\}
\quad \text{and} \quad
\left\{
    \substack{
      \bm{k}_{\mathcal{I}} \, \neq \, \bm{l}_{\mathcal{I}} \, \leq \, \bm{N}_{u, \mathcal{I}}, \\[3pt]
      \min_{i \in \mathcal{I}} | k_i - l_i | > 1 \\[3pt]
      \bm{k}_{\mathcal{I}^c} = \bm{l}_{\mathcal{I}^c} = \bm{0}
    }
\right\}
\end{equation*}
correspondingly.
Each term of the sum \(\mathbb{\Sigma}_{2, 2} ( u, \bm{\Lambda}, \bm{S} )\) satisfies
conditions of Lemma~\ref{lemma:double-sum}, and therefore
\begin{multline*}
\frac{
\mathbb{\Sigma}_{2, 2} ( u, \bm{\Lambda}, \bm{S} )
}{
H_u ( \bm{S} \bm{k} ) \,
\mathbb{P} \left\{ \bm{X} ( \bm{0} ) > u \bm{b} \right\}
}
\\[7pt]
\leq
C \, e^{c \sum_{j \in \mathcal{I}^c} \Lambda_j^{\beta_j}} \,
\sum_{\bm{k}_{\mathcal{I}} \leq \bm{N}_{u, \mathcal{I}}} \
\sum_{ \bm{k}_{\mathcal{I}} \, + \, \bm{1}_{\mathcal{I}} \, \leq \, \bm{l}_{\mathcal{I}} \, \leq \, \bm{N}_{u, \mathcal{I}} } \
\prod_{i \in \mathcal{I}}
( l_i - k_i - 1 )^2
\exp \left(
-\frac{\varkappa_i}{64} \, S_i^{\alpha_i} ( l_i - k_i - 1 )^{\alpha_i}
\right)
\\[7pt]
\leq
C
e^{c \sum_{j \in \mathcal{I}^c} \Lambda_j^{\beta_j}} \,
\prod_{i \in \mathcal{I}} \sum_{l = 1}^{\infty} \ l^2
\exp \left(
-\frac{\varkappa}{64} \, S_i^{\alpha_i} \, l^{\alpha_i}
\right).
\end{multline*}
It remains to note that
\begin{equation} \label{eq:16}
\sum_{k = 0}^{N_{u, j}}
\exp \Bigg(
-\frac{\xi_i}{4} S_j^{\beta_j} k^{\beta_j} u^{-\beta_j \zeta_j}
\Bigg)
\sim
\frac{4^{1/\beta_i}u^{\zeta_i}}{S_i \, \xi_i^{1/\beta_i}}
\int_0^{\Lambda_i} e^{-x^{\beta_i}}
\mathop{dx},
\end{equation}
which together with~\eqref{eq:S1-lower} gives
\begin{equation*}
\frac{
\mathbb{\Sigma}_{2,2} (u, \bm{\Lambda}, \bm{S})
}{
\mathbb{\Sigma}_1' ( u, \bm{\Lambda}, \bm{S} )
}
\leq
\frac{
C \, e^{c \sum_{j \in \mathcal{I}^c} \Lambda_j^{\beta_j}}
}{
H_{\bm{\nu}, \mathbb{V}_{\bm{w}}, \mathbb{W}_{\bm{w}}}
\left( [ \bm{0}, \bm{S} ] \right)
}
\prod_{i \in \mathcal{I}} \exp \left( -\frac{\varkappa_i}{64} \, S_i^{\alpha_i} \right)
\end{equation*}
and therefore
\begin{equation*}
  \lim_{\bm{\Lambda}_{\mathcal{I}^c} \to \infty}
  \lim_{\bm{S}_{\mathcal{I}} \to \infty}
  \limsup_{u \to \infty}
  \frac{\mathbb{\Sigma}_{2,2} (u, \bm{\Lambda}, \bm{S})}{\mathbb{\Sigma}_1' ( u, \bm{\Lambda}, \bm{S} )}
  = 0.
\end{equation*}

Next, we bound the sum \(\mathbb{\Sigma}_{2, 1}\) over adjacent events, that is,
those with \(\min_{i \in \mathcal{I}} | k_i - l_i | \leq 1\). To this end, introduce the
following reindexing of the sum: for a given pair \(( \bm{k}, \bm{l} )\),
indexing \(\mathbb{\Sigma}_{2, 1}\), let \(\mathcal{I}_0 ( \bm{k}, \bm{l} )\) denote those
indices, in which \(\bm{l}_{\mathcal{I}_0} = \bm{k}_{\mathcal{I}_0}\). This set may be empty, but
it cannot be equal to \(\mathcal{I}\), since in this case \(\bm{l}_{\mathcal{I}}\) would be
equal to \(\bm{k}_{\mathcal{I}}\). Let also \(\mathcal{I}_1 ( \bm{k}, \bm{l} )\) denote those
indices, in which \(\bm{l}_{\mathcal{I}_1} = \bm{k}_{\mathcal{I}_1} + 1\). This set may be equal
to \(\mathcal{I}\), but it cannot be empty, since these terms are already included in
\(\mathbb{\Sigma}_{2, 2} ( u, \bm{\Lambda}, \bm{S} )\). Finally, denote
and \(\mathcal{I}_2 = \mathcal{I} \setminus \left( \mathcal{I}_0 \cup \mathcal{I}_1 \right)\). We have
\begin{equation*}
\mathbb{\Sigma}_{2, 1} ( u, \bm{\Lambda}, \bm{S} )
=
\sum_{\mathcal{I}_0 \subsetneq \mathcal{I}} \
\sum_{\varnothing \, \neq \mathcal{I}_1 \subset \mathcal{I}} \
\sum_{\bm{k}_{\mathcal{I}_0} = \bm{l}_{\mathcal{I}_0} \leq \bm{N}_{u, \mathcal{I}_0}} \
\sum_{
  \substack{
    \bm{l}_{\mathcal{I}_1} = \bm{k}_{\mathcal{I}_1} + 1, \\[3pt]
    \bm{k}_{\mathcal{I}_1} \leq \bm{N}_{u, \mathcal{I}_1}
    }
  } \
\sum_{
  \substack{
    \bm{k}_{\mathcal{I}_2} \neq \bm{l}_{\mathcal{I}_2} \leq \bm{N}_{u, \mathcal{I}_2}, \\[3pt]
    \min_{i \in \mathcal{I}_2} | k_i - l_i | > 1
  }
} \
2 P_{\bm{b}} ( \bm{S} \bm{k}, \bm{S} \bm{l}, \bm{S} ).
\end{equation*}
Let us split \([ \bm{0}, \bm{S} ]\) along axes from \(\mathcal{I}_1\)
into \(2^{| \mathcal{I}_1 |}\) subintervals
\[
  [ \bm{0}, \bm{S} ]
  = \bm{S} \left[ \bm{1}_{\mathcal{I}_1} \bm{S}^{-1/2}, \bm{1} \right]
  \cup \bigcup_{j =0}^{2^{| \mathcal{I}_1 |} - 1} C_j,
  \qquad
  \bm{1}_{\mathcal{I}_1, i}
  = \begin{cases}
    1, & i \in \mathcal{I}_1, \\
    0, & i \not\in \mathcal{I}_1,
  \end{cases}
\]
and use it to obtain the following bound:
\begin{equation} \label{eq:14}
  P_{\bm{b}} ( \bm{S} \bm{k}, \bm{S} \bm{l}, \bm{S} )
  \leq
  A ( \bm{k}, \bm{l} )
  +\sum_{j = 0}^{2^{| \mathcal{I}_1 |} - 1}
  \mathbb{P} \left\{
  \exists \, \bm{t} \in C_j \colon \bm{X} ( \bm{t} )
  > u \bm{b}
  \right\}
  \eqqcolon
  \mathbb{\Sigma}_{2, 1}' ( u, \bm{\Lambda}, \bm{S} )
  +\mathbb{\Sigma}_{2, 1}'' ( u, \bm{\Lambda}, \bm{S} ).
\end{equation}
where
\begin{equation*}
  A ( \bm{k}, \bm{l} )
  =
  \mathbb{P} \left\{
    \begin{aligned}
      & \exists \, \bm{t} \in u^{-2 / \bm{\nu}} \bm{S}
        \left[ \bm{k}, \bm{k} + \bm{1} \right] \colon
      & \bm{X} ( \bm{t} ) > u \bm{b}
      \\[7pt]
      & \exists \, \bm{s}_{\mathcal{I}_0} \in u^{-2 / \bm{\nu}_{\mathcal{I}_0}} \bm{S}_{\mathcal{I}_0}
        \left[ \bm{k}_{\mathcal{I}_0}, \bm{k}_{\mathcal{I}_0} + \bm{1}_{\mathcal{I}_0} \right],
      &
      \\[7pt]
      & \exists \, \bm{s}_{\mathcal{I}_1} : u^{-2 / \bm{\nu}_{\mathcal{I}_1}} \bm{S}_{\mathcal{I}_1}
        \left[ \bm{S}_{\mathcal{I}_{1}}^{-1/2} + \bm{k}_{\mathcal{I}_1}
          +\bm{1}, \bm{k}_{\mathcal{I}_1} + \bm{2}_{\mathcal{I}_1} \right],
      &
      \\[7pt]
      & \exists \, \bm{s}_{\mathcal{I}_2} \in u^{-2 / \bm{\nu}_{\mathcal{I}_2}} \bm{S}_{\mathcal{I}_2}
        \left[ \bm{l}_{\mathcal{I}_2}, \bm{l}_{\mathcal{I}_2} + \bm{1}_{\mathcal{I}_2} \right] \colon
      & \bm{X} ( \bm{s} ) > u \bm{b}
    \end{aligned}
  \right\}
\end{equation*}
In order to apply Lemma~\ref{lemma:double-sum}, we lengthen the \(\mathcal{I}_1\)
interval interval in the definition of \(A ( \bm{k}, \bm{l} )\)
by \(\bm{S}_{\mathcal{I}'}^{1/2}\) so that it fell under the definition of double event
probabilty~\eqref{def:double-event}, and therefore
\begin{align*}
  \frac{A ( \bm{k}, \bm{l} )}{\mathbb{P} \left\{ \bm{X} ( \bm{0} ) > u \bm{b} \right\}}
  & \leq \frac{
    P_{\bm{b}} (
      \bm{S} \bm{k},
      \bm{S} \bm{k} \bm{1}_{\mathcal{I}_0 \cup \mathcal{I}_1}
      +\bm{S}^{1/2} \bm{1}_{\mathcal{I}_1}
      +\bm{1}_{\mathcal{I}_1}
      +\bm{S} \bm{l} \bm{1}_{\mathcal{I}_2},
      \bm{S}
    )
  }{
    \mathbb{P} \left\{ \bm{X} ( \bm{0} ) > u \bm{b} \right\}
  }
  \\[7pt]
  & \leq
  C
  \prod_{i_c \in \mathcal{I}^c} A_{c, i_c}
  \prod_{i_0 \in \mathcal{I}_0} A_{0, i_0}
  \prod_{i_1 \in \mathcal{I}_1} A_{1, i_1}
  \prod_{i_2 \in \mathcal{I}_2} A_{2, i_2},
\end{align*}
where \(A_{c, i} = \exp \left( c \Lambda_i^{\alpha_i} \right)\), \(A_{0, i} = S_i\)
\begin{gather*}
  A_{1, i}
  = S_i \exp \left(
    -\frac{\xi_i ( \bm{w} )}{32} S_i^{\alpha_i / 2}
    -\frac{\tau_i ( \bm{w} )}{4} S_i^{\beta_i} k_i^{\beta_i} u^{2 - 2 \beta_i / \nu_i}
  \right),
  \\[7pt]
  A_{2, i}
  =
  ( l_i - k_i - 1 )^{-2} \exp \left(
    -\frac{\xi_i ( \bm{w} )}{32} ( l_i - k_i - 1 )^{\alpha_i}
    -\frac{\tau_i ( \bm{w} )}{4} S_i^{\beta_i} k_i^{\beta_i} u^{2 - 2 \beta_i / \nu_i}
  \right).
\end{gather*}
Summing up \(\mathbb{\Sigma}_{2, 1}' ( u, \bm{\Lambda}, \bm{S} )\) in \(\bm{k}\)
and \(\bm{l}\), we obtain
\begin{equation*}
\frac{
  \mathbb{\Sigma}_{2, 1}' ( u, \bm{\Lambda}, \bm{S} )
}{
  \mathbb{P} \left\{ \bm{X} ( \bm{0} ) > u \bm{b} \right\}
}
\leq
C
\prod_{i_c \in \mathcal{I}^c} \mathbb{A}_{c, i_c}
\sum_{\mathcal{I}_0 \subsetneq \mathcal{I}} \
\sum_{\varnothing \, \neq \, \mathcal{I}_1 \subset \mathcal{I}} \
\prod_{i_0 \in \mathcal{I}_0} \mathbb{A}_{0, i_0}
\prod_{i_1 \in \mathcal{I}_1} \mathbb{A}_{1, i_1}
\prod_{i_2 \in \mathcal{I}_2} \mathbb{A}_{2, i_2},
\end{equation*}
where
\(\mathbb{A}_{c, i} = A_{c, i} = \exp \left( c \Lambda_i^{\alpha_i} \right)\),
\(\mathbb{A}_{0, i} = S_i u^{\zeta_i} \Lambda_i / S_i = u^{\zeta_i} \Lambda_i\)
and
\begin{gather*}
  \mathbb{A}_{1, i}
  = \frac{u^{\zeta_i}}{S_i}
  \left( \frac{4}{\tau_i ( \bm{w} )} \right)^{1/\beta_i}
  \Gamma \left( \frac{1}{\beta_i} + 1 \right)
  S_i \exp \left( -\frac{\xi_i ( \bm{w} )}{32} S_i^{\alpha_i / 2} \right),
  \\[7pt]
  \mathbb{A}_{2, i}
  = \frac{u^{\zeta_i}}{S_i}
  \left( \frac{4}{\tau_i ( \bm{w} )} \right)^{1/\beta_i}
  \Gamma \left( \frac{1}{\beta_i} + 1 \right)
  \exp \left( -\frac{\xi_i ( \bm{w} )}{32} S_i^{\alpha_i} \right).
\end{gather*}
\sloppy
Therefore, by~\eqref{eq:S1-lower}, \eqref{eq:16} and
Lemma~\ref{lemma:const-lower-bound} to bound
\(1 / H_{\bm{\nu}, \mathbb{V}_{\bm{w}}, \mathbb{W}_{\bm{w}}}\) from above by
\(C / \prod_{i \in \mathcal{I}} S_i\), we obtain
\begin{multline*}
  \frac{\mathbb{\Sigma}_{2,1}' ( u, \bm{\Lambda}, \bm{S} )
    }{\mathbb{\Sigma}_{1}' ( u, \bm{\Lambda}, \bm{S} )}
  \leq
  C
  \prod_{i_c \in \mathcal{I}^c}
  \exp \left(
    c \Lambda_{i_c}^{\alpha_{i_c}}
  \right)
  \times \\[7pt] \times
  \sum_{\mathcal{I}_0 \subsetneq \mathcal{I}} \
  \sum_{\varnothing \, \neq \, \mathcal{I}_1 \subset \mathcal{I}}
  \prod_{i_0 \in \mathcal{I}_0} \frac{\Lambda_{i_0}}{S_i}
  \prod_{i_1 \in \mathcal{I}_1} \exp \left(
    -\frac{\xi_{i_1} ( \bm{w} )}{32} S_i^{\alpha_{i_1} / 2}
  \right)
  \prod_{i_2 \in \mathcal{I}_2} \frac{1}{S_i} \exp \left(
    -\frac{\xi_{i_2} ( \bm{w} )}{32} S_{i_2}^{\alpha_{i_2}}
  \right).
\end{multline*}
and
\begin{equation*}
  \lim_{\bm{\Lambda}_{\mathcal{I}^c} \to \infty}
  \lim_{\bm{S}_{\mathcal{I}} \to \infty}
  \limsup_{u \to \infty}
  \frac{\mathbb{\Sigma}_{2,1}' ( u, \bm{\Lambda}, \bm{S} )
    }{\mathbb{\Sigma}_{1}' ( u, \bm{\Lambda}, \bm{S} )}
  = 0.
\end{equation*}
Finally, let us find a bound for the second term of \eqref{eq:14}.
By local Pickands lemma
\begin{equation*}
  \mathbb{P} \left\{ \exists \, \bm{t} \in C_j \colon \bm{X} ( \bm{t} ) > u \bm{b} \right\}
  \sim
  H_{\bm{\nu}, \mathbb{V}_{\bm{w}}, \mathbb{W}_{\bm{w}}}
  \left( \left[
  \bm{0}, \bm{1}_{\mathcal{I}_1} \bm{S}^{1/2}
  +\bm{1}_{\mathcal{I} \setminus \mathcal{I}_1} \bm{S} \right] \right) \
  \mathbb{P} \left\{ \bm{X} ( \bm{0} ) > u \bm{b} \right\}
\end{equation*}
and therefore, using \(N_{u, i} = \ceil*{\Lambda_i u^{\zeta_i} / S_i}\), obtain
\begin{multline*}
  \mathbb{\Sigma}_{2,1}'' ( u, \bm{\Lambda}, \bm{S} )
  \coloneqq
  \sum_{\bm{k} \leq \bm{N}_u} \
  \sum_{\varnothing \, \neq \, \mathcal{I}_1 \subset \mathcal{I}} \quad
  \sum_{j = 0}^{2^{| \mathcal{I}_1 |} - 1}
  \mathbb{P} \left\{ \exists \, \bm{t} \in C_j \colon \bm{X} ( \bm{t} ) > u \bm{b} \right\}
  \sim \\[7pt] \sim
  \mathbb{P} \left\{ \bm{X} ( \bm{0} ) > u \bm{b} \right\}
  \prod_{i \in \mathcal{I}} \frac{\Lambda_i u^{\zeta_i}}{S_i}
  \sum_{\varnothing \, \neq \, \mathcal{I}_1 \subset \mathcal{I}}
  2^{| \mathcal{I}_1 |}
  H_{\bm{\nu}, \mathbb{V}_{\bm{w}}, \mathbb{W}_{\bm{w}}}
  \left( \left[ \bm{0}, \bm{1}_{\mathcal{I}_1} \bm{S}^{1/2}
    +\bm{1}_{\mathcal{I} \setminus \mathcal{I}_1} \bm{S} \right] \right)
\end{multline*}
and
\begin{equation*}
  \frac{
    \mathbb{\Sigma}_{2,1}'' ( u, \bm{\Lambda}, \bm{S} )
  }{
    \mathbb{\Sigma}_1' ( u, \bm{\Lambda}, \bm{S} )
  }
  \leq
  \Bigg[
    \prod_{i \in \mathcal{I}} \Lambda_i
  \Bigg]
  \sum_{\varnothing \, \neq \, \mathcal{I}_1 \subset \mathcal{I}}
  2^{| \mathcal{I}_1 |}
  H_{\bm{\nu}, \mathbb{V}_{\bm{w}}, \mathbb{W}_{\bm{w}}}
  \left( \left[ \bm{0}, \bm{1}_{\mathcal{I}_1} \bm{S}^{1/2}
    +\bm{1}_{\mathcal{I} \setminus \mathcal{I}_1} \bm{S} \right] \right)
  \prod_{j \in \mathcal{I}} \frac{1}{S_i}.
\end{equation*}
By subadditivity of the generalized Pickands-Piterbarg constant,
\[
\lim_{\bm{S} \to \infty}
H_{\bm{\nu}, \mathbb{V}_{\bm{w}}, \mathbb{W}_{\bm{w}}}
\left( [ \bm{0}, \bm{1}_{\mathcal{I}'} \bm{S}^{1/2} + \bm{1}_{\mathcal{I} \setminus \mathcal{I}'} \bm{S} ] \right)
\prod_{j \in \mathcal{I}} \frac{1}{S_i}
\leq
\lim_{\bm{S} \to \bm{\infty}}
H_{\bm{\nu}, \mathbb{V}_{\bm{w}}, \mathbb{W}_{\bm{w}}}
\left( [ \bm{0}, \bm{1} ] \right)
\prod_{j \in \mathcal{I}_1} S_j^{-1/2}
= 0,
\]
and therefore
\begin{equation*}
  \lim_{\bm{\Lambda}_{\mathcal{I}^c} \to \infty}
  \lim_{\bm{S}_{\mathcal{I}} \to \infty}
  \lim_{u \to \infty}
  \frac{
    \mathbb{\Sigma}_{2,1}'' ( u, \bm{\Lambda}, \bm{S} )
  }{
    \mathbb{\Sigma}_1' ( u, \bm{\Lambda}, \bm{S} )
  }
  = 0.
\end{equation*}
\end{proof}
\section{Appendix}
\label{sec:org8f085fb}
\label{SEC:appendix}
\subsection{Covariance lemma}
\label{sec:orgb0a80e4}
\begin{lemma}\label{lemma:quadratic-covariance}
Let \(D = ( D_{i, j} )_{i, j \in \mathcal{F}}\) be a block matrix, each block of which is
a \(d \times d\) matrix. Then the following are equivalent:
\begin{enumerate}
\item \(D \succcurlyeq 0\).
\item There exists a family \(( C_{i, k} )_{i, k \in \mathcal{F}}\) of \(d \times d\) matrices such that
\begin{equation*}
D_{i, j} = \sum_{k \in \mathcal{F}} C_{i, k} \, C_{k, j}^\top.
\end{equation*}
\item A matrix-valued function \(R\) defined by
\begin{equation*}
R ( \bm{t}, \bm{s} ) = \sum_{i, j \in \mathcal{F}} D_{i, j} \, t_i \, s_j,
\qquad \bm{t}, \bm{s} \in \mathbb{R}^n
\end{equation*}
is a covariance function of some \(\mathbb{R}^d\)-Gaussian random field.
\end{enumerate}

If either of statements hold, then \(R\) is the covariance function of
\begin{equation*}
\bm{Z} ( \bm{t} ) \coloneqq \sum_{k \in \mathcal{F}} C_k ( \bm{t} ) \, \bm{\mathcal{N}}_k,
\quad
C_k ( \bm{t} ) \coloneqq \sum_{i \in \mathcal{F}} C_{i, k} \, t_i,
\end{equation*}
where \(\bm{\mathcal{N}}_k \sim N ( \bm{0}, I )\) are independent standard Gaussian
vectors.
\end{lemma}

\begin{proof}[Proof of Lemma~\ref{lemma:quadratic-covariance}]
Let \(n = | \mathcal{F} |\). If \(D \succcurlyeq 0\), then there exists a \((nd) \times (nd)\)
matrix \(C\) such that \(D = C C^\top\). By viewing \(C\) as a block matrix
with the same block structure as \(D\), we find that
\begin{equation*}
D_{i, j} = \sum_{k \in \mathcal{F}} C_{i, k} \, C_{k, j}^\top.
\end{equation*}
Hence,
\begin{equation*}
R ( \bm{t}, \bm{s} )
= \sum_{i, j \in \mathcal{F}} \sum_{k \in \mathcal{F}} C_{i, k} \, C_{k, j}^\top \, t_i \, s_j
= \sum_{k \in \mathcal{F}}
\left( \sum_{i \in \mathcal{F}} C_{i, k} \, t_i \right)
\left( \sum_{j \in \mathcal{F}} C_{j, k} \, s_j \right)^\top,
\end{equation*}
which clearly is positive definite.

If, on the other hand, \(R\) is positive definite, then for all collections
\(\{ ( \bm{t}_i, \bm{z}_i ) \in \mathbb{R}^n \times \mathbb{R}^{\mathcal{F}}, \ i = 1, \ldots, p \}\) holds
\begin{equation*}
0 \leq
\sum_{k, l = 1}^p \bm{z}_k^\top R ( \bm{t}_k, \bm{t}_l ) \, \bm{z}_l
= \sum_{k, l = 1}^p \sum_{i, j \in \mathcal{F}} \bm{z}_k^\top D_{i, j} \, \bm{z}_l \, t_{k, i} \, s_{l, j}
= \sum_{i, j \in \mathcal{F}} \bm{y}_i^\top D_{i, j} \, \bm{y}_j
= \bm{y}^\top D \, \bm{y},
\end{equation*}
where
\begin{equation*}
\bm{y}_i = \sum_{k = 1}^p \bm{z}_k \, t_{k, i},
\quad
\bm{y} = ( \bm{y}_i )_{i \in \mathcal{F}}.
\end{equation*}
Since \(\bm{z}_k\) and \(\bm{t}_k\) are arbitrary, \(\bm{y}\) is also
arbitrary. Hence, \(D \succcurlyeq 0\), and the chain of implications has come full
circle.
\end{proof}
\subsection{Supplementary lemma on the quadratic optimization problem}
\label{sec:orge838649}
\begin{proof}[Proof of Lemma~\ref{lemma:QPP-supp}]
For \(j \in \{ 1, \ldots, d \}\) define
\begin{equation*}
f_j ( \bm{t} ) = \widetilde{\bm{b}} ( \bm{t} ) \, \Sigma^{-1} ( \bm{t} ) \, \bm{e}_j,
\quad
g_j ( \bm{t} ) = \widetilde{b}_j ( \bm{t} ) - b_j.
\end{equation*}
By continuity of \(\Sigma^{-1} ( \bm{t} )\) and the fact that \(\Sigma \mapsto \widetilde{\bm{b}}\)
is Lipschitz continuous, these functions are continuous and we have
\begin{align*}
I ( \bm{t} )
& =
\max \left\{
V \, \middle| \,
V = \left\{
i \in \{ 1, \ldots, d \} \colon \
g_i ( \bm{t} ) = 0, \
f_i ( \bm{t} ) > 0
\right\}
\right\},
\\
K ( \bm{t} )
& =
\max \left\{
V \, \middle| \,
V =
\left\{
j \in \{ 1, \ldots, d \} \colon \
g_j ( \bm{t} ) = 0, \
f_j ( \bm{t} ) = 0
\right\}
\right\},
\\
L ( \bm{t} )
& =
\max
\left\{
V \, \middle| \,
V =
\left\{
l \in \{ 1, \ldots, d \} \colon \
g_l ( \bm{t} ) > 0, \
f_l ( \bm{t} ) = 0
\right\}
\right\},
\end{align*}
where the maximum is in the power set of \(\{ 1, \ldots, d \}\) ordered by
inclusion. It follows from the quadratic optimization lemma that
these three sets are disjoint and the equality
\(I ( \bm{t} ) \cup K ( \bm{t} ) \cup L ( \bm{t} ) = \{ 1, \ldots, d \}\)
holds for each \(\bm{t} \in E\).

For \(U, \, V \subset \{ 1, \ldots, d \}\) define
\begin{equation*}
A_U
\coloneqq
\bigcap_{j \in V} f_j^{-1} ( 0 )
\cap
\bigcap_{j \in V^c} \left( f_j^{-1} ( 0 ) \right)^c,
\quad
B_V
\coloneqq
\bigcap_{j \in V} g_j^{-1} ( 0 )
\cap
\bigcap_{j \in V^c} \left( g_j^{-1} ( 0 ) \right)^c,
\quad
C_{U, V} = A_U \cap B_V.
\end{equation*}
One can easily check the following properties of these sets:
\begin{enumerate}
\item \(V \neq U \iff A_V \cap A_U = \varnothing \iff \ B_V \cap B_U = \varnothing\)
and hence \(( U, V ) \neq ( U', V' ) \iff C_{U, V} \neq C_{U', V'}\).
\item \(( A_V )_{V \in 2^d}, \  ( B_V )_{V \in 2^d}\)
and \(( C_{U, V} )_{U, \, V \in 2^d}\) are finite covers of \(E\).
\item \(A_V\), \(B_V\) and \(C_{U, V}\) are locally closed sets
(that is, intersections of an open and a closed set).
\item \(I ( \bm{t} ) = U^c\) on \(A_U\).
\item \(L ( \bm{t} ) = V^c\) on \(B_V\).
\item \(K ( \bm{t} ) = U \cap V\) on \(C_{U, V}\).
\end{enumerate}
In other words, we have
\begin{equation*}
I ( \bm{t} ) = \sum_{U \in 2^d} U^c \, \mathbb{1}_{A_U} ( \bm{t} ),
\quad
L ( \bm{t} ) = \sum_{V \in 2^d} V^c \, \mathbb{1}_{B_V} ( \bm{t} ),
\quad
K ( \bm{t} ) = \sum_{U, \, V \in 2^d} U \cap V \, \mathbb{1}_{C_{U, V}} ( \bm{t} ).
\end{equation*}
Note that some of \(A_V\)'s, \(B_V\)'s or \(C_{U, V}\)'s may be empty.

Let us split \(2^d\) in two parts
\(2^d = \mathbb{V}_1 ( \bm{t} ) \cup \mathbb{V}_2 ( \bm{t} )\)
in the following way:
\begin{enumerate}
\item if \(V \in \mathbb{V}_1 ( \bm{t} )\), then for all \(\varepsilon > 0\) holds
\(A_V \cap  B_\varepsilon ( \bm{t} ) \neq \varnothing\),
\item if \(V \in \mathbb{V}_2 ( \bm{t} )\), then there exists
\(\varepsilon_0 = \varepsilon_0 ( V, \bm{t} ) > 0\) such that
\(A_V \cap B_{\varepsilon_0} ( \bm{t} ) = \varnothing\).
\end{enumerate}
Define
\begin{equation*}
\varepsilon ( \bm{t} ) \coloneqq \min_{V \in \mathbb{V}_2 ( \bm{t} )} \varepsilon_0 ( V, \bm{t} ).
\end{equation*}
If \(\bm{s} \in B_\varepsilon ( \bm{t} )\) with \(\varepsilon < \varepsilon ( \bm{t} )\) there
exists \(V \in \mathbb{V}_1 ( \bm{t} )\)
such that  \(\bm{s} \in A_V\) and therefore \(I ( \bm{s} ) = V^c\).
Since \(j \in I^c ( \bm{t} ) = K ( \bm{t} ) \cup L ( \bm{t} ) = V\) if and only
if \(f_j ( \bm{s} ) = 0\), using continuity of \(f_j\) and letting
\(\varepsilon \to 0\) we obtain that \(f_j ( \bm{t} ) = 0\) and therefore
\(K ( \bm{s} ) \cup L ( \bm{s} ) = V \subset K ( \bm{t} ) \cup L ( \bm{t} )\).
The latter is equivalent to \(I ( \bm{t} ) \subset I ( \bm{s} )\).

Now we proceed to the last claim of the lemma. Similarly to what we did above,
split again \(2^d\) in two parts
\(2^d = \mathbb{V}_1' ( \bm{t} ) \cup \mathbb{V}_2' ( \bm{t} )\),
where
\begin{enumerate}
\item \(V \in \mathbb{V}_1' ( \bm{t} )\) if for all \(\varepsilon > 0\)
holds \(B_V \cap B_\varepsilon ( \bm{t} ) \neq \varnothing\),
\item \(V \in \mathbb{V}_2' ( \bm{t} )\) if there exists
\(\varepsilon_0' = \varepsilon_0' ( V, \bm{t} ) > 0\) such that
\(B_V \cap B_{\varepsilon_0} ( \bm{t} ) = \varnothing\).
\end{enumerate}
Define
\begin{equation*}
\varepsilon' ( \bm{t} ) = \min_{V \in \mathbb{V}_2' ( \bm{t} )} \varepsilon_0' ( V, \bm{t} ).
\end{equation*}
Then, if \(\bm{s} \in B_\varepsilon ( \bm{t} )\) with \(\varepsilon < \varepsilon' ( \bm{t} )\),
there exists \(V \in \mathbb{V}_1' ( \bm{t} )\) such that \(\bm{s} \in B_V\) and therefore
\(L ( \bm{s} ) = V^c\). Since \(j \in L^c ( \bm{s} ) = I ( \bm{s} ) \cup K ( \bm{s} ) = V\)
is equivalent to \(g_j ( \bm{s} ) = 0\), letting \(\varepsilon \to 0\) and using
continuity of \(g_j\) we obtain that \(g_j ( \bm{t} ) = 0\). Hence,
\(I ( \bm{s} ) \cup K ( \bm{s} ) = V \subset I ( \bm{t} ) \cup K ( \bm{t} )\).
The latter is equivalent to \(L ( \bm{t} ) \subset L ( \bm{s} )\).

From these two properties follows that if \(\bm{s} \in B_\varepsilon ( \bm{t} )\)
with \(\varepsilon < \min \{ \varepsilon ( \bm{t} ), \, \varepsilon' ( \bm{t} ) \}\),
then \(K ( \bm{s} ) \subset K ( \bm{t} )\).
\end{proof}
\subsection{Expansions}
\label{sec:orge65020a}
In this section, we develop some asymptotic expansions, first without
assuming~\ref{A2.6}.
\subsubsection{Inverse of Sigma}
\label{sec:org22e13b9}
\begin{lemma}\label{lemma:Sigma-inverse}
The inverse of the variance matrix \(\Sigma ( \bm{t} ) = R ( \bm{t}, \bm{t} )\)
admits the following asymptotic formula
\begin{equation}\label{Sigma-inv-1}
\Sigma^{-1} ( \bm{t} ) - \Sigma^{-1}
\sim
\Sigma^{-1} \left[
\sum_{i = 1}^n
\Big[
B_{1, i} \, t_i^{\beta_i'}
+B_{2, i} \, t_i^{\beta_i}
\Big]
+\sum_{i, j \in \mathcal{F}}^n
\widetilde{D}_{i,j} \, t_i^{\beta_i/2} \, t_j^{\beta_j/2}
\right] \Sigma^{-1},
\end{equation}
where
\begin{equation}\label{Sigma-inv-coeff}
B_{k, i} = A_{k, i} + A_{k, i}^\top,
\quad
k = 1, 2
\quad \text{and} \quad
\widetilde{D}_{i,j}
= A_{6, i, j} + B_{1, i} \Sigma^{-1} B_{1, j},
\end{equation}
and the error is of order
\begin{equation*}
o \left( \sum_{i = 1}^n t_i^{\beta_i} \right).
\end{equation*}
\end{lemma}

\begin{proof}[Proof of Lemma~\ref{lemma:Sigma-inverse}]
Plugging \(\bm{t} = \bm{s}\) into~\eqref{Sigma-minus-R}, we obtain
\begin{equation}\label{Sigma-exp-formula}
\Sigma - \Sigma ( \bm{t} )
=
\sum_{i = 1}^n \Big[ B_{1, i} \, t_i^{\beta_i'} + B_{2, i} \, t_i^{\beta_i} \Big]
+\sum_{i, \, j \, \in \mathcal{F}}
A_{6, i, j} \, t_i^{\beta_i/2} \, t_j^{\beta_j/2}
+o \left( \sum_{i = 1}^n t_i^{\beta_i} \right)
\end{equation}
with \(B_{k, i}\), \(k = 1, 2\) defined in~\eqref{Sigma-inv-coeff}.
To find its expansion up to the second order, use the Neumann power series
\begin{equation*}
\Sigma^{-1} ( \bm{t} )
= \Big[ \Sigma - \big[ \Sigma - \Sigma ( \bm{t} ) \big] \Big]^{-1}
= \Big[ I - \Sigma^{-1} \big[ \Sigma - \Sigma ( \bm{t} ) \big] \Big]^{-1} \Sigma^{-1}
= \sum_{k \geq 0} \left( \Sigma^{-1} \big[ \Sigma - \Sigma ( \bm{t} ) \big] \right)^k \Sigma^{-1}.
\end{equation*}
Note that
\begin{equation*}
\sum_{k \geq 3} \left( \Sigma^{-1} [ \Sigma - \Sigma ( \bm{t} ) ] \right)^k \Sigma^{-1}
= o \left( \sum_{i = 1}^n t_i^{\beta_i} \right)
\end{equation*}
and the first three terms give
\begin{equation*}
\Sigma^{-1} ( \bm{t} )
\sim \Sigma^{-1} + \Sigma^{-1} \big[ \Sigma - \Sigma ( \bm{t} ) \big] \Sigma^{-1}
+\Sigma^{-1} \big[ \Sigma - \Sigma ( \bm{t} ) \big] \Sigma^{-1} \big[ \Sigma - \Sigma ( \bm{t} ) \big] \Sigma^{-1}.
\end{equation*}
Plugging~\eqref{Sigma-exp-formula} into the latter, we obtain the desired
result.
\end{proof}
\subsubsection{Exponential prefactor}
\label{sec:org5d0598a}
On several occasions we shall need a formula for the following quantity, which
we shall refer to as \textit{the exponential prefactor}:
\begin{align}
\ln
\frac{
\varphi_{\Sigma ( \bm{\tau} )} \left( u \bm{b} - u^{-1} \bm{x} \right)
}{
\varphi_{\Sigma} ( u \, \bm{b} )
}
& =
-\frac{1}{2} \,
\left( u \bm{b} - u^{-1} \bm{x} \right)^\top
\Sigma^{-1} ( \bm{\tau} )
\left( u \bm{b} - u^{-1} \bm{x} \right)
+\frac{1}{2} \, u^2 \, \bm{b}^\top \Sigma^{-1} \bm{b}
\notag
\\[7pt]
& =
-\frac{1}{2} \, u^2 \, \bm{b}^\top \Big[ \Sigma^{-1} ( \bm{\tau} ) - \Sigma^{-1} \Big] \bm{b}
+\bm{b}^\top \Sigma^{-1} ( \bm{\tau} ) \, \bm{x}
-\frac{1}{2 u^2} \, \bm{x}^\top \Sigma^{-1} ( \bm{\tau} ) \, \bm{x}
\label{prefactor-equality}
\\[7pt]
& \leq
-\frac{1}{2} \, u^2 \, \bm{b}^\top \Big[ \Sigma^{-1} ( \bm{\tau} ) - \Sigma^{-1} \Big] \bm{b}
+\bm{b}^\top \Sigma^{-1} ( \bm{\tau} ) \, \bm{x}.
\label{prefactor-inequality}
\end{align}

Using~\eqref{Sigma-inv-1}
and observing that
\begin{equation}\label{wB1w-equals-zero}
\bm{b}^\top \Sigma^{-1} B_{1, i} \Sigma^{-1} \bm{b}
= \bm{w}^\top B_{1, i} \bm{w}
= \bm{w}^\top ( A_{1, i} + A_{1, i}^\top ) \, \bm{w}
= \bm{w}^\top \, \underbracket[.1pt]{A_{1, i} \, \bm{w}}_{= \, \bm{0}}
+\underbracket[.1pt]{\bm{w}^\top A_{1, i}}_{= \, \bm{0}} \bm{w}
= 0,
\end{equation}
we obtain
\begin{equation}\label{exp-prefactor-1}
\bm{b}^\top \Big[ \Sigma^{-1} ( \bm{\tau} ) - \Sigma^{-1} \Big] \bm{b}
= \sum_{i, j = 1}^n
\Xi_{i,j}
\, \tau_i^{\beta_i/2} \, \tau_j^{\beta_j/2}
+o \left( \sum_{i = 1}^n \tau_i^{\beta_i} \right),
\end{equation}
where
\begin{equation}\label{def:XI}
\Xi_{i,j} = \bm{w}^\top \Big[
B_{2, i} \, \mathbb{1}_{i = j} + \widetilde{D}_{i,j} \, \mathbb{1}_{i, j \in \mathcal{F}}
\Big] \, \bm{w}.
\end{equation}
Note that
\begin{equation}\label{XI-alternative}
\Xi_{i, j} = \bm{w}^\top \Big[
2 \, A_{2, i} \, \mathbb{1}_{i = j} + D_{i, j} \, \mathbb{1}_{i, j \in \mathcal{F}}
\Big]
\, \bm{w},
\qquad
D_{i, j}
= A_{6, i, j} + A_{1, i} \Sigma^{-1} A_{1, j}^\top,
\end{equation}
because \(A_1 \, \bm{w} = \bm{0}\) and
\(\bm{w}^\top B_{2, i} \, \bm{w}
= \bm{w}^\top ( A_{2, i} + A_{2, i}^\top ) \, \bm{w}
= 2 \, \bm{w}^\top A_{2, i} \, \bm{w}\).
\smallskip

Assuming~\ref{A2.6}, we find that
\begin{equation*}
\sum_{i, j \in \mathcal{F}} \bm{w}^\top D_{i, j} \, \bm{w} \, \tau_i^{\beta_i/2} \, \tau_j^{\beta_j/2}
= \widehat{\bm{w}}^\top D \, \widehat{\bm{w}}
\geq 0,
\end{equation*}
where
\(\widehat{\bm{w}} = ( \bm{w} \, \tau_i^{\beta_i/2} )_{i \in \mathcal{F}} \in \mathbb{R}^{\mathcal{F}}\).
Therefore,
\begin{equation}\label{XI-elliptical}
\sum_{i, j = 1}^n \Xi_{i, j} \, \tau_i^{\beta_i/2} \, \tau_j^{\beta_j/2}
= 2 \sum_{i = 1}^n \bm{w}^\top A_{2, i} \, \bm{w} \, \tau_i^{\beta_i}
+\underbracket[.5pt]{
\sum_{i, i \in \mathcal{F}} \bm{w}^\top D_{i, j} \, \bm{w} \, \tau_i^{\beta_i/2} \, \tau_j^{\beta_j/2}
}_{\geq \, 0}
\geq 2 \sum_{i = 1}^n \bm{w}^\top A_{2, i} \, \bm{w} \, \tau_i^{\beta_i}.
\end{equation}
\subsubsection{Generalized variance expansion}
\label{sec:orgc097537}
\begin{lemma}\label{lemma:gen-var-exp}
The generalized variance
\begin{equation*}
\sigma_{\bm{b}}^{-2} ( \bm{\tau} )
=
\min_{\bm{x} \geq \bm{b}} \bm{x}^\top \, \Sigma^{-1} ( \bm{\tau} ) \, \bm{x}
\end{equation*}
admits the following asymptotic formula:
\begin{equation}\label{gen-var-exp-formula}
\sigma_{\bm{b}}^{-2} ( \bm{\tau} ) - \sigma_{\bm{b}}^{-2} ( \bm{0} )
=
\sum_{i, j = 1}^n \Xi_{i, j} \, \tau_i^{\beta_i/2} \, \tau_j^{\beta_j/2}
+o \left( \sum_{i = 1}^n \tau_i^{\beta_i} \right),
\end{equation}
where \(\Xi\) is defined by~\eqref{def:XI}
or~\eqref{XI-alternative}.
\end{lemma}

\begin{proof}[Proof of Lemma~\ref{lemma:gen-var-exp}]
We remind the reader of our notation convention
\(A_I = A_{II} = ( A_{ij} )_{i, \, j \in I}\) for a given \(d \times d\) matrix
\(A\) and \(I \subset \{ 1, \ldots, d \}\).

Define a vector-valued function \(\bm{b} ( \bm{\tau} )\) by
\begin{equation*}
\sigma_{\bm{b}}^{-2} ( \bm{\tau} )
=
\min_{\bm{x} \geq \bm{b}} \bm{x}^\top \, \Sigma^{-1} ( \bm{\tau} ) \, \bm{x}
\eqqcolon
\bm{b}^\top ( \bm{\tau} ) \, \Sigma^{-1} ( \bm{\tau} ) \, \bm{b} ( \bm{\tau} ).
\end{equation*}
We have
\begin{equation}\label{eq:gen-var-exp-1-step}
\sigma_{\bm{b}}^{-2} ( \bm{\tau} ) - \sigma_{\bm{b}}^{-2} ( \bm{0} )
=
\bm{b}^\top ( \bm{\tau} ) \, \Sigma^{-1} ( \bm{\tau} ) \, \bm{b} ( \bm{\tau} )
-\bm{b}^\top ( \bm{0} ) \, \Sigma^{-1} ( \bm{0} ) \, \bm{b} ( \bm{0} ).
\end{equation}
Adding and subtracting \(\bm{b} ( \bm{0} )\) and using
Lemma~\ref{lemma:QPP}, we obtain
\begin{align*}
\bm{b}^\top ( \bm{\tau} ) \, \Sigma^{-1} ( \bm{\tau} ) \, \bm{b} ( \bm{\tau} )
& =
\bm{b}^\top ( \bm{\tau} ) \, \Sigma^{-1} ( \bm{\tau} ) \, \bm{b} ( \bm{0} )
+\bm{b} ( \bm{\tau} ) \,
\Sigma ( \bm{\tau} ) \,
\Big[ \bm{b} ( \bm{\tau} ) - \bm{b} ( \bm{0} ) \Big]
\\[7pt]
& =
\bm{b}^\top ( \bm{\tau} ) \, \Sigma^{-1} ( \bm{\tau} ) \, \bm{b} ( \bm{0} )
+\bm{b}_{I ( \bm{\tau} )}^\top ( \bm{\tau} ) \,
\Sigma_{I ( \bm{\tau} )}^{-1} ( \bm{\tau} ) \,
\Big[ \bm{b} ( \bm{\tau} ) - \bm{b} ( \bm{0} ) \Big]_{I ( \bm{\tau} )}.
\end{align*}
Take \(\varepsilon ( \bm{0} ) > 0\) as in Lemma~\ref{lemma:QPP-supp} and let
\(| \bm{\tau} | < \varepsilon ( \bm{0} )\). Then, we have that
\(I ( \bm{\tau} ) \subset I ( \bm{0} )\), and therefore
\(\bm{b}_{I ( \bm{\tau} )} ( \bm{\tau} )
= \bm{b}_{I ( \bm{\tau} )}
= \bm{b}_{I ( \bm{\tau} )} ( \bm{0} )\),
hence the last term on the right is zero.
Similarly, adding and subtracting \(\bm{b} ( \bm{\tau} )\) in the second term
of~\eqref{eq:gen-var-exp-1-step} and applying Lemma~\ref{lemma:QPP},
we obtain
\begin{equation*}
\bm{b}^\top ( \bm{0} ) \, \Sigma^{-1} ( \bm{0} ) \, \bm{b} ( \bm{0} )
=
\bm{b}^\top ( \bm{\tau} ) \, \Sigma^{-1} ( \bm{0} ) \, \bm{b} ( \bm{0} )
+\Big[ \bm{b} ( \bm{0} ) - \bm{b} ( \bm{\tau} ) \Big]_{I ( \bm{0} )}^\top \,
\Sigma_{I ( \bm{0} )}^{-1} ( \bm{0} ) \,
\bm{b}_{I ( \bm{0} )} ( \bm{0} ).
\end{equation*}
The last term is zero, since
\(I ( \bm{0} ) \subset I ( \bm{t} ) \cup K ( \bm{t} )\),
and therefore
\(\bm{b}_{I ( \bm{0} )} ( \bm{0} ) = \bm{b}_{I ( \bm{0} )} = \bm{b}_{I ( \bm{0} )} ( \bm{\tau} )\).
Therefore,
\begin{equation}\label{gen-var-true}
\sigma_{\bm{b}}^{-2} ( \bm{\tau} ) - \sigma_{\bm{b}}^{-2} ( \bm{0} )
=
\bm{b}^\top ( \bm{\tau} ) \,
\Big[ \Sigma^{-1} ( \bm{\tau} ) - \Sigma^{-1} ( \bm{0} ) \Big] \,
\bm{b} ( \bm{0} ).
\end{equation}
Note the small difference with the formula~\eqref{exp-prefactor-1}: here
we have \(\bm{b} ( \bm{\tau} )\) on the left instead of
\(\bm{b} = \bm{b} ( \bm{0} )\).
This, however, does not matter within the given error, since by~\ref{A2.3}
and~\eqref{Sigma-inv-1} the difference
\begin{equation*}
\Big[ \bm{b} ( \bm{\tau} ) - \bm{b} ( \bm{0} ) \Big]^\top
\Big[ \Sigma^{-1} ( \bm{\tau} ) - \Sigma^{-1} \Big] \, \bm{b} ( \bm{0} )
= o \left( \sum_{i, j = 1}^n \tau_i^{\beta_i'} \, \tau_j^{\beta_j'} \right)
\end{equation*}
is within the required error. By~\eqref{wB1w-equals-zero}, we arrive
at~\eqref{gen-var-exp-formula}.
\end{proof}
\subsubsection{Conditional mean}
\label{sec:org0fe267f}
\begin{lemma}\label{lemma:cond-mean}
The conditional mean vector
\begin{equation*}
\bm{d}_{\, \bm{\tau}} ( \bm{t} )
= \Big[ I - R ( \bm{\tau} + \bm{t}, \bm{\tau} ) \, \Sigma^{-1} ( \bm{\tau} ) \Big] \bm{b}
\end{equation*}
admits the following asymptotic formula
\begin{multline}\label{cond-mean-formula}
\bm{d}_{\, \bm{\tau}} ( \bm{t} )
=
\sum_{i = 1}^n \Bigg[
A_{2, i} \Big[ ( \tau_i + t_i )^{\beta_i} - \tau_i^{\beta_i} \Big]
+S_{\alpha_i, A_{5, i}} ( t_i )
\Bigg]
\bm{w}
\\[7pt]
+\sum_{i, j \in \mathcal{F}}
\tau_j^{\beta_j/2} \,
\Big[ ( \tau_i + t_i )^{\beta_i/2} - \tau_i^{\beta_i/2} \Big]
D_{i, j} \, \bm{w}
+\epsilon ( \bm{\tau}, \bm{t} ),
\end{multline}
with \(D_{i, j} = A_{6, i, j} + A_{1,i} \Sigma^{-1} A_{1, j}^\top\) and the error
\(\epsilon\) satisfying
\begin{equation*}
\epsilon ( \bm{\tau}, \bm{t} )
= o \left( \sum_{i = 1}^n \Big[
\tau_i^{\beta_i} + t_i^{\beta_i} + | t_i |^{\alpha_i}
\Big] \right).
\end{equation*}
\end{lemma}

\begin{proof}[Proof of Lemma~\ref{lemma:cond-mean}]
Throughout this calculation, the equivalence relation \(f \sim g\) is defined as
follows:
\begin{equation*}
f \sim g
\iff
f - g
= o \left( \sum_{i = 1}^n \Big[ t_i^{\beta_i} + \tau_i^{\beta_i} + | t_i |^{\alpha_i} \Big] \right).
\end{equation*}
In order to simplify the computation, define \(C\) and \(B\) by
\begin{equation*}
R ( \bm{t}, \bm{s} ) \eqqcolon \Sigma - C,
\qquad
\Sigma^{-1} ( \bm{\tau} )
\eqqcolon \Sigma^{-1} + \Sigma^{-1} \, B \, \Sigma^{-1}.
\end{equation*}
With these shorthands, we have
\begin{equation*}
I - R ( \bm{\tau} + \bm{t}, \bm{\tau} ) \, \Sigma^{-1} ( \bm{\tau} )
= I - \Big[ \Sigma - C \Big] \Big[ \Sigma^{-1} + \Sigma^{-1} B \, \Sigma^{-1} \Big]
= ( C - B ) \, \Sigma^{-1} + C \, \Sigma^{-1} B \, \Sigma^{-1}
\end{equation*}
Note that
\begin{align*}
C
& \sim
\sum_{i = 1}^n \Big[
A_{1, i} \, ( \tau_i + t_i )^{\beta_i'}
+A_{2, i} \, ( \tau_i + t_i )^{\beta_i}
+A_{1, i}^\top \, \tau_i^{\beta_i'}
+A_{2, i}^\top \, \tau_i^{\beta_i}
+S_{\alpha_i, A_{5, i}} ( t_i )
\Big]
+\sum_{i, \, j \in \mathcal{F}} A_{6, i, j} \, ( \tau_i + t_i )^{\beta_i / 2} \, \tau_j^{\beta_j / 2},
\\[7pt]
B
& \sim
\sum_{i = 1}^n \Big[
B_{1, i} \, \tau_i^{\beta_i'}
+B_{2, i} \, \tau_i^{\beta_i}
\Big]
+\sum_{i, j \in \mathcal{F}}
( A_{6, i, j} + B_{1, i} \Sigma^{-1} B_{1, j} )
\, \tau_i^{\beta_i/2} \, \tau_j^{\beta_j/2},
\end{align*}
with \(B_{k, i} = A_{k, i} + A_{k, i}^\top\), \(k = 1, 2\).
First term:
\begin{multline*}
( C - B ) \, \Sigma^{-1}
\sim
\sum_{i = 1}^n \Bigg[
A_{1, i} \left[ ( \tau_i + t_i )^{\beta_i'} - \tau_i^{\beta_i'} \right]
+A_{2, i} \left[ ( \tau_i + t_i )^{\beta_i} - \tau_i^{\beta_i} \right]
+S_{\alpha_i, A_{5, i}} ( t_i )
\Bigg]
\Sigma^{-1}
\\[7pt]
-\sum_{i, j \in \mathcal{F}} B_{1, i} \Sigma^{-1} B_{1, j}
 \Sigma^{-1} \, \tau_i^{\beta_i/2} \, \tau_j^{\beta_j/2}
+\sum_{i, j \in \mathcal{F}} A_{6, i, j} \Sigma^{-1} \Big[
( \tau_i + t_i )^{\beta_i/2} - \tau_i^{\beta_i/2}
\Big] \, \tau_j^{\beta_j/2}.
\end{multline*}
Second term:
\begin{equation*}
C \, \Sigma^{-1} B \, \Sigma^{-1}
\sim
\sum_{i, j \in \mathcal{F}}
\Big[
A_{1, i} \Sigma^{-1} B_{1, j} ( \tau_i + t_i )^{\beta_i/2} \tau_j^{\beta_j/2}
+A_{1, i}^\top \Sigma^{-1} B_{1, j} \tau_i^{\beta_i/2} \, \tau_j^{\beta_j/2}
\Big]
\Sigma^{-1}.
\end{equation*}
Adding the two formulae together, we observe a few cancellations:
\begin{multline}\label{I-minus-R-Sigma-inverse}
I - R ( \bm{\tau} + \bm{t}, \bm{\tau} ) \, \Sigma^{-1} ( \bm{\tau} )
\sim
\sum_{i = 1}^n \Bigg[
A_{1, i} \Big[ ( \tau_i + t_i )^{\beta_i'} - \tau_i^{\beta_i'} \Big]
+A_{2, i} \Big[ ( \tau_i + t_i )^{\beta_i} - \tau_i^{\beta_i} \Big]
+S_{\alpha_i, A_{5, i}} ( t_i )
\Bigg] \Sigma^{-1}
\\[7pt]
+\sum_{i, j \in \mathcal{F}}
\tau_j^{\beta_j/2}
\Big[ ( \tau_i + t_i )^{\beta_i/2} - \tau_i^{\beta_i/2} \Big]
\Big[ A_{6, i, j} + A_{1, i} \Sigma^{-1} B_{1, j} \Big] \Sigma^{-1}.
\end{multline}
Multiplying by \(\bm{b}\) on the right and using \(A_{1, i} \, \bm{w} = \bm{0}\)
yields~\eqref{cond-mean-formula}.
\end{proof}
\subsubsection{Conditional covariance}
\label{sec:orge933d99}
\begin{lemma} \label{lemma:conditional-covariance}
The conditional covariance function
\begin{equation*}
\mathcal{R}_{\bm{\tau}} ( \bm{t}, \bm{s} )
\coloneqq
R ( \bm{\tau} + \bm{t}, \bm{\tau} + \bm{s} )
-R ( \bm{\tau} + \bm{t}, \bm{\tau} ) \,
\Sigma^{-1} ( \bm{\tau} ) \,
R ( \bm{\tau}, \bm{\tau} + \bm{s} )
\end{equation*}
admits the following asymptotic formula
\begin{multline}\label{cond-cov-asympt}
\mathcal{R}_{\bm{\tau}} ( \bm{t}, \bm{s} )
=
\sum_{i = 1}^n \Big[
S_{\alpha_i, A_{5, i}} ( t_i ) + S_{\alpha_i, A_{5, i}} ( -s_i ) - S_{\alpha_i, A_{5, i}} ( t_i - s_i )
\Big]
\\[7pt]
+\sum_{i, j \in \mathcal{F}}
D_{i, j}
\left( ( \tau_i + t_i )^{\beta_i/2} - \tau_i^{\beta_i/2} \right)
\left( ( \tau_j + s_j )^{\beta_j/2} - \tau_j^{\beta_j/2} \right)
+\epsilon ( \bm{\tau}, \bm{t}, \bm{s} )
\end{multline}
with \(D_{i, j} = A_{6, i, j} + A_{1, i} \Sigma^{-1} A_{1, j}^\top\) and the error
\(\epsilon\) satisfying
\begin{equation}\label{cond-cov-error}
\epsilon ( \bm{\tau}, \bm{t}, \bm{s} )
= o \left(
\sum_{i = 1}^n \Big[ t_i^{\beta_i} + s_i^{\beta_i} + \tau_i^{\beta_i}
+|t_i|^{\alpha_i} + |s_i|^{\alpha_i} + | t_i - s_i |^{\alpha_i} \Big]
\right)
\end{equation}
as \(( \bm{\tau}, \bm{t}, \bm{s} ) \to \bm{0}\).
\end{lemma}

\begin{proof}[Proof of Lemma~\ref{lemma:conditional-covariance}]
Let us, for the sake of clarity, denote
\begin{equation}\label{B-and-C-shorthand}
R ( \bm{t}, \bm{s} ) \eqqcolon \Sigma - C ( \bm{t}, \bm{s} ),
\qquad
\Sigma^{-1} ( \bm{\tau} )
\eqqcolon \Sigma^{-1} + \Sigma^{-1} \, B ( \bm{\tau} ) \, \Sigma^{-1}.
\end{equation}
Therefore, we have
\begin{align*}
\notag
\mathcal{R} ( \bm{\tau}, \bm{t}, \bm{s} )
& =
\Sigma - C ( \bm{\tau} + \bm{t}, \bm{\tau} + \bm{s} )
\\[7pt]
-\Big[ \Sigma - C ( \bm{\tau} + \bm{t}, \bm{\tau} ) \Big]
\Big[ \Sigma^{-1} + \Sigma^{-1} \, B ( \bm{\tau} ) \, \Sigma^{-1} \Big]
\Big[ \Sigma - C ( \bm{\tau}, \bm{\tau} + \bm{s} ) \Big]
\span \omit
\\[7pt]
& =
-C ( \bm{\tau} + \bm{t}, \bm{\tau} + \bm{s} )
+C ( \bm{\tau} + \bm{t}, \bm{\tau} )
+C ( \bm{\tau}, \bm{\tau} + \bm{s} )
-B ( \bm{\tau} )
+F
+\epsilon,
\end{align*}
where
\[
F
\coloneqq
B ( \bm{\tau} ) \, \Sigma^{-1} \, C ( \bm{\tau}, \bm{\tau} + \bm{s} )
+C ( \bm{\tau} + \bm{t}, \bm{\tau} ) \, \Sigma^{-1} \, B ( \bm{\tau} )
-C ( \bm{\tau} + \bm{t}, \bm{\tau} ) \, \Sigma^{-1} \, C ( \bm{\tau}, \bm{\tau} + \bm{s} )
\]
and
\begin{equation}\label{CC-error}
\epsilon
\coloneqq
-C ( \bm{\tau} + \bm{t}, \bm{\tau} ) \,
\Sigma^{-1} \,
B ( \bm{\tau} ) \,
\Sigma^{-1} \,
C ( \bm{\tau}, \bm{\tau} + \bm{s} )
= O \left(
\sum_{i = 1}^n \left[
\tau_i^{3 \beta_i'}
+t_i^{3 \beta_i'}
+s_i^{3 \beta_k'}
+| t_i |^{3 \alpha_i}
+| s_i |^{3 \alpha_i}
\right]
\right)
\end{equation}
can be subsumed into~\eqref{cond-cov-error}. Let us substitute \(C\) and
\(B\) with their expressions
\begin{equation}\label{B-and-C-expr}
\begin{aligned}
C ( \bm{t}, \bm{s} )
& \sim
\sum_{i = 1}^n \Big[
A_{1, i} \, t_i^{\beta_i'}
+A_{2, i} \, t_i^{\beta_i}
+A_{1, i}^\top \, s_i^{\beta_i'}
+A_{2, i}^\top \, s_i^{\beta_i}
+S_{\alpha_i, A_{5, i}} ( t_i - s_i )
\Big]
+\sum_{i, j \in \mathcal{F}} A_{6, i, j} \, t_i^{\beta_i/2} \, t_j^{\beta_j/2},
\\[7pt]
B ( \bm{\tau} )
& \sim
\sum_{i = 1}^n \Big[ B_{1, i} \, \tau_i^{\beta_i'} + B_{2, i} \, \tau_i^{\beta_i} \Big]
+\sum_{i, j \in \mathcal{F}} \Big[ A_{6, i, j} + B_{1, i} \Sigma^{-1} B_{1, j} \Big] \,
\tau_i^{\beta_i/2} \, \tau_j^{\beta_j/2}.
\end{aligned}
\end{equation}
Let us first compute the coefficients in front of various powers of \(\tau\),
\(\tau + t\) and \(\tau + s\) in
\begin{equation}\label{top-order}
-C ( \bm{\tau} + \bm{t}, \bm{\tau} + \bm{s} )
+C ( \bm{\tau} + \bm{t}, \bm{\tau} )
+C ( \bm{\tau}, \bm{\tau} + \bm{s} )
-B ( \bm{\tau} ).
\end{equation}
Coefficients of the following terms are zero:
\begin{equation*}
( \tau_i + t_i )^{\beta_i'}, \
( \tau_i + t_i )^{\beta_i}, \
( \tau_i + s_i )^{\beta_i'}, \
( \tau_i + s_i )^{\beta_i}, \
\tau_i^{\beta_i'}, \
\tau_i^{\beta_i}.
\end{equation*}
Next, we proceed to the mixed terms:
\begin{align*}
& ( \tau_i + t_i )^{\beta_i/2} \, ( \tau_j + s_j )^{\beta_j/2} \colon \quad
& -A_{6, i, j},
\qquad
& \tau_i^{\beta_i/2} \, ( \tau_j + s_j )^{\beta_j/2} \colon \quad
& A_{6, i, j},
\\
& ( \tau_i + t_i )^{\beta_i/2} \, \tau_j^{\beta_j/2} \colon \quad
& A_{6, i, j},
\qquad
& \tau_i^{\beta_i/2} \, \tau_j^{\beta_j/2} \colon \quad
& -A_{6, i, j} - B_{1, i} \Sigma^{-1} B_{1, j}
\end{align*}
Let us compare these with the corresponding orders in
\begin{equation*}
F =
B ( \bm{\tau} ) \, \Sigma^{-1} \, C ( \bm{\tau}, \bm{\tau} + \bm{s} )
+C ( \bm{\tau} + \bm{t}, \bm{\tau} ) \, \Sigma^{-1} \, B ( \bm{\tau} )
-C ( \bm{\tau} + \bm{t}, \bm{\tau} ) \, \Sigma^{-1} \, C ( \bm{\tau}, \bm{\tau} + \bm{s} ).
\end{equation*}
We have:
\begin{align*}
& ( \tau_i + t_i )^{\beta_i/2} \, ( \tau_j + s_j )^{\beta_j/2} \colon
& -A_{1, i} \Sigma^{-1} A_{1, j}^\top,
\\[7pt]
& \tau^{\beta_i/2} \, ( \tau_j + s_j )^{\beta_j/2} \colon
& A_{1, i} \Sigma^{-1} A_{1, j}^\top,
\\[7pt]
& ( \tau_i + t_i )^{\beta_i/2} \, \tau_j^{\beta_j/2} \colon
& A_{1, i} \Sigma^{-1} A_{1, j}^\top,
\\[7pt]
& \tau_i^{\beta_i/2} \, \tau_j^{\beta_j/2} \colon
& B_{1, i} \Sigma^{-1} A_{1, j} + A_{1, i}^\top \Sigma^{-1} A_{1, j}^\top.
\end{align*}
Combining these, we find that the aggregate contribution of these terms to
\(\mathcal{R}\) is
\begin{equation}\label{power-contributions}
\sum_{i, j \in \mathcal{F}}
\Big[ A_{6, i, j} + A_{1, i} \Sigma^{-1} A_{1, j}^\top \Big]
\left( ( \tau_i + t_i )^{\beta_i/2} - \tau_i^{\beta_i/2} \right)
\left( ( \tau_j + s_j )^{\beta_j/2} - \tau_j^{\beta_j/2} \right).
\end{equation}
The significant \(S\)-terms appear only from~\eqref{top-order}. Their
combined contribution is
\begin{equation}\label{S-contributions}
\sum_{i = 1}^n \Big[
S_{\alpha_i, A_{5, i}} ( t_i ) + S_{\alpha_i, A_{5, i}} ( -s_i ) - S_{\alpha_i, A_{5, i}} ( t_i - s_i )
\Big].
\end{equation}
The accumulated error is of order
\begin{equation*}
o \left(
\sum_{i = 1}^n \Big[ t_i^{\beta_i} + s_i^{\beta_i} + \tau_i^{\beta_i}
+| t_i |^{\alpha_i} + | s_i |^{\alpha_i} + | t_i - s_i |^{\alpha_i} \Big]
\right).
\end{equation*}
\end{proof}
\subsubsection{Limiting process}
\label{sec:orgc5e60ad}
\begin{lemma}\label{lemma:limiting-process}
Let \(\Lambda > 0\), \(\bm{S} > \bm{0}\) and let
\begin{equation*}
Q_u
= \left\{
\bm{\tau} \in \mathbb{R}^n_{ + } \colon
\tau_i = 0
\ \text{if} \  i \in \mathcal{I}^c
\ \text{and} \
\tau_i \leq u^{2/\alpha_i - 2/\nu_i} \Lambda_i / S_i
\ \text{if} \ i \in \mathcal{I}
\right\}.
\end{equation*}
Then, the rescaled by \(u^{-2/\bm{\nu}}\) contitional mean vector
\(\bm{d}_{u, \bm{\tau}}\) converges uniformly
\begin{equation*}
\lim_{u \to \infty}
\sup_{\bm{\tau} \in Q_u}
\sup_{\bm{t} \in [\bm{0}, \bm{S}]}
\left|
u^2 \, \bm{d}_{u, \bm{\tau}} \left( \bm{t} \right)
-\bm{d} ( \bm{t} )
\right|
= 0
\end{equation*}
to the following limit
\begin{equation}\label{lim-mean}
\bm{d} ( \bm{t} )
=
\sum_{i \in \mathcal{I} \cup \mathcal{J}}
S_{\alpha_i, A_{5, i}} ( t_i ) \, \bm{w}
+\sum_{i \in \mathcal{J} \cup \mathcal{K}} A_{2, i} \, \bm{w} \, t_i^{\beta_i}.
\end{equation}
The rescaled conditional covariance matrix \(\mathcal{R}_{u, \bm{\tau}}\) converges
uniformly
\begin{equation*}
\lim_{u \to \infty}
\sup_{\bm{\tau} \in Q_u}
\sup_{\bm{t}, \, \bm{s} \in [\bm{0}, \bm{S}]}
\left\|
u^2 \, \mathcal{R}_{u, \bm{\tau}} \left( \bm{t}, \bm{s} \right)
-\mathcal{R} ( \bm{t}, \bm{s} )
\right\|
= 0
\end{equation*}
to the following limit:
\begin{equation}\label{lim-cov}
\mathcal{R} ( \bm{t}, \bm{s} )
=
\sum_{i \in \mathcal{I} \cup \mathcal{J}} R_{\alpha_i, A_{5, i}} ( t_i, s_i )
+\sum_{i, j \in ( \mathcal{J} \cup \mathcal{K} ) \cap \mathcal{F}}
D_{i, j}
\, t_i^{\beta_i/2} \, s_j^{\beta_j/2},
\end{equation}
where \(R_{\alpha, V}\) is defined in~\eqref{def:R-and-S} and \(D_{i, j}\)
in~\eqref{XI-alternative}.
\end{lemma}

Assuming~\ref{A2.6} and using Lemma~\ref{lemma:quadratic-covariance}
we find that the covariance~\eqref{lim-cov} corresponds to the following
Gaussian random field:
\begin{equation*}
\bm{Y} ( \bm{t} )
=
\sum_{i \in \mathcal{I} \cup \mathcal{J}} \bm{Y}_{\alpha_i, A_{5, i}} ( t_i )
+\sum_{k \in \mathcal{F}} C_k ( \bm{t} ) \, \bm{\mathcal{N}}_k,
\qquad
C_k ( \bm{t} ) = \sum_{i \in \mathcal{F}} C_{i, k} \, t_i^{\beta_i/2},
\end{equation*}
where
\begin{enumerate}
\item \(\bm{Y}_{\alpha_i, A_{5, i}}\) are operator fractional Brownian motions
associated to \(R_{\alpha_i, A_{5, i}}\), independent of each other,
\item \(\bm{\mathcal{N}}_k\) are standard Gaussian vectors, independent of each other and
of \(\bm{Y}_{\alpha_i, A_{5, i}}\),
\item \(( C_{i, k} )_{i, k \in \mathcal{F}}\), is a family of \(d \times d\) matrices
satisfying~\eqref{A2.6-D-as-sum}, whose existence is guaranteed
by~\ref{A2.6} and Lemma~\ref{lemma:quadratic-covariance}.
\end{enumerate}
\subsection{Calculations from the log-layer lemma}
\label{sec:orgb6b584f}
\label{LL:calculations}
\begin{proof}[Proof of~\eqref{LL:G-plus-sigma2}]
By the definition of \(\bm{\chi}_{u, \bm{k}}\) as the condional process
\begin{equation*}
\bm{\chi}_{u, \bm{k}} ( \bm{t} )
=
u \,
\Big(
\bm{X}_{u, \bm{k}} ( \bm{t} ) - \bm{X}_{u, \bm{k}} ( \bm{0} )
\ \Big| \
\bm{X}_{u, \bm{k}} ( \bm{0} ) = u \bm{b} - \frac{\bm{x}}{u}
\Big)
\end{equation*}
we have
\begin{align}\label{eq:4}
\mathbb{E} \left\{ \bm{\chi}_{u, \bm{k}} ( \bm{t} ) \right\}
& =
-u \Big[
I - R_u ( \bm{\Lambda} \bm{k} + \bm{t}, \bm{\Lambda} \bm{k} ) \,
\Sigma_u^{-1} ( \bm{\Lambda} \bm{k} )
\Big]
\left( u \, \bm{b} - \frac{\bm{x}}{u} \right)
\\[7pt]
& =
-u^2 \, \bm{d}_{u, \bm{\Lambda} \bm{k}} ( \bm{t} )
+\Big[
I - R_u ( \bm{\Lambda} \bm{k} + \bm{t}, \bm{\Lambda} \bm{k} ) \,
\Sigma_u^{-1} ( \bm{\Lambda} \bm{k} )
\Big] \bm{x}.
\end{align}
By~\ref{I-minus-R-Sigma-inverse}, we have that for every \(\varepsilon > 0\)
exists \(\delta > 0\) such that if \(| \bm{\tau} |, \, | \bm{t} | < \delta\), then
\begin{equation*}
\left\|
I - R ( \bm{\tau} + \bm{t}, \bm{\tau} ) \,
\Sigma^{-1} ( \bm{\tau} )
\right\|
\leq \varepsilon
\end{equation*}
Setting
\(\bm{\tau} \leadsto u^{-2/\bm{\nu}} \bm{\Lambda} \bm{k}\) and \(\bm{t} \leadsto u^{-2/\bm{\nu}} \bm{t}\),
with \(\bm{k}\) now belonging to \(Q_u\) defined by~\eqref{LL:1},
and noting that both belong to a shrinking,
as \(u \to \infty\), vicinity of zero, we obtain the following result: for every
\(\varepsilon > 0\) there exists \(u_0\) such that for all \(u \geq u_0\) holds
\begin{equation*}
\left|
-u \Big[
I - R ( \bm{\tau} + \bm{t}, \bm{\tau} ) \,
\Sigma^{-1} ( \bm{\tau} )
\Big] \, \frac{\bm{x}}{u}
\right|
\leq
\varepsilon \, \sum_{j = 1}^d | x_j |
\end{equation*}
Next, we want to study the other term of~\eqref{eq:4}.
By~\eqref{cond-mean-formula},
\begin{multline*}
\bm{d}_{\, \bm{\tau}} ( \bm{t} )
\sim
\sum_{i = 1}^n \Bigg[
A_{2, i} \Big[ ( \tau_i + t_i )^{\beta_i} - \tau_i^{\beta_i} \Big]
+S_{\alpha_i, A_{5, i}} ( t_i )
\Bigg]
\bm{w}
\\[7pt]
+\sum_{i, j \in \mathcal{F}}
\tau_j^{\beta_j/2} \,
\Big[ ( \tau_i + t_i )^{\beta_i/2} - \tau_i^{\beta_i/2} \Big]
\Big[ A_{6, i, j} + A_{1, i} \Sigma^{-1} A_{1, j}^\top \Big] \bm{w}
+\epsilon ( \bm{\tau}, \bm{t} )
\end{multline*}
with \(\epsilon\) satisfying the following condition:
\begin{equation}\label{LL:epsilon}
\begin{gathered}
\text{for every} \ \varepsilon > 0, \ \text{there exists} \ \delta > 0 \
\text{such that if} \ | \bm{\tau} |, \, | \bm{t} | < \delta, \ \text{then}
\\
\left| \epsilon ( \bm{\tau}, \bm{t} ) \right|
\leq \varepsilon \sum_{i = 1}^n \Big[ \tau_i^{\beta_i} + t_i^{\beta_i} + | t_i |^{\alpha_i} \Big].
\end{gathered}
\end{equation}
Bounding the mixed terms sum by
\begin{equation*}
2 \, \tau_j^{\beta_j/2} \, \Big[ ( \tau_i + t_i )^{\beta_i/2} - \tau_i^{\beta_i/2} \Big]
\leq \varepsilon \, \tau_j^{\beta_j} + \varepsilon^{-1} \, \Big[ ( \tau_i + t_i )^{\beta_i/2} - \tau_i^{\beta_i/2} \Big]^2
\end{equation*}
with the same \(\varepsilon\) as above, we obtain the following inequality:
\begin{equation*}
\left| \bm{d}_{\, \bm{\tau}} ( \bm{t} ) \right|
\leq
c \sum_{i = 1}^n \left[
\Big[ ( \tau_i + t_i )^{\beta_i} - \tau_i^{\beta_i} \Big]
+\varepsilon^{-1} \, \Big[ ( \tau_i + t_i )^{\beta_i/2} - \tau_i^{\beta_i/2} \Big]^2
+| t_i |^{\alpha_i}
\right]
+\varepsilon \sum_{i = 1}^n \Big[ \tau_i^{\beta_i} + t_i^{\beta_i} \Big].
\end{equation*}
Set again
\(\bm{\tau} \leadsto u^{-2/\bm{\nu}} \, \bm{\Lambda} \bm{k}\) and \(\bm{t} \leadsto u^{-2/\bm{\nu}} \, \bm{t}\).
New \(\bm{k}\) and \(\bm{t}\) belong to \(Q_u\) and \([\bm{0}, \bm{\Lambda}]\)
correspondingly.

Using the following inequality
\begin{equation}\label{eq:ineq1}
( x + 1 )^\beta - x^\beta \leq c \, ( x \vee 1 )^{\beta-1}
\quad \text{if} \quad x \geq 0,
\end{equation}
which is valid with some constant \(c > 0\), we obtain that
\begin{equation*}
( \Lambda_i \, k_i + t_i )^{\zeta} - ( \Lambda_i \, k_i )^{\zeta}
\leq
\Big[ ( \Lambda_i \, k_i + \Lambda_i )^{\zeta} - ( \Lambda_i \, k_i )^{\zeta} \Big]
\leq
c_2 \, \Lambda_i^{\zeta} \, ( k_i \vee 1 )^{\zeta - 1}
\end{equation*}
for \(\zeta = \beta_i\) or \(\zeta = \beta_i / 2\). Therefore, the first condition
of~\eqref{eq:log-layer-G-and-sigma} holds with
\begin{equation*}
G \coloneqq
c_3 \sum_{i = 1}^n \left[
u^{2 - 2 \beta_i / \nu_i} \Lambda_i^{\beta_i}
\Big[
( k_i \vee 1 )^{\beta_i - 1}
+\varepsilon^{-1} \, ( k_i \vee 1 )^{\beta_i - 2}
+\varepsilon \left( k_i^{\beta_i} + 1 \right)
\Big]
+u^{2 - 2 \alpha_i / \nu_i} \Lambda_i^{\alpha_i}
\right]
\end{equation*}
where \(\varepsilon\) can be made as small as required by taking \(u\) large enough.
\smallskip

Next, we seek for a bound of the variance of
\(\bm{w}_F^\top \, \bm{\chi}_{u, \bm{k}, F} ( \bm{t} )\).
We have:
\begin{equation*}
\var \left\{ \bm{w}_F^\top \, \bm{\chi}_{u, \bm{k}, F} ( \bm{t} ) \right\}
=
\var \left\{ \sum_{j \in F} w_j \chi_{u, \bm{k}, j} ( \bm{t} ) \right\}
=
\sum_{j \in F} w_j^2 \left[ K_{u, \bm{k}} ( \bm{t}, \bm{t} ) \right]_{jj},
\end{equation*}
where \(K_{u, \bm{k}}\) is the covariance of \(\bm{\chi}_{u, \bm{k}} ( \bm{t} )\):
\begin{equation*}
K_{u, \bm{k}} ( \bm{t}, \bm{s} )
=
\mathbb{E} \left\{
\Big[
\bm{\chi}_{u, \bm{k}} ( \bm{t} )
- \bm{d}_{u, \bm{k}} ( \bm{t} )
\Big]
{\Big[
\bm{\chi}_{u, \bm{k}} ( \bm{s} )
- \bm{d}_{u, \bm{k}} ( \bm{s} )
\Big]}^\top
\right\}
=
u^2 \, \mathcal{R}_{u, \bm{\Lambda} \bm{k}} ( \bm{t}, \bm{s} ).
\end{equation*}
By Lemma~\ref{lemma:conditional-covariance}, we have that
\begin{multline*}
\mathcal{R}_{\bm{\tau}} ( \bm{t}, \bm{t} )
=
\sum_{i = 1}^n \Big[
S_{\alpha_i, A_{5, i}} ( t_i ) + S_{\alpha_i, A_{5, i}} ( -t_i )
\Big]
\\[7pt]
+\sum_{i, j \in \mathcal{F}}
\Big[ A_{6, i, j} + A_{1, i} \Sigma^{-1} A_{1, j}^\top \Big]
\left( ( \tau_i + t_i )^{\beta_i/2} - \tau_i^{\beta_i/2} \right)
\left( ( \tau_j + t_j )^{\beta_j/2} - \tau_j^{\beta_j/2} \right)
+\epsilon ( \bm{\tau}, \bm{t} )
\end{multline*}
with \(\epsilon\) satisfying~\eqref{LL:epsilon}. Setting
\(\bm{\tau} \leadsto u^{-2/\bm{\nu}} \bm{\Lambda} \bm{k}\) and \(\bm{t} \leadsto u^{-2/\bm{\nu}} \bm{t}\),
and using again~\eqref{eq:ineq1}, we find that the second condition
of~\eqref{eq:log-layer-G-and-sigma} holds with
\begin{equation*}
\sigma^2 \coloneqq
c_4 \sum_{i = 1}^n
\left[
u^{2-2\beta_i/\nu_i} \, \Lambda_i^{\beta_i}
\Big[ (k_i \vee 1)^{\beta_i - 2} + \varepsilon \left( k_i^{\beta_i} + 1 \right) \Big]
+u^{2 - 2\alpha_i / \nu_i}
\Lambda_i^{\alpha_i}
\right],
\end{equation*}
where \(c_4 > 0\) and \(\varepsilon\) can be made arbitrarily small by taking \(u\)
to be large enough. Adding the two bounds together
yields
\begin{equation*}
G + \sigma^2 =
c_5 \sum_{i = 1}^n \Bigg[
u^{2 - 2 \beta_i / \nu_i} \Lambda_i^{\beta_i}
\Big[
( k_i \vee 1 )^{\beta_i - 1}
+\varepsilon^{-1} \, ( k_i \vee 1 )^{\beta_i - 2}
+\varepsilon \left( k_i^{\beta_i} + 1 \right)
\Big]
+u^{2 - 2 \alpha_i / \nu_i} \Lambda_i^{\alpha_i}
\Bigg].
\end{equation*}
\end{proof}
\subsection{Calculations from the double sum lemma}
\label{sec:org88e9c48}
\label{DS:calculations}
The covariance function of the field
\(( \bm{X} ( \bm{t} ) + \bm{X} ( \bm{s} ) ) / 2\)
is given by:
\begin{align}
\notag
R ( \bm{t}_1, \bm{s}_1, \bm{t}_2, \bm{s}_2 )
& =
\frac{1}{4} \,
\mathbb{E} \left\{
\Big[ \bm{X} ( \bm{t}_1 ) + \bm{X} ( \bm{s}_1 ) \Big]
\Big[ \bm{X} ( \bm{t}_2 ) + \bm{X} ( \bm{s}_2 ) \Big]^\top
\right\}
\\[7pt]
\label{DS-cov}
& =
\frac{1}{4} \Big[
R ( \bm{t}_1, \bm{t}_2 )
+R ( \bm{t}_1, \bm{s}_2 )
+R ( \bm{s}_1, \bm{t}_2 )
+R ( \bm{s}_1, \bm{s}_2 )
\Big].
\end{align}
\subsubsection{Inverse of Sigma}
\label{sec:orgd9696ce}
By~\eqref{DS-cov}, we have
\begin{equation*}
\Sigma ( \bm{t}, \bm{s} )
= \frac{1}{4} \Big[
\Sigma ( \bm{t} ) + \Sigma ( \bm{s} ) + R ( \bm{t}, \bm{s} ) + R ( \bm{s}, \bm{t} )
\Big].
\end{equation*}
Using~\eqref{B-and-C-shorthand}, we can rewrite it as follows:
\begin{equation*}
\Sigma ( \bm{t}, \bm{s} )
=
\Sigma -
\frac{1}{4}
\Big[
C ( \bm{t}, \bm{t} )
+C ( \bm{s}, \bm{s} )
+C ( \bm{t}, \bm{s} )
+C ( \bm{s}, \bm{t} )
\Big]
\eqqcolon
\Sigma - B ( \bm{t}, \bm{s} ),
\end{equation*}
where we have introduced one more shorthand \(B ( \bm{t}, \bm{s} )\).
Similarly to what we did in Lemma~\ref{lemma:Sigma-inverse},
\begin{equation}\label{DS-Sigma-inverse-1}
\Sigma^{-1} ( \bm{t}, \bm{s} ) - \Sigma^{-1}
=
\Sigma^{-1}
B ( \bm{t}, \bm{s} ) \,
\Sigma^{-1}
+\Sigma^{-1}
B ( \bm{t}, \bm{s} ) \,
\Sigma^{-1}
B ( \bm{t}, \bm{s} ) \,
\Sigma^{-1}
+\epsilon,
\end{equation}
where, similarly to~\eqref{CC-error}, \(\epsilon\) can be shown to satisfy
\begin{equation*}
\epsilon
= O \left(
\sum_{i = 1}^n \Big[
t_i^{3 \beta_i} + s_i^{3 \beta_i} + | t_i - s_i |^{3 \alpha_i}
\Big]
\right).
\end{equation*}
Using~\eqref{B-and-C-expr}, we can find an expression for
\(B ( \bm{t}, \bm{s} )\):
\begin{equation}\label{DS-B}
\begin{aligned}
B ( \bm{t}, \bm{s} )
& \sim
\frac{1}{4}
\sum_{i = 1}^n \Big[
2 \, B_{1, i} \, t_i^{\beta_i'}
+2 \, B_{2, i} \, t_i^{\beta_i}
+2 \, B_{1, i} \, s_i^{\beta_i'}
+2 \, B_{2, i} \, s_i^{\beta_i}
+S_{\alpha_i, A_{5, i}} ( t_i - s_i )
+S_{\alpha_i, A_{5, i}} ( s_i - t_i )
\Big]
\\[7pt]
& \hspace{10pt}
+\frac{1}{4} \, \sum_{i, j \in \mathcal{F}}
A_{6, i, j}
\Big[
t_i^{\beta_i/2} \, t_j^{\beta_j/2}
+s_i^{\beta_i/2} \, s_j^{\beta_j/2}
+t_i^{\beta_i/2} \, s_j^{\beta_j/2}
+s_i^{\beta_i/2} \, t_j^{\beta_j/2}
\Big]
\end{aligned}
\end{equation}
with \(B_{k, i} = A_{k, i} + A_{k, i}^\top\), \(k = 1, 2\) and within the same
magnitude of error.

The quadratic term of~\eqref{DS-Sigma-inverse-1} reads:
\begin{equation*}
B ( \bm{t}, \bm{s} ) \, \Sigma^{-1} B ( \bm{t}, \bm{s} )
\sim
\frac{1}{4} \,
\sum_{i, j \in \mathcal{F}}
B_{1, i} \, \Sigma^{-1} B_{1, j}
\Big[
t_i^{\beta_i/2} \, t_j^{\beta_j/2}
+s_i^{\beta_i/2} \, s_j^{\beta_j/2}
+t_i^{\beta_i/2} \, s_j^{\beta_j/2}
+s_i^{\beta_i} \, t_j^{\beta_j/2}
\Big].
\end{equation*}
Finally, we arrive at
\begin{align*}
\Sigma^{-1} ( \bm{t}, \bm{s} ) - \Sigma^{-1}
& \sim
\frac{1}{4} \,
\Sigma^{-1} \,
\sum_{i = 1}^n \Big[
2 \, B_{1, i} \, t_i^{\beta_i'}
+2 \, B_{2, i} \, t_i^{\beta_i}
+2 \, B_{1, i} \, s_i^{\beta_i'}
+2 \, B_{2, i} \, s_i^{\beta_i}
+S_{\alpha_i, A_{5, i}} ( t_i - s_i )
+S_{\alpha_i, A_{5, i}} ( s_i - t_i )
\Big]
\Sigma^{-1}
\\[7pt]
& \hspace{10pt}
+\frac{1}{4} \,
\sum_{i, j \in \mathcal{F}}
\Sigma^{-1}
\Big[ A_{6, i, j} + B_{1, i} \, \Sigma^{-1} B_{1, j} \Big] \,
\Sigma^{-1}
\Big[
t_i^{\beta_i/2} \, t_j^{\beta_j/2}
+s_i^{\beta_i/2} \, s_j^{\beta_j/2}
+t_i^{\beta_i/2} \, s_j^{\beta_j/2}
+s_i^{\beta_i/2} \, t_j^{\beta_j/2}
\Big].
\end{align*}
Rewriting it in terms of
\(\widetilde{D}_{i, j} = A_{6, i, j} + B_{1, j} \, \Sigma^{-1} B_{1, j}\)
at
\begin{equation}\label{DS-Sigma-inverse}
\begin{aligned}
\Sigma^{-1} ( \bm{t}, \bm{s} ) - \Sigma^{-1}
& =
\frac{1}{4} \,
\Sigma^{-1} \,
\sum_{i = 1}^n
\Big[
2 \, B_{1, i} \, t_i^{\beta_i'}
+2 \, B_{2, i} \, t_i^{\beta_i}
+2 \, B_{1, i} \, s_i^{\beta_i'}
+2 \, B_{2, i} \, s_i^{\beta_i}
\\
+S_{\alpha_i, A_{5, i}} ( t_i - s_i )
+S_{\alpha_i, A_{5, i}} ( s_i - t_i )
\Big]
\Sigma^{-1}
\span \omit
\\[7pt]
& \hspace{10pt}
+\frac{1}{4} \,
\Sigma^{-1} \,
\sum_{i, j \in \mathcal{F}}
\widetilde{D}_{i, j} \,
\Big[
t_i^{\beta_i/2} \, s_j^{\beta_j/2}
+s_i^{\beta_i/2} \, t_j^{\beta_j/2}
+t_i^{\beta_i/2} \, t_j^{\beta_j/2}
+s_i^{\beta_i/2} \, s_j^{\beta_j/2}
\Big]
\Sigma^{-1}
+\epsilon.
\end{aligned}
\end{equation}
The error term \(\epsilon\) satisfies
\begin{equation}\label{DS-Sigma-error}
\epsilon = o \left(
\sum_{i = 1}^n \Big[
t_i^{\beta_i} + s_i^{\beta_i} + | t_i - s_i |^{\alpha_i}
\Big]
\right).
\end{equation}
\subsubsection{Exponential prefactor}
\label{sec:orgad72226}
Multiplying~\eqref{DS-Sigma-inverse} by \(\bm{b}\) on both sides, and
using the fact that by~\ref{A2.3}
\begin{equation*}
\bm{w}^\top B_{1, i} \, \bm{w} = 0, \quad
\bm{w}^\top B_{2, i} \, \bm{w} = 2 \, \bm{w}^\top A_{2, i} \, \bm{w}, \quad
\bm{w}^\top \widetilde{D}_{i, j} \, \bm{w}
= \Xi_{i, j},
\end{equation*}
where \(\Xi_{i, j}\) is defined in~\eqref{XI-alternative}, and also
\begin{equation*}
\bm{w}^\top S_{\alpha, A} ( t ) \, \bm{w}
= \bm{w}^\top A \, \bm{w} \, \mathbb{1}_{t \geq 0} \, | t |^{\alpha}
+\underbracket[.1pt]{\bm{w}^\top A^\top \, \bm{w}}_{= \, \bm{w}^\top A \, \bm{w}} \,
\mathbb{1}_{t < 0} \, | t |^{\alpha}
= \bm{w}^\top A \, \bm{w}  \, | t |^{\alpha},
\end{equation*}
we obtain
\begin{multline}\label{DS-prefactor}
\bm{b}^\top \Big[ \Sigma^{-1} ( \bm{\tau}, \bm{\lambda} ) - \Sigma^{-1} \Big] \bm{b}
\sim
\sum_{i = 1}^n
\left[
\bm{w}^\top A_{2, i} \, \bm{w} \, \Big[ \tau_i^{\beta_i} + \lambda_i^{\beta_i} \Big]
+\frac{\bm{w}^\top A_{5, i} \, \bm{w}}{2} \,
\left| \lambda_i - \tau_i \right|^{\alpha_i}
\right]
\\[7pt]
+\frac{1}{4} \,
\sum_{i, j \in \mathcal{F}}
\Xi_{i, j}
\Big[
\tau_i^{\beta_i/2} \, \lambda_j^{\beta_j/2}
+\lambda_i^{\beta_i/2} \, \tau_j^{\beta_j/2}
+\tau_i^{\beta_i/2} \, \tau_j^{\beta_j/2}
+\lambda_i^{\beta_i/2} \, \lambda_j^{\beta_j/2}
\Big]
\end{multline}
with the same error as in~\eqref{DS-Sigma-error}.
\subsubsection{Conditional mean: formula}
\label{sec:org21ab9ae}
The conditional mean vector \(\bm{d}_{\bm{\tau}, \bm{\lambda}} ( \bm{t}, \bm{s} )\)
of the field
\(( \bm{X} ( \bm{t} ) + \bm{X} ( \bm{s} ) ) / 2\)
is given by
\begin{equation*}
\bm{d}_{\bm{\tau}, \bm{\lambda}} ( \bm{t}, \bm{s} )
= \Big[
I - R ( \bm{\tau} + \bm{t}, \bm{\lambda} + \bm{s}, \bm{\tau}, \bm{\lambda} ) \,
\Sigma^{-1} ( \bm{\tau}, \bm{\lambda} )
\Big] \bm{b},
\end{equation*}
where \(R ( \bm{t}_1, \bm{s}_1, \bm{t}_2, \bm{s}_2 )\) is defined
by~\eqref{DS-cov}.

As in the proof of~\ref{lemma:cond-mean}, to simplify the computations
define \(C\) and \(B\) by
\begin{equation*}
R ( \bm{\tau} + \bm{t}, \bm{\lambda} + \bm{s}, \bm{\tau}, \bm{\lambda} )
\eqqcolon \Sigma - C,
\qquad
\Sigma^{-1} ( \bm{\tau}, \bm{\lambda} )
\eqqcolon \Sigma^{-1} + \Sigma^{-1} \, B \, \Sigma^{-1}.
\end{equation*}
With these shorthands, we have
\begin{equation}\label{DS-I-minus-R-Sigma-inv}
I - R ( \bm{\tau} + \bm{t}, \bm{\lambda} + \bm{s}, \bm{\tau}, \bm{\lambda} ) \,
\Sigma^{-1} ( \bm{\tau}, \bm{\lambda} )
= I - \Big[ \Sigma - C \Big] \Big[ \Sigma^{-1} + \Sigma^{-1} B \, \Sigma^{-1} \Big]
= ( C - B ) \, \Sigma^{-1} + C \, \Sigma^{-1} B \, \Sigma^{-1}.
\end{equation}
By~\eqref{DS-Sigma-inverse-1},
\begin{equation*}
C = \frac{1}{4} \Big[
C ( \bm{\tau} + \bm{t}, \bm{\tau} )
+C ( \bm{\tau} + \bm{t}, \bm{\lambda} )
+C ( \bm{\lambda} + \bm{s}, \bm{\tau} )
+C ( \bm{\lambda} + \bm{s}, \bm{\lambda} )
\Big]
\end{equation*}
with
\begin{equation*}
C ( \bm{t}, \bm{s} )
\sim
\sum_{i = 1}^n \Big[
A_{1, i} \, t^{\beta_i'}
+A_{2, i} \, t^{\beta_i}
+A_{1, i}^\top \, s_i^{\beta_i'}
+A_{2, i}^\top \, s_i^{\beta_i}
+S_{\alpha_i, A_{5, i}} ( t_i - s_i )
\Big]
+\sum_{i, \, j \in \mathcal{F}} A_{6, i, j} \, t_i^{\beta_i / 2} \, s_j^{\beta_j / 2},
\end{equation*}
and
\begin{multline*}
B \sim
\frac{1}{4} \,
\sum_{i = 1}^n
\Big[
2 \, B_{1, i} \, \tau_i^{\beta_i'}
+2 \, B_{2, i} \, \tau_i^{\beta_i}
+2 \, B_{1, i} \, \lambda_i^{\beta_i'}
+2 \, B_{2, i} \, \lambda_i^{\beta_i}
+S_{\alpha_i, A_{5, i}} ( \tau_i - \lambda_i )
+S_{\alpha_i, A_{5, i}} ( \tau_i - \lambda_i )
\Big]
\\[7pt]
+\frac{1}{4} \,
\sum_{i, j \in \mathcal{F}}
\widetilde{D}_{i, j} \,
\Big[
\tau_i^{\beta_i/2} \, \tau_j^{\beta_j/2}
+\tau_i^{\beta_i/2} \, \lambda_j^{\beta_j/2}
+\lambda_i^{\beta_i/2} \, \tau_j^{\beta_j/2}
+\lambda_i^{\beta_i/2} \, \lambda_j^{\beta_j/2}
\Big]
\end{multline*}
Let us calculate the leading order coefficients
in~\eqref{DS-I-minus-R-Sigma-inv}. First, the \(S\)-type contributions
are
\begin{multline}\label{DS-S-contrib}
F_1 \coloneqq
\frac{1}{4}
\sum_{i = 1}^n
\Big[
2 \, S_{\alpha_i, A_{5, i}} ( t_i )
+2 \, S_{\alpha_i, A_{5, i}} ( s_i )
+S_{\alpha_i, A_{5, i}} ( \tau_i + t_i - \lambda_i )
\\[7pt]
+S_{\alpha_i, A_{5, i}} ( \lambda_i + s_i - \tau_i )
-S_{\alpha_i, A_{5, i}} ( \tau_i - \lambda_i )
-S_{\alpha_i, A_{5, i}} ( \lambda_i - \tau_i )
\Big]
\Sigma^{-1}.
\end{multline}
Next, the leading power-type orders give
\begin{gather*}
( \tau_i + t_i )^{\beta_i'} - \tau_i^{\beta_i'}, \
( \lambda_i + s_i )^{\beta_i'} - \lambda_i^{\beta_i} \colon
\quad
\frac{1}{2} \, A_{1, i},
\\[7pt]
( \tau_i + t_i )^{\beta_i} - \tau_i^{\beta_i}, \
( \lambda_i + s_i )^{\beta_i} - \lambda_i^{\beta_i} \colon
\quad
\frac{1}{2} \, A_{2, i}.
\end{gather*}
It remains to compute the mixed terms: from \(C - D\):
\begin{equation}\label{eq:3}
\begin{gathered}
( \tau_i + t_i )^{\beta_i/2} \, \tau_j^{\beta_j/2}, \quad
( \lambda_i + s_i )^{\beta_i/2} \, \lambda_j^{\beta_j/2},
\\
( \tau_i + t_i )^{\beta_i/2} \, \lambda_j^{\beta_j/2}, \quad
( \lambda_i + s_i )^{\beta_i/2} \, \tau_j^{\beta_j/2}
\end{gathered}
\colon \quad
\frac{1}{4} \, A_{6, i, j}
\end{equation}
and
\begin{equation}\label{eq:6}
\begin{gathered}
\tau_i^{\beta_i/2} \, \tau_j^{\beta_j/2}, \quad
\tau_i^{\beta_i/2} \, \lambda_j^{\beta_j/2},
\\
\lambda_i^{\beta_i/2} \, \tau_j^{\beta_j/2}, \quad
\lambda_i^{\beta_i/2} \, \lambda_j^{\beta_j/2}
\end{gathered}
\colon \quad
-\frac{1}{4} \, \widetilde{D}_{i, j}.
\end{equation}
From \(C \, \Sigma^{-1} B\):
\begin{equation}\label{eq:5}
\begin{gathered}
( \tau_i + t_i )^{\beta_i/2} \, \tau_j^{\beta_j/2}, \quad
( \tau_i + t_i )^{\beta_i/2} \, \lambda_j^{\beta_j/2},
\\
( \lambda_i + s_i )^{\beta_i/2} \, \tau_j^{\beta_j/2}, \quad
( \lambda_i + s_i )^{\beta_i/2} \, \lambda_j^{\beta_j/2}
\end{gathered}
\colon \quad
\frac{1}{4} \,
A_{1, i} \, \Sigma^{-1} B_{1, j}
\end{equation}
and
\begin{equation}\label{eq:7}
\begin{gathered}
\tau_i^{\beta_i/2} \, \tau_j^{\beta_j/2}, \quad
\tau_i^{\beta_i/2} \, \lambda_j^{\beta_j/2},
\\
\lambda_i^{\beta_i/2} \, \tau_j^{\beta_j/2}, \quad
\lambda_i^{\beta_i/2}, \, \lambda_j^{\beta_j/2}
\end{gathered}
\colon \quad
\frac{1}{4} \, A_{1, i}^\top \, \Sigma^{-1} B_{1, j}.
\end{equation}
Combining~\eqref{eq:3} and~\eqref{eq:5} gives
\begin{equation*}
\begin{gathered}
( \tau_i + t_i )^{\beta_i/2} \, \tau_j^{\beta_j/2}, \quad
( \tau_i + t_i )^{\beta_i/2} \, \lambda_j^{\beta_j/2},
\\
( \lambda_i + s_i )^{\beta_i/2} \, \tau_j^{\beta_j/2}, \quad
( \lambda_i + s_i )^{\beta_i/2} \, \lambda_j^{\beta_j/2}
\end{gathered}
\colon \quad
\frac{1}{4} \, A_{6, i, j}
+\frac{1}{4} \, A_{1, i} \, \Sigma^{-1} B_{1, j}.
\end{equation*}
Similarly for~\eqref{eq:6} and~\eqref{eq:7}:
\begin{equation*}
\begin{gathered}
\tau_i^{\beta_i/2} \, \tau_j^{\beta_j/2}, \quad
\tau_i^{\beta_i/2} \, \lambda_j^{\beta_j/2},
\\
\lambda_i^{\beta_i/2} \, \tau_j^{\beta_j/2}, \quad
\lambda_i^{\beta_i/2}, \, \lambda_j^{\beta_j/2}
\end{gathered}
\colon \quad
-\frac{1}{4} \, A_{6, i, j} -\frac{1}{4} \, A_{1, i} \, \Sigma^{-1} B_{1, j},
\end{equation*}
where we have used
\begin{equation*}
-\widetilde{D}_{i, j} + A_{1, i}^\top \, \Sigma^{-1} B_{1, j}
= -A_{6, i, j} -B_{1, i} \, \Sigma^{-1} B_{1, j} + A_{1, i}^\top \, \Sigma^{-1} B_{1, j}
= -A_{6, i, j} -A_{1, i} \, \Sigma^{-1} B_{1, j}
\end{equation*}
The aggregate power-type contribution is
\begin{equation}\label{DS-F2}
\begin{aligned}
F_2
& \coloneqq
\frac{1}{2} \,
\sum_{i = 1}^n
\Bigg[
A_{1, i}
\left[
\Big( ( \tau_i + t_i )^{\beta_i'} - \tau_i^{\beta_i'} \Big)
+\Big( ( \lambda_i + s_i )^{\beta_i'} - \lambda_i^{\beta_i'} \Big)
\right]
\\[7pt]
+A_{2, i}
\left[
\Big( ( \tau_i + t_i )^{\beta_i} - \tau_i^{\beta_i} \Big)
+\Big( ( \lambda_i + s_i )^{\beta_i} - \lambda_i^{\beta_i} \Big)
\right]
\Bigg]
\Sigma^{-1}
\span \omit
\\[7pt]
& \hspace{10pt}
+\frac{1}{4} \sum_{i, j \in \mathcal{F}}
\Big[ A_{6, i, j} + A_{1, i} \, \Sigma^{-1} B_{1, j} \Big] \Sigma^{-1}
\left( \tau_j^{\beta_j/2} + \lambda_j^{\beta_j/2} \right)
\\[7pt]
\times
\Bigg[
\left( ( \tau_i + t_i )^{\beta_i/2} - \tau_i^{\beta_i/2} \right) \,
+\left( ( \lambda_i + s_i )^{\beta_i/2} - \lambda_i^{\beta_i/2} \right) \,
\Bigg]
\span \omit
\end{aligned}
\end{equation}
We have thus shown that
\begin{equation}\label{DS:I-minus-R-Sigma-inv}
I - R ( \bm{\tau} + \bm{t}, \bm{\lambda} + \bm{s}, \bm{\tau}, \bm{\lambda} ) \,
\Sigma^{-1} ( \bm{\tau}, \bm{\lambda} )
= F_1 + F_2 + \epsilon
\end{equation}
with \(F_2\) defined in~\eqref{DS-F2}, \(F_1\) defined
in~\eqref{DS-S-contrib} and the error \(\epsilon\) of order
\begin{equation}
\label{DS-bound-error}
o \Bigg(
\sum_{i = 1}^n \Big[
\tau_i^{\beta_i} + \lambda_i^{\beta_i} + t_i^{\beta_i} + s_i^{\beta_i}
+| t_i |^{\alpha_i} +| s_i |^{\alpha_i} + | t_i - s_i |^{\alpha_i}
+| \tau_i + t_i - \lambda_i |^{\alpha_i} + | \lambda_i + s_i - \tau_i |^{\alpha_i}
\Big]
\Bigg).
\end{equation}
\subsubsection{Conditional mean: upper bound}
\label{sec:org4e1dc87}
Recall that we are interested in the upper bound for the rescaled conditional
mean vector \(\bm{d}_{u, \bm{\tau}, \bm{\lambda}}\) uniform in
\(\bm{t}, \, \bm{s} \in [\bm{0}, \bm{S}]\) and \(\bm{\tau}\) and \(\bm{\lambda}\)
such that with some \(\mathcal{I}' \subset \mathcal{I}\)
\begin{equation}\label{DS-tau-lambda-cond}
\bm{0} \leq \bm{\tau} \leq \bm{\lambda} \leq u^{2/\bm{\nu} - 2/\bm{\beta}} \bm{\Lambda} / \bm{S}, \quad
\bm{\tau}_{\mathcal{J} \cup \mathcal{K}} = \bm{\lambda}_{\mathcal{J} \cup \mathcal{K}} = \bm{0}_{\mathcal{J} \cup \mathcal{K}}, \quad
\bm{\tau}_{\mathcal{I}'} = \bm{\lambda}_{\mathcal{I}'}, \quad
\bm{\tau}_{\mathcal{I} \setminus \mathcal{I}'} + \mathcal{S}_{\mathcal{I} \setminus \mathcal{I}'} < \bm{\lambda}_{\mathcal{I}' \setminus \mathcal{I}}.
\end{equation}
Let us bound \(F_1 \, \bm{w}\) and \(F_2 \, \bm{w}\) separately.

\textbf{Bound for \(F_1 \, \bm{w}\).}
Multiplying \(F_1\) by \(\bm{w}\) on the right and rescaling all time
parameters by \(u^{-2/\bm{\nu}}\), we find that
\begin{multline*}
u^2 \, F_{1, u} \, \bm{w} \coloneqq
\frac{1}{4}
\sum_{i = 1}^n
u^{2 - 2 \alpha_i / \nu_i}
\Big[
2 \, S_{\alpha_i, A_{5, i}} ( t_i )
+2 \, S_{\alpha_i, A_{5, i}} ( s_i )
+S_{\alpha_i, A_{5, i}} ( \tau_i + t_i - \lambda_i )
\\[7pt]
+S_{\alpha_i, A_{5, i}} ( \lambda_i + s_i - \tau_i )
-S_{\alpha_i, A_{5, i}} ( \tau_i - \lambda_i )
-S_{\alpha_i, A_{5, i}} ( \lambda_i - \tau_i )
\Big]
\, \bm{w}.
\end{multline*}
Then, with some \(c_1 > 0\) holds
\begin{equation*}
\left| S_{\alpha_i, A_{5, i}} ( t_i ) \, \bm{w} \right|, \
\left| S_{\alpha_i, A_{5, i}} ( s_i ) \, \bm{w} \right|
\leq
c_1 \, S_i^{\alpha_i}
\end{equation*}
and, using
\begin{equation}\label{eq:ineq2}
\left| \left| x \pm 1 \right|^{\zeta} - x^{\zeta} \right|
\leq c_2 \, ( x \vee 1 )^{\zeta-1}
\quad \text{for} \quad x \geq 0,
\end{equation}
we find that
\begin{align*}
\left|
\Big[
S_{\alpha_i, A_{5, i}} ( \tau_i + t_i - \lambda_i )
-S_{\alpha_i, A_{5, i}} ( \tau_i - \lambda_i )
\Big]
\, \bm{w}
\right|
& \leq
c_3 \Big[
( \lambda_i - \tau_i - t_i )^{\alpha_i} - ( \lambda_i - \tau_i )^{\alpha_i}
\Big]
\\[7pt]
& \leq
c_4 \, S_i \, \left( ( \lambda_i - \tau_i ) \vee S_i \right)^{\alpha_i - 1},
\\[7pt]
\left|
\Big[
S_{\alpha_i, A_{5, i}} ( \lambda_i + s_i - \tau_i )
-S_{\alpha_i, A_{5, i}} ( \lambda_i - \tau_i )
\Big]
\, \bm{w}
\right|
& \leq
c_5 \Big[
( \lambda_i - \tau_i + s_i )^{\alpha_i} - ( \lambda_i - \tau_i )^{\alpha_i}
\Big]
\\[7pt]
& \leq
c_6 \, S_i \, \left( ( \lambda_i - \tau_i ) \vee S_i \right)^{\alpha_i - 1}.
\end{align*}

\textbf{Bound for \(F_2 \, \bm{w}\).}
Multiplying \(F_2\) by \(\bm{w}\) on the right, rescaling all time parameters
by \(u^{-2/\bm{\nu}}\) and using
\(A_{1, i} \, \bm{w} = \bm{0}\) yields
\begin{equation}\label{DS-F2-bound}
u^2 \, F_{2, u} \, \bm{w}
=
\sum_{i = 1}^n u^{2 - 2 \beta_i / \nu_i} \, B_i
+\sum_{i, j \in \mathcal{F}} u^{2 - \beta_i / \nu_i - \beta_j / \nu_j} \, B_{i, j},
\end{equation}
where we have introduced two shorthands
\begin{gather*}
B_i \coloneqq
\frac{1}{2} \,
A_{2, i} \, \bm{w}
\left[
\Big( ( \tau_i + t_i )^{\beta_i} - \tau_i^{\beta_i} \Big)
+\Big( ( \lambda_i + s_i )^{\beta_i} - \lambda_i^{\beta_i} \Big)
\right]
\\[7pt]
B_{i, j} \coloneqq
\frac{1}{4} \,
D_{i,j} \, \bm{w}
\left[
\left( ( \tau_i + t_i )^{\beta_i/2} - \tau_i^{\beta_i/2} \right) \,
+\left( ( \lambda_i + s_i )^{\beta_i/2} - \lambda_i^{\beta_i/2} \right) \,
\right]
\left( \tau_i^{\beta_j/2} + \lambda_j^{\beta_j/2} \right).
\end{gather*}
They can all be bounded using~\eqref{eq:ineq1} as follows:
\begin{equation*}
\left| B_i \right| \leq c_1 S_i \, ( \lambda_i \vee S_i )^{\beta_i - 1}, \quad
\left| B_{i, j} \right| \leq c_1
\Big[
S_i \, ( \tau_i \vee S_i )^{\beta_i - 1} + S_j \, ( \lambda_j \vee S_j )^{\beta_j - 1}
\Big] \,
\left( \tau_j^{\beta_j/2} + \lambda_j^{\beta_j/2} \right).
\end{equation*}
To simplify the bound further, let us get rid of the mixing, applying
\begin{equation}\label{ineq3}
2 \, x \, y \leq \varepsilon x + \varepsilon^{-1} y
\end{equation}
to each of the terms. Hence, the right-hand side of~\eqref{DS-F2-bound} is
at most
\begin{equation*}
c_1
\sum_{i = 1}^n
u^{2-2\beta_i/\nu_i}
\Big[
S_i \, ( \lambda_i \vee S_i )^{\beta_i - 1}
+\varepsilon \, \lambda_j^{\beta_i}
+\varepsilon^{-1} \, S_j^2 \, ( \lambda_j \vee S_j )^{\beta_j - 2}
\Big].
\end{equation*}

\textbf{Error.}
The error~\eqref{DS-bound-error} is no larger than
\begin{equation*}
\varepsilon
\sum_{i = 1}^n
\left[
u^{2 - 2 \beta_i / \nu_i} \, ( \lambda_i \vee S_i )^{\beta_i}
+u^{2 - 2 \alpha_i / \nu_i} \, \left( ( \lambda_i - \tau_i) \vee S_i \right)^{\alpha_i}
\right],
\end{equation*}
where \(\varepsilon\) can be made arbitrarily small by choosing \(u\) large enough.

\textbf{Combined bound.}
\begin{multline*}
\left| u^2 \, \bm{d}_{u, \bm{\tau}, \bm{\lambda}} ( \bm{t}, \bm{s} ) \right|
\leq
c_1
\sum_{i = 1}^n \Bigg[
u^{2 - 2 \alpha_i / \nu_i} \,
\Big[
S_i \, \left( ( \lambda_i - \tau_i ) \vee S_i \right)^{\alpha_i - 1}
+\varepsilon \left( ( \lambda_i - \tau_i ) \vee S_i \right)^{\alpha_i}
\Big]
\\[7pt]
+u^{2-2\beta_i/\nu_i} \,
\Big[
S_i \, ( \lambda_i \vee S_i )^{\beta_i - 1}
+\varepsilon \, ( \lambda_i \vee S_i )^{\beta_i}
+\varepsilon^{-1} \, S_j^2 \, ( \lambda_j \vee S_j )^{\beta_j - 2}
\Big]
\Bigg]
\end{multline*}
The right-hand side of this bound can be taken as \(G\)
for~\eqref{DS:G-and-sigma}.
\subsubsection{Conditional covariance}
\label{sec:orgcb67e45}
Finally, we need a bound for
\begin{multline*}
\mathcal{R}_{\bm{\tau}, \bm{\lambda}} ( \bm{t}_1, \bm{s}_1, \bm{t}_2, \bm{s}_2 )
=
R ( \bm{\tau} + \bm{t}_1, \bm{\lambda} + \bm{s}_1, \bm{\tau} + \bm{t}_2, \bm{\lambda} + \bm{s}_2 )
\\
+R ( \bm{\tau} + \bm{t}_1, \bm{\lambda} + \bm{s}_1, \bm{\tau}, \bm{\lambda} ) \,
\Sigma^{-1} ( \bm{\tau}, \bm{\lambda} ) \,
R ( \bm{\tau}, \bm{\lambda}, \bm{\tau} + \bm{t}_2, \bm{\lambda} + \bm{s}_2 ).
\end{multline*}
Introduce the following shorthands:
\begin{align*}
& R ( \bm{\tau} + \bm{t}_1, \bm{\lambda} + \bm{s}_1, \bm{\tau} + \bm{t}_2, \bm{\lambda} + \bm{s}_2 )
\eqqcolon \Sigma - C_1,
\\[7pt]
& R ( \bm{\tau} + \bm{t}_1, \bm{\lambda} + \bm{s}_1, \bm{\tau}, \bm{\lambda} )
\eqqcolon \Sigma - C_2,
\\[7pt]
& R ( \bm{\tau}, \bm{\lambda}, \bm{\tau} + \bm{t}_2, \bm{\lambda} + \bm{s}_2 )
\eqqcolon \, \Sigma - C_3
\\[7pt]
& \Sigma^{-1} ( \bm{\tau}, \bm{\lambda} )
\eqqcolon \, \Sigma^{-1} + \Sigma^{-1} B \, \Sigma^{-1} \, + \Sigma^{-1} B \, \Sigma^{-1} B \, \Sigma^{-1},
\end{align*}
Using these shorthands, we find that \(\mathcal{R}_{\bm{\tau}, \bm{\lambda}}\) satisfies
\begin{equation*}
\mathcal{R}_{\bm{\tau}, \bm{\lambda}} ( \bm{t}_1, \bm{s}_1, \bm{t}_2, \bm{s}_2 )
\sim
G_1 + G_2,
\quad \text{where} \quad
\begin{cases}
G_1 \coloneqq -C_1 + C_2 -B + C_3, \\[3pt]
\begin{multlined}
G_2 \coloneqq B \, \Sigma^{-1} C_3 + C_2 \, \Sigma^{-1} B
\\[-10pt] -C_2 \, \Sigma^{-1} C_3 - B \, \Sigma^{-1} B.
\end{multlined}
\end{cases}
\end{equation*}
We shall also need the following formula:
\begin{equation*}
C ( \bm{t}, \bm{s} )
\sim
\sum_{i = 1}^n \Big[
A_{1, i} \, t^{\beta_i'}
+A_{2, i} \, t^{\beta_i}
+A_{1, i}^\top \, s_i^{\beta_i'}
+A_{2, i}^\top \, s_i^{\beta_i}
+S_{\alpha_i, A_{5, i}} ( t_i - s_i )
\Big]
+\sum_{i, \, j \in \mathcal{F}} A_{6, i, j} \, t_i^{\beta_i / 2} \, s_j^{\beta_j / 2},
\end{equation*}

\textbf{Terms with \(( \bm{t}_1, \bm{t}_2 )\).}
In \(G_1\), these terms come from
\begin{equation*}
G_{1, 1}
=
\frac{1}{4}
\Big[
-C ( \bm{\tau} + \bm{t}_1, \bm{\tau} + \bm{t}_2 )
+C ( \bm{\tau} + \bm{t}_1, \bm{\tau} )
-C ( \bm{\tau}, \bm{\tau} )
+C ( \bm{\tau}, \bm{\tau} + \bm{t}_2 )
\Big].
\end{equation*}
The following coefficients are zero:
\begin{equation*}
( \tau_i + t_{1, i} )^{\beta_i'}, \
( \tau_i + t_{1, i} )^{\beta_i}, \
( \tau_i + t_{2, i} )^{\beta_i'}, \
( \tau_i + t_{2, i} )^{\beta_i}, \
\tau_i^{\beta_i'}, \
\tau_i^{\beta_i}.
\end{equation*}
Mixed terms from \(G_{1, 1}\) and \(G_2\):
\begin{equation*}
\frac{1}{4}
\sum_{i, j \in \mathcal{F}}
D_{i, j}
\Big[ ( \tau_i + t_{1, i} )^{\beta_i/2} - \tau_i^{\beta_i/2} \Big]
\Big[ ( \tau_j + t_{1, j} )^{\beta_j/2} - \tau_j^{\beta_j/2} \Big].
\end{equation*}
\(S\)-terms:
\begin{equation*}
\frac{1}{4}
\sum_{i = 1}^n
\Big[
-S_{\alpha_i, A_{5, i}} ( t_{1, i} - t_{2, i} )
+S_{\alpha_i, A_{5, i}} ( t_{1, i} )
+S_{\alpha_i, A_{5, i}} ( -t_{2, i} )
\Big].
\end{equation*}
Rescaling everything by \(u^{-2/\bm{\nu}}\), we obtain
\begin{equation*}
u^2 \, \left\| \substack{\text{terms with} \\ \bm{t}_1 \ \text{and} \ \bm{t}_2} \right\|
\leq
c_1 \sum_{i = 1}^n u^{2-2\alpha_i/\nu_i} \, S_i^{\alpha_i}
+c_1 \sum_{i, j \in \mathcal{F}} u^{2-\beta_i/\nu_i-\beta_j/\nu_j} \,
S_i \, ( \tau_i \vee S_i )^{\beta_i/2 - 1} \,
S_j \, ( \tau_j \vee S_j )^{\beta_j/2 - 1}.
\end{equation*}
We can also simplify the bound by getting rid of the mixing:
\begin{equation*}
u^2 \, \left\| \substack{\text{terms with} \\ \bm{t}_1 \ \text{and} \ \bm{t}_2} \right\|
\leq
c_2 \sum_{i = 1}^n \left[
u^{2-2\alpha_i/\nu_i} \, S_i^{\alpha_i}
+u^{2 - 2 \beta_i / \nu_i} S_i^2 \, ( \tau_i \vee S_i )^{\beta_i - 2}
\right].
\end{equation*}

\textbf{Terms with \(( \bm{t}_1, \bm{s}_2 )\).}
In \(G_1\), these terms come from
\begin{equation*}
G_{1, 2}
=
\frac{1}{4}
\Big[
-C ( \bm{\tau} + \bm{t}_1, \bm{\lambda} + \bm{s}_2 )
+C ( \bm{\tau} + \bm{t}_1, \bm{\lambda} )
-C ( \bm{\lambda}, \bm{\lambda} )
+C ( \bm{\tau}, \bm{\lambda} + \bm{s}_2 )
\Big].
\end{equation*}
The following coefficients are zero:
\begin{equation*}
( \tau_i + t_{1, i} )^{\beta_i'}, \
( \tau_i + t_{1, i} )^{\beta_i}, \
( \lambda_i + s_{2, i} )^{\beta_i'}, \
( \lambda_i + s_{2, i} )^{\beta_i}, \
\lambda_i^{\beta_i'}, \
\lambda_i^{\beta_i}.
\end{equation*}
Mixed terms from \(G_{1, 1}\) and \(G_2\):
\begin{equation*}
\frac{1}{4}
\sum_{i, j \in \mathcal{F}}
D_{i, j}
\Big[ ( \tau_i + t_{1, i} )^{\beta_i/2} - \tau_i^{\beta_i/2} \Big]
\Big[ ( \lambda_j + s_{1, j} )^{\beta_j/2} - \lambda_j^{\beta_j/2} \Big].
\end{equation*}
\(S\)-terms:
\begin{equation*}
\frac{1}{4}
\sum_{i = 1}^n
\Big[
-S_{\alpha_i, A_{5, i}} ( \tau_i + t_{1, i} - \lambda_i - s_{2, i} )
+S_{\alpha_i, A_{5, i}} ( \tau_i + t_{1, i} )
-S_{\alpha_i, A_{5, i}} ( \tau_i - \lambda_i )
+S_{\alpha_i, A_{5, i}} ( -\lambda_i -s_{2, i} )
\Big].
\end{equation*}
Rescaling everything by \(u^{-2/\bm{\nu}}\), and getting rid of the mixed terms
as above, we obtain
\begin{equation*}
u^2 \, \left\|
\substack{
\text{terms with}
\\ \bm{t}_1 \ \text{and} \ \bm{t}_2
}
\right\|
\leq
c_1 \sum_{i = 1}^n
\left[
u^{2-2\alpha_i/\nu_i} \, S_i \left( ( \lambda_i - \tau_i ) \vee S_i \right)^{\alpha_i - 1}
+u^{2 - 2 \beta_i / \nu_i} \, S_i^2 \, ( \lambda_i \vee S_i )^{\beta_i - 2}.
\right]
\end{equation*}

The remaining terms (with \(( \bm{s}_1, \bm{t}_2 )\) and
\(( \bm{s}_1, \bm{s}_2 )\)) may be estimated similarly.

\textbf{Error}.
The accumulated error is at most
\begin{equation*}
\varepsilon
\sum_{i = 1}^n \Bigg[
u^{2 - 2 \beta_i / \nu_i} ( \lambda_i \vee S_i )^{\beta_i}
+u^{2 - 2 \alpha_i / \nu_i} \,
\left( ( \lambda_i - \tau_i ) \vee S_i \right)^{\alpha_i}
\Bigg]
\end{equation*}
with \(\varepsilon\), which can be made arbitrarily small by taking \(u\) large
enough.

\textbf{Combined bound.}
Combining the bounds together, we find that
\begin{multline*}
u^2 \,
\left\|
\mathcal{R}_{u, \bm{\tau}, \bm{\lambda}} ( \bm{t}_1, \bm{s}_1, \bm{t}_1, \bm{s}_1 )
\right\|
\leq
c_1 \sum_{i = 1}^n
\Bigg[
u^{2 - 2 \alpha_i / \nu_i}
\Big[
S_i^{\alpha_i}
+S_i \, \left( ( \lambda_i - \tau_i ) \vee S_i \right)^{\alpha_i - 1}
+\varepsilon \, ( \lambda_i - \tau_i )^{\alpha_i}
\Big]
\\[7pt]
+u^{2 - 2 \beta_i / \nu_i} \,
\Big[
S_i ( \lambda_i \vee S_i )^{\beta_i - 1}
+\varepsilon ( \lambda_i \vee S_i )^{\beta_i}
\Big]
\Bigg]
\end{multline*}

We have therefore obtained that~\eqref{DS:G-and-sigma} holds with
\begin{multline*}
G + \sigma^2
=
c_1
\sum_{i = 1}^n \Bigg[
u^{2 - 2 \alpha_i / \nu_i} \,
\Big[
S_i^{\alpha_i}
+S_i \, \left( ( \lambda_i - \tau_i ) \vee S_i \right)^{\alpha_i - 1}
+\varepsilon ( \lambda_i - \tau_i )^{\alpha_i}
\Big]
\\[7pt]
+u^{2-2\beta_i/\nu_i} \,
\Big[
S_i \, ( \lambda_i \vee S_i )^{\beta_i - 1}
+\varepsilon \, ( \lambda_i \vee S_i )^{\beta_i}
+\varepsilon^{-1} \, S_j^2 \, ( \lambda_j \vee S_j )^{\beta_j - 2}
\Big]
\Bigg].
\end{multline*}
\subsection{Integral estimate}
\label{sec:org3915b78}
\begin{proof}[Proof of Lemma~\ref{lemma:integral_estimate}]
Define a collection of sets
\(\Omega_F = \left\{
\bm{x} \in \mathbb{R}^d \colon
\bm{x}_F > \bm{0}, \  \bm{x}_{F^c} < \bm{0}
\right\}\) indexed by \(F \subset \{ 1, \ldots, d \}\) and
split the integral:
\[
  \int_{\mathbb{R}^d}
  e^{\bm{w}^\top \bm{x}} \,
  \mathbb{P} \left\{
    \exists \, \bm{t} \in [ \bm{0}, \bm{\Lambda} ] \colon
    \bm{\chi}_{\bm{x}} ( \bm{t} ) > \bm{x}
  \right\}
  \mathop{d \bm{x}}
  =
  \sum_{F \in 2^d}
  \int_{\Omega_F}
  e^{\bm{w}^\top \bm{x}} \,
  \mathbb{P} \left\{
    \exists \, \bm{t} \in [ \bm{0}, \bm{\Lambda} ] \colon
    \bm{\chi}_{\bm{x}} ( \bm{t} ) > \bm{x}
  \right\}
  \mathop{d \bm{x}}.
\]
For \(\bm{x} \in \Omega_F\) the probability under the integral may be bounded as follows:
\begin{align*}
  &
  \mathbb{P} \left\{
    \exists \, \bm{t} \in [ \bm{0}, \bm{\Lambda} ] \colon
    \bm{\chi}_{\bm{x}} ( \bm{t} ) > \bm{x}
  \right\}
  \\[7pt]
  & \hspace{30pt} \leq
  \mathbb{P} \left\{
    \exists \, \bm{t} \in [ \bm{0}, \bm{\Lambda} ] \colon
    \bm{w}_F^\top \big(
    \bm{\chi}_{\bm{x}, F} ( \bm{t} )
    -\mathbb{E} \left\{ \bm{\chi}_{\bm{x}, F} ( \bm{t} ) \right\}
    \big)
    > \bm{w}_F^\top \bm{x}_F
    -\bm{w}_F^\top \mathbb{E} \left\{ \bm{\chi}_{\bm{x}, F} ( \bm{t} ) \right\}
  \right\}
  \\[7pt]
  & \hspace{30pt} \leq
  \mathbb{P} \left\{
    \exists \, \bm{t} \in [ \bm{0}, \bm{\Lambda} ] \colon
    \bm{w}_F^\top \big( \bm{\chi}_{\bm{x}}
    -\mathbb{E} \left\{ \bm{\chi}_{\bm{x}} ( \bm{t} ) \right\}
    \big)
    > \bm{w}_F^\top \bm{x}_F - G - \varepsilon \sum_{j = 1}^d | x_j |
  \right\}
  \\[7pt]
  & \hspace{30pt} =
  \mathbb{P} \left\{
    \exists \, \bm{t} \in [ \bm{0}, \bm{\Lambda} ] \colon
    \eta_{\bm{x}, F} ( \bm{t} ) > r_{F, \varepsilon} ( \bm{x} ) - G
  \right\},
\end{align*}
where
\[
  r_{F, \varepsilon} ( \bm{x} )
  = \bm{w}_F^\top \, \bm{x}_F - \varepsilon \sum_{j = 1}^d | x_j |
  \qquad \text{and} \qquad
  \eta_{\bm{x}, F} ( \bm{t} )
  = \bm{w}_F^\top \big(
    \bm{\chi}_{\bm{x}, F} ( \bm{t} )
    -\mathbb{E} \left\{ \bm{x}_{\bm{x}, F} ( \bm{t} ) \right\}
  \big).
\]
Let us split the domain \(\Omega_{F}\) into two parts
\[
  \Omega_{F,+} = \left\{ \bm{x} \in \Omega_F \colon r_{F, \varepsilon} ( \bm{x} ) > G \right\}
  \qquad \text{and} \qquad
  \Omega_{F, -} = \Omega_F \setminus \Omega_{F, +}.
\]
Let us first deal with the integral over \(\Omega_{F,-}\).
It follows from \(\bm{w}_F^\top \, \bm{x}_F - \varepsilon \sum_{j = 1}^d | x_j | < G\) that
\[
  \sum_{j \in F} ( w_i - \varepsilon ) | x_j | - \varepsilon \sum_{j \in F^c} | x_j | < G
\]
or, with \(w_{ * } = \min_{j \in F} w_j > 0\) and \(\varepsilon < w_{ * }\),
\[
\varepsilon \sum_{j \in F} | x_j | \leq \frac{\varepsilon G}{w_{ * } - \varepsilon} + \frac{\varepsilon^2}{w_{ * } - \varepsilon} \sum_{j \in F^c} | x_j |
\]
Therefore, with \(r = r_{F, \varepsilon} ( \bm{x} )\), we have
\begin{multline*}
  \bm{w}^\top \bm{x}
  = r + \bm{w}_{F^c}^\top \, \bm{x}_{F^c} + \varepsilon \sum_{j = 1}^d | x_j |
  = r + \varepsilon \sum_{j \in F} | x_j | - \sum_{j \in F^c} ( w_j - \varepsilon ) | x_j |
  \\[7pt]
  \leq r + \frac{\varepsilon G}{w_{ * } - \varepsilon}
  -\left( w_{ * } - \frac{\varepsilon^2}{w_{ * } - \varepsilon} - \varepsilon \right)
  \sum_{j \in F^c} | x_j |
  \leq r + \frac{\varepsilon G}{w_{ * } - \varepsilon},
\end{multline*}
provided that \(\varepsilon\) is small enough.
Bounding the probability under the integral by \(1\) and changing the variables, we obtain
\begin{multline*}
  \int_{\Omega_{F,-}}
  e^{\bm{w}^\top \bm{x}} \,
  \mathbb{P} \left\{
    \exists \, \bm{t} \in [ \bm{0}, \bm{\Lambda} ] \colon
    \bm{\chi}_{\bm{x}} ( \bm{t} ) > \bm{x}
  \right\}
  \mathop{d \bm{x}}
  \leq
  \int_{\Omega_{F,-}}
  e^{\bm{w}^\top \bm{x}}
  \mathop{d \bm{x}}
  =
  \int_{-\infty}^G \mathop{dr} \int \mathop{dS}
  e^{\bm{w}^\top \bm{x}} \,
  r^{d - 1}
  \leq \\[7pt] \leq
  \int_{-\infty}^G \int \mathop{dS}
  e^{r + \varepsilon G / ( w_{*} - \varepsilon )}
  r^{d - 1} \mathop{d r} \mathop{d S}
  \leq
  c_1 e^{\varepsilon G / ( w_{*} - \varepsilon )} \int_{-\infty}^G e^{(1 + \varepsilon) r} \mathop{dr}
  = c_1 \, e^{c_2 G}.
\end{multline*}

Next, we concentrate on the intergral over \(\Omega_{F,+}\). By Piterbarg inequality,
we have the following uniform in \(\bm{x} \in \Omega_{F, +}\) upper bound:
\[
  \mathbb{P} \left\{
    \exists \, \bm{t} \in [ \bm{0}, \bm{\Lambda} ] \colon
    \eta_{\bm{x}, F} ( \bm{t} ) > \bm{x}
  \right\}
  \leq
  c_3 \left( \frac{r ( \bm{x} ) - G}{\sigma} \right)^{2/\gamma}
  \exp \left( -\frac{1}{2} \left( \frac{r ( \bm{x} ) - G}{\sigma} \right)^2 \right).
\]
Plugging this bound into the integral and changing the variables, we obtain
\begin{multline*}
  \int_{\Omega_{F,+}}
  e^{\bm{w}^\top \bm{x}} \,
  \mathbb{P} \left\{
    \exists \, \bm{t} \in [ \bm{0}, \bm{\Lambda} ] \colon
    \bm{\chi}_{\bm{x}} ( \bm{t} ) > \bm{x}
  \right\}
  \mathop{d \bm{x}}
  \leq \\[7pt] \leq
  c_3
  \int_{\Omega_{F,+}}
  e^{\bm{w}^\top \bm{x}}
  \left( \frac{r ( \bm{x} ) - G}{\sigma} \right)^{2/\gamma}
  \exp \left( -\frac{1}{2} \left( \frac{r ( \bm{x} ) - G}{\sigma} \right)^2 \right)
  \mathop{d \bm{x}}
  = \\[7pt] =
  c_3 \int_G^{\infty} \mathop{dr} \int \mathop{dS}
  e^{\bm{w}^\top \bm{x}}
  \left( \frac{r - G}{\sigma} \right)^{2/\gamma + d - 1}
  \exp \left( -\frac{1}{2} \left( \frac{r - G}{\sigma} \right)^2 \right).
\end{multline*}
Note that with \(w^{ * } = \max_{i = 1, \dots, d} w_i\) we have
\[
  r
  =
  \sum_{i \in F} (w_i - \varepsilon) | x_i |
  -\varepsilon \sum_{i \in F^{c}} | x_i |
  \geq
  \left( w^{ * } - \varepsilon \right) \sum_{i \in F} | x_i |
  -\varepsilon \sum_{i \in F^c} | x_i |
\]
and it follows that for all \(\varepsilon < w^{ * }\) the following bound holds:
\[
  \varepsilon \sum_{i \in F} | x_i |
  \leq
  \frac{\varepsilon r}{w^{ * } - \varepsilon}
  +\frac{\varepsilon^2}{w^{ * } - \varepsilon}
  \sum_{i \in F^{c}} | x_i |.
\]
This bound yields
\[
  \left( \bm{w}, \bm{x} \right)
  =
  r
  +\varepsilon \sum_{i \in F}^d | x_i |
  -\sum_{i \in F^{c}} (w_i - \varepsilon) | x_i |
  \leq
  \left( 1 + \frac{\varepsilon}{w^{ * } - \varepsilon} \right) r
  -\left(
    w_{ * } - \varepsilon
    -\frac{\varepsilon^2}{w^{ * } - \varepsilon}
  \right)
  \sum_{i \in F^{c}} | x_i |,
\]
from which for small enough \(\varepsilon\) follows that \(\left( \bm{w}, \bm{x} \right)
\leq
(1 + \varepsilon') r,\) with \(\varepsilon' = \varepsilon / (w^{ * } - \varepsilon)\).
Hence,
\begin{multline*}
  c_3 \int_G^{\infty} \mathop{dr} \int \mathop{dS}
  e^{\bm{w}^\top \bm{x}}
  \left( \frac{r - G}{\sigma} \right)^{2/\gamma + d - 1}
  \exp \left( -\frac{1}{2} \left( \frac{r - G}{\sigma} \right)^2 \right)
  \leq \\[7pt] \leq
  c_4 \int_{-\infty}^{\infty}
  e^{(1 + \varepsilon') r} \exp \left( -\frac{1}{2} \left( \frac{r - G}{\sigma} \right)^2 \right) \mathop{dr}
  \leq
  c_5 \, e^{c_6 ( G + \sigma^2 )},
\end{multline*}
where in the last step we used the Gaussian mgf formula
\(\mathbb{E} \left\{ e^{t \mathcal{N} ( \mu, \sigma^2 )} \right\} = e^{t \mu + t^2 \sigma^2 / 2}\) with \(t = 1 + \varepsilon'\).
\end{proof}
\subsection{Double crossing: vicinity of the diagonal}
\label{sec:orgf70551c}
\begin{proof}[Proof of Lemma~\ref{double-crossing-diagonal}]
We begin the proof with the following upper bound:
\begin{multline*}
\mathbb{P} \left\{ \exists \, \bm{t} \in D_{\varepsilon} \colon
X ( t_1 ) > a u, \ X ( t_2 ) < -bu
\right\}
\leq \mathbb{P} \left\{
\exists \, \bm{t} \in D_{\varepsilon}^+ \colon
X ( t_1 ) - X ( t_2 ) > ( a + b ) \, u
\right\}
\\
+\mathbb{P} \left\{
\exists \, \bm{t} \in D_{\varepsilon}^- \colon
X ( t_1 ) - X ( t_2 ) > ( a + b ) \, u
\right\},
\end{multline*}
where
\[
D_{\varepsilon}^+ = \left\{
\bm{t} = ( t, s ) \in [ 0, T ]^2 \colon t < s \leq t + \varepsilon
\right\},
\qquad
D_{\varepsilon}^- = \left\{
\bm{t} = ( t, s ) \in [ 0, T ]^2 \colon s < t \leq s + \varepsilon
\right\}.
\]
Define a Gaussian field
\[
\mathcal{X} ( s, l ) \coloneqq X ( s + l ) - X ( s ),
\qquad ( s, l ) \in \mathbb{T} \coloneqq [ 0, T ] \times [ 0, \varepsilon ]
\]
and use it to coarsen the bound above:
\[
\mathbb{P} \left\{
\exists \, \bm{t} \in D_{\varepsilon}^- \colon
X ( t_1 ) - X ( t_2 ) > ( a + b ) \, u
\right\}
\leq
\mathbb{P} \left\{
\exists \, ( s, l ) \in \mathbb{T} \colon
\mathcal{X} ( s, l ) > ( a + b ) \, u
\right\}.
\]
The variance of this Gaussian random field is
\begin{equation*}
\sigma^2 ( s, l )
= \var \{ \mathcal{X} ( s, l ) \}
= \mathbb{E} \left\{ \left[ X ( s + l ) - X ( s ) \right]^2 \right\}
\leq f ( \varepsilon ).
\end{equation*}
By Borell-TIS inequality~\eqref{Borell-TIS}, there exists \(\mu > 0\) such
that for all \(u > \mu\)
\begin{equation*}
\mathbb{P} \left\{ \exists \, ( s, l ) \in \mathbb{T} \colon \mathcal{X} ( s, l ) > ( a + b ) \, u \right\}
\leq \exp \left( -\frac{( u - \mu )^2}{2 \, f ( \varepsilon )} \right).
\end{equation*}
Since \(f ( \varepsilon ) \to 0\) by the hypotheses of the theorem, for any \(\delta > 0\)
there exists some \(\varepsilon > 0\) such that \(f ( \varepsilon ) < 1/4 \, \delta\). Therefore,
\[
\mathbb{P} \left\{
\exists \, \bm{t} \in D_{\varepsilon} \colon
\mathcal{X} ( s, l ) > ( a + b ) \, u
\right\}
\leq o \left( e^{-\delta u^2} \right).
\]
\end{proof}
\subsection{Double crossing for stationary processes: expansions}
\label{sec:orga180aa1}
\begin{proof}[Proof of Lemma~\ref{double-crossing-expansions}]
Consider \(\bm{X}_1 ( \bm{t} ) = ( X ( t_1 ), -X ( T - t_2 ) )^\top\). We want to
find the expansion of
\begin{equation*}
R ( \bm{t}, \bm{s} )
= \mathbb{E} \left\{ \bm{X}_1 ( \bm{t} ) \, \bm{X}_1 ( \bm{s} )^\top \right\}
= \begin{pmatrix}
\rho ( | t_1 - s_1 | ) & -\rho ( | T - s_2 - t_1 | ) \\
-\rho ( | T - t_2 - s_1 | ) & \rho ( | t_2 - s_2 | )
\end{pmatrix}
\end{equation*}
near \(\bm{t} = \bm{0}\). We have:
\begin{equation*}
\Sigma = \begin{pmatrix}
1 & -\rho ( T ) \\
-\rho ( T ) & 1
\end{pmatrix},
\end{equation*}
and
\begin{align*}
\Sigma - R ( \bm{t}, \bm{s} )
& = \begin{pmatrix}
1 - \rho ( | t_1 - s_1 | ) & \rho ( | T - s_2 - t_1 | ) - \rho ( T ) \\
\rho ( | T - t_2 - s_1 | ) - \rho ( T ) & 1 - \rho ( | t_2 - s_2 | )
\end{pmatrix}
\\[7pt]
& =
A_{2, 1} \, t_1 + A_{2, 2} \, t_2
+A_{2, 1}^\top \, s_2 +A_{2, 2}^\top \, s_2
+A_{5, 1} \, | t_1 - s_1 |^{\alpha}
+A_{5, 2} \, | t_2 - s_2 |^{\alpha}
\\[7pt]
& \quad
+o \left( t_1 + t_2 + s_1 + s_2 + | t_1 - s_1 |^{\alpha} + | t_2 - s_2 |^{\alpha} \right),
\end{align*}
where the matrix coefficients are given by
\begin{equation*}
A_{2, 1} = A_{2, 2}^\top =
-\rho' ( T )
\begin{pmatrix}
0 & 1 \\
0 & 0
\end{pmatrix},
\quad
A_{5, 1} =
\vartheta
\begin{pmatrix}
1 & 0 \\
0 & 0
\end{pmatrix},
\quad
A_{5, 2} =
\vartheta
\begin{pmatrix}
0 & 0 \\
0 & 1
\end{pmatrix}.
\end{equation*}
Clearly, \(\alpha_1 = \alpha_2 = \alpha\). Next, we need the optimal vector
\begin{equation*}
\bm{w}
= \Sigma^{-1} ( \bm{0} ) \, ( a, b )^\top
= \frac{1}{1 - \rho^2 ( T )}
\begin{pmatrix}
1 & \rho ( T ) \\
\rho ( T ) & 1
\end{pmatrix}
\begin{pmatrix}
a \\ b
\end{pmatrix}
= \frac{1}{1 - \rho^2 ( T )}
\begin{pmatrix}
a + b \rho ( T ) \\
b + a \rho ( T )
\end{pmatrix}
\end{equation*}
to check whether we have correctly identified \(\beta_1 = \beta_2 = 1\). That is, we
need to check~\ref{A2.4}
\begin{equation*}
\xi_1 = \bm{w}^\top \, A_{2, 1} \, \bm{w}
= \xi_2 = \bm{w}^\top \, A_{2, 2} \, \bm{w}
= \frac{
-\rho' ( T ) ( a + b \rho ( T ) )
( b + a \rho ( T ) )
}{( 1 - \rho^2 ( T ) )^2}
> 0.
\end{equation*}
We also need to check~\ref{A2.5}:
\begin{equation*}
\varkappa_1 = \bm{w}^\top \, A_{5, 1} \, \bm{w}
= \varkappa_2 = \bm{w}^\top \, A_{5, 2} \, \bm{w}
= \frac{C ( b + a \rho ( T ) )^2}{( 1 - \rho^2 ( T ) )^2}
> 0.
\end{equation*}
Assumption~\ref{A3} may be shown as follows:
\begin{align*}
\mathbb{E} \left\{ \left| \bm{X} ( \bm{t} ) - \bm{X} ( \bm{s} ) \right|^2 \right\}
& =
\mathbb{E} \left\{
( X ( t_1 ) - X ( s_1 ) )^2
+( X ( t_2 ) - X ( s_2 ) )^2
\right\}
\\
& =
2 ( 1 - \rho ( | t_1 - s_1 | ) )
+2 ( 1 - \rho ( | t_2 - s_2 | ) )
\\
& \leq
c_1 \Big( | t_1 - t_2 |^{\alpha} + | t_2 - s_2 |^{\alpha} \Big)
\end{align*}
with some constant \(c_1 > 0\). The last inequality follows from the
asymptotics of \(\rho ( t )\) near \(t = 0\). Hence, Assumption~\ref{A3}
is satisfied with \(\bm{\gamma} = ( \alpha, \alpha )^\top\).
\end{proof}
\subsection{Double crossing for fBm: minimization of generalized variance}
\label{sec:orge8d58b4}
\begin{proof}[Proof of Lemma~\ref{lemma:double-crossing-minimizer-fBm}]
We begin the proof by making use of the positive homogenity of the
generalized variance \(\sigma_{a, b}^{-2}\). Namely, for \(c > 0\) we have
\begin{equation*}
\sigma_{a, b}^{-2} ( c \bm{t} )
= c^{-2H} \sigma_{a, b}^{-2} ( \bm{t} ).
\end{equation*}
Therefore, if we assume that \(\bm{t} = ( t_1, t_2 )^\top\) lies in the lower
triangle \(t_1 > t_2\), we obtain from the equation above
\begin{equation*}
\sigma_{a, b}^{-2} ( t_1, t_2 )
= \left( \frac{t_1}{T} \right)^{-2H} \sigma_{a, b}^{-2}
\left( T, \frac{T t_2}{t_1} \right)
\geq  \sigma_{a, b}^{-2} \left( T, \frac{T t_2}{t_1} \right),
\end{equation*}
since \(( t_1 / T )^{-2H} \geq  1\). Since \(t_1 > t_2\), we have
that \(t_2' = T t_2 / t_1 \in ( 0, T )\) and it follows that
\begin{equation*}
\min_{t_1 > t_2} \sigma_{a, b}^{-2} ( t_1, t_2 ) \geq \min_{t_2' \in [ 0, T ]} \sigma_{a, b}^{-2} ( T, t_2' ),
\end{equation*}
hence, to minimize \(\bm{t} \mapsto \sigma_{a, b}^{-2} ( \bm{t} )\) in the lower triangle
we only need to minimize \(t_2 \mapsto \sigma_{a, b}^{-2} ( T, t_2 )\) in \(t_2 \in ( 0, T )\).
Similarly, to minimize \(\bm{t} \mapsto \sigma_{a, b}^{-2} ( \bm{t} )\) in the upper triangle,
we only need to minimize \(t_1 \mapsto \sigma_{a, b}^{-2} ( t_1, T )\) in \(t_1 \in ( 0, T )\).
\newline

If \(a < b\), the following trivial inequality
\begin{equation*}
\sigma_{a, b}^{-2} ( t, T ) =
\frac{a^2 t^{2H} + 2 a b r ( t, T ) + b^2 T^{2H}}{( t T )^{2H} - r^2 ( t, T )}
>
\frac{a^2 T^{2H} + 2 a b r ( t, T ) + b^2 t^{2H}}{( t T )^{2H} - r^2 ( t, T )}
= \sigma_{a, b}^{-2} ( T, t )
\end{equation*}
shows, that the minimum over lower triangle is strictly smaller than the minimum
over upper triangle. Therefore, the global minimum lies in \(t_1 > t_2\).
Similarly, if \(a > b\), the global minimum lies in \(t_1 < t_2\).
If \(a = b\), we have two global minima. \newline

Let us proceed to showing that the function
\begin{equation*}
t \mapsto \sigma_{a, b}^{-2} ( T, s )
= \frac{a^2 s^{2H} + 2 a b r ( s, T ) + b^2 T^{2H}}{( s T )^{2H} - r^2 ( s, T )}
\end{equation*}
possesses a unique minimum in \(t \in ( 0, T )\). Without loss of generality,
we may rewrite this function as
\begin{equation} \label{eq:sigma_D_representation}
\sigma_{a, b}^{-2} ( T, s )
=
\frac{a^2}{T^{2H}} \,
D_{b / a} \left( \frac{s}{T} \right),
\end{equation}
where
\[
D_{\alpha} ( s )
\coloneqq \frac{\alpha^2 + 2 \alpha f ( s ) + s^{2H}}{s^{2H} - f^2 ( s )}
= \frac{( \alpha + f ( s ) )^2}{s^{2H} - f^2 ( s )} + 1
\qquad s \in ( 0, 1 ),
\]
where we introduced the function
\[
f ( s ) = \frac{1}{2} \Big( s^{2H} + 1 - ( 1 - s )^{2H} \Big).
\]
A straightforward approach would be to show that
\(D_{\alpha}' ( 0+ ) < 0\), \(D_{\alpha}' ( 1- ) > 0\) and that
\(D_{\alpha}'' > 0\). The first two claims are easily seen to be true, but, unfortunately,
the third is false. The idea we shall employ to get around this issue is to
multiply the function \(D_{\alpha}'\) by an appropriately chosen and strictly positive
function \(U > 0\), so that the roots of \(D_{\alpha}' U\) remained the same as the
roots of \(D_{\alpha}'\), but \(D_{\alpha}' U\) became strictly increasing.
We now proceed to finding such multiplier.
\newline

First, rewrite \(D_{\alpha}'\) collecting \(\alpha\)-free and \(\alpha\)-dependent terms:
\begin{equation} \label{eq:D_prime}
D_{\alpha}' ( s ) = \frac{2 H ( \alpha + f ( s ) )}{( s^{2H} - f^2 ( s ) )^2}
\Big( \alpha G_{\alpha} ( s ) + G_0 ( s ) \Big),
\end{equation}
where
\begin{equation*}
G_{\alpha} ( s ) = \frac{f ( s ) f' ( s )}{H} - s^{2H - 1},
\qquad
G_0 ( s ) = \frac{f' ( s ) s^{2H}}{H} - f ( s ) s^{2H - 1}.
\end{equation*}
We can drop the positive factor
\begin{equation*}
\frac{2 H ( \alpha + f ( s ) )}{( s^{2H} - f^2 ( s ) )^2} > 0,
\end{equation*}
since the roots of \(D_{\alpha}'\) are that of \(\alpha \, G_{\alpha} ( s ) + G_0 ( s )\).
Unfortunately, this remainder is still non-monotone.
\newline

Proceding with the computations, we expand the derivatives and find that
\begin{align*}
G_\alpha ( s )
& = f ( s ) \Big( s^{2H - 1} + ( 1 - s )^{2H - 1} \Big) - s^{2H - 1} \\
& = f ( s ) ( 1 - s )^{2H - 1} - ( 1 - f ( s ) ) s^{2H - 1} \\
& = f ( s ) ( 1 - s )^{2H - 1} - f ( 1 - s ) s^{2H - 1} \\
& = s^{2 H - 1} ( 1 - s )^{2H - 1}
\Big( f ( s ) s^{1 - 2H} - f ( 1 - s ) ( 1 - s )^{1 - 2H} \Big).
\end{align*}
where in the second to last equality we used the identity \(f ( s ) + f ( 1 - s ) = 1\).
We can now represent \(G_{\alpha} ( s )\) as
\begin{equation*}
G_{\alpha} ( s )
= s^{2H - 1} ( 1 - s )^{2H - 1} \Big( A ( s ) - A ( 1 - s ) \Big),
\qquad A ( s ) = f ( s ) s^{1 - 2H}.
\end{equation*}
Similarly, but using the identity \(f ( s ) - f ( 1 - s ) = s^{2H} - ( 1 - s )^{2H}\),
we obtain a representation of \(G_0 ( s )\)
\begin{align*}
G_0 ( s )
& = \Big( s^{2H - 1} + ( 1 - s )^{2H - 1} \Big) s^{2H} - f ( s ) s^{2H - 1} \\
& = \Big( s^{2H - 1} + ( 1 - s )^{2H - 1} \Big) s^{2H}
-\Big( f ( 1 - s ) + s^{2H} - ( 1 - s )^{2H} \Big) s^{2H - 1} \\
& = -f ( 1 - s ) s^{2H - 1}
+\Big( ( 1 - s )^{2H - 1} s^{2H} + ( 1 - s )^{2H} s^{2H - 1} \Big) \\
& = s^{2H - 1} ( 1 - s )^{2H - 1}
\Big( -A ( 1 - s ) + 1 \Big).
\end{align*}
We can now rewrite~\eqref{eq:D_prime} as
\begin{equation} \label{eq:D_prime_product}
D_{\alpha}' ( s )
=
\widetilde{D}_{\alpha} ( s ) \,
\widetilde{G}_{\alpha} ( s ),
\qquad
\end{equation}
with
\begin{align}
\label{eq:G_tilde}
\widetilde{G}_{\alpha} ( s )
& \coloneqq
\alpha \Big( A ( s ) - A ( 1 - s ) \Big) - A ( 1 - s ) + 1,
\\[7pt]
\widetilde{D}_{\alpha} ( s )
& \coloneqq
\frac{
2 H ( \alpha + f ( s ) )
s^{2H - 1} ( 1 - s )^{2H - 1}
}{( s^{2H} - f^2 ( s ) )^2}.
\end{align}
We claim now that the function \(\widetilde{G}_{\alpha} ( s )\) is increasing.
Provided that this is true, we immediately obtain both existence and uniqueness
of the optimal point \(s_{ * }\), as well as positivity of the second
derivative at this point. Indeed,
\begin{equation} \label{eq:D_double_prime}
D_{\alpha}'' ( s_{ * } )
=
\widetilde{D}_{\alpha}' ( s_{ * } ) \,
\underbracket[0.1pt]{\widetilde{G}_{\alpha} ( s_{ * } )}_{= 0}
+\underbracket[0.1pt]{\widetilde{D}_{\alpha} ( s_{ * } )}_{> 0} \,
\underbracket[0.1pt]{\widetilde{G}_{\alpha}' ( s_{ * } )}_{> 0}
> 0.
\end{equation}
To prove this claim it clearly suffices to show that \(A ( s )\) is increasing.
\newline

We have,
\[
A' ( s ) = f' ( s ) s^{1 - 2H} + ( 1 - 2H ) f ( s ) s^{-2H}
\]
and its positivity is equivalent to that of
\[
s f' ( s ) + ( 1 - 2 H ) f ( s ).
\]
In case \(H \leq 1 / 2\), the inequality
\[
s f' ( s ) + ( 1 - 2 H ) f ( s ) > 0
\]
is clear, since \(f' ( s ) > 0\).
If \(H > 1 / 2\), we use the Bernoulli inequality
\[
( 1 - s )^{2H - 1} \leq 1 - ( 2 H - 1 ) s,
\]
which gives
\begin{align*}
s f' ( s ) + ( 1 - 2 H ) f ( s )
& = ( 1 - 2H ) + s^{2H} - ( 1 - s )^{2H} + 2H ( 1 - s )^{2H - 1} \\
& \geq ( 1 - 2H ) + s^{2H} - ( 1 - s ) \Big( 1 - ( 2 H - 1 ) s \Big) + 2 H ( 1 - s )^{2H - 1} \\
& = s^{2H} - ( 2 H - 1 ) s^2 + 2 H \Big( ( 1 - s )^{2H - 1} - ( 1 - s ) \Big) \\
& \geq s^{2H} - ( 2 H - 1 ) s^2 = s^2 \Big( s^{2 H - 2} - 2H + 1 \Big) \\
& \geq 2 s^2 \Big( 1 - H \Big) > 0.
\end{align*}

As a corollary of the above, we obtain
\(\alpha \, G_{\alpha} ( s_{ * } ) + G_0 ( s_{ * } ) = 0\)
or
\begin{equation} \label{eq:identity_from_op_lemma}
\alpha \Bigg[
r ( s_{ * }, 1 )
\Big( s_{ * }^{2H - 1} + ( 1 - s_{ * } )^{2H - 1} \Big)
-s_{ * }^{2H - 1}
\Bigg]
+\Bigg[
s_{ * }^{2H}
\Big( s_{ * }^{2H - 1} + ( 1 - s_{ * } )^{2H - 1} \Big)
-r ( s_{ * }, 1 ) s_{ * }^{2H - 1}
\Bigg]
= 0,
\end{equation}
which will be useful for us in Lemma~\ref{double-crossing-expansions-fBm}.
\newline

We have thus shown that the function
\[
\sigma_{\bm{b}}^{-2} ( t, s )
= \frac{a^2 s^{2H} + 2 a b r ( t, s ) + b^{2H} t^2}{( t s )^{2H} - r^2 ( t, s )}
\]
posesses a unique minimum in the lower triangle. By
\eqref{eq:sigma_D_representation}, we see that this point is given by
\[
t_{ * } = T s_{ * },
\]
where \(s_{ * }\) is the minimizer of \(D\). Moreover, we have
by~\eqref{eq:D_prime_product}
\[
\frac{\partial \sigma_{\bm{b}}^{-2}}{\partial s} ( T, t_{ * } )
=
\frac{a^2}{T^{2H}}
\widetilde{D}_{b/a} ( s_{ * } )
\underbracket[0.1pt]{
\widetilde{G}_{b/a} ( s_{ * } )
}_{= 0}
= 0
\]
and by~\eqref{eq:D_double_prime} we obtain
\[
\kappa_2 \coloneqq
\frac{\partial^2 \sigma_{\bm{b}}^{-2}}{\partial s^2} ( T, t_{ * } )
=
\frac{a^2}{T^{2H}} \, D_{b / a}'' ( s_{ * } )
=
\frac{a^2}{T^{2H}} \,
\widetilde{D}_{b/a} ( s_{ * } )
\widetilde{G}_{b/a}' ( s_{ * } )
> 0
\]

Similarly to~\eqref{eq:sigma_D_representation}, let us rewrite
\(\sigma_{\bm{b}}^{-2}\) as follows:
\[
\sigma_{\bm{b}}^{-2} ( t, s )
=
\frac{a^2}{t^{2H}} \, D_{b/a} \left( \frac{s}{t} \right).
\]
Therefore,
\[
-\kappa_1 \coloneqq
\frac{\partial \sigma_{\bm{b}}^{-2}}{\partial t} ( T, t_{ * } )
=
-\frac{2 H a^2}{T^{2H + 1}} \, D_{b/a} ( s_{ * } )
+\frac{a^2}{T^{2H}} \,
\left[ \frac{-s_{ * }}{T^2} \right]
\underbracket[0.1pt]{D_{b/a}' ( s_* )}_{= 0}
< 0.
\]
Finally, we have
\begin{equation*}
\sigma_{\bm{b}}^{-2} ( \bm{t}_{ * } )
-\sigma_{\bm{b}}^{-2} ( \bm{t}_{ * } - \bm{\tau} )
\sim
-\kappa_1 \, \tau_1
-\kappa_2 \, \tau_2^2.
\end{equation*}
\end{proof}
\subsection{Double crossing for fBm: matrix expansions}
\label{sec:orgc36be16}
\begin{proof}[Proof of Lemma~\ref{double-crossing-expansions-fBm}]
Let \(\bm{t}_{ * } = ( T, t_{ * } )\) be a point in \([ 0, T ]^2\)
minimizing the generalized variance \(\sigma_{a, b}^{-2} ( \bm{t} )\).

Recall that
\begin{equation*}
\Sigma ( \bm{t}_{ * } )
= \begin{pmatrix}
T^{2H} & -r ( T, t_{ * } ) \\
-r ( T, t_{ * } ) & t_{ * }^{2H}
\end{pmatrix},
\qquad
R ( \bm{t}, \bm{s} )
= \begin{pmatrix}
r ( t_1, s_1 ) & -r ( t_1, s_2 ) \\
-r ( t_2, s_1 ) & r ( t_2, s_2 )
\end{pmatrix},
\end{equation*}
where
\begin{equation*}
r ( t, s ) = \frac{1}{2} \left( t^{2H} + s^{2H} - |t-s|^{2H} \right).
\end{equation*}
We have
\begin{equation*}
\Sigma ( \bm{t}_{ * } ) - R ( \bm{t}_{ * } + \, \bm{t}, \bm{t}_{ * } + \bm{s} )
=
\begin{pmatrix}
T^{2H} - r ( T + t_1, T + s_1 )
& -r ( T, t_{ * } ) + r ( T + t_1, t_{ * } + s_2 ) \\[7pt]
-r ( t_{ * }, T ) + r ( t_{ * } + t_2, T + s_1 )
& t_{ * }^{2H} - r ( t_{ * } + t_2, t_{ * } + s_2 )
\end{pmatrix}.
\end{equation*}
For the top left cell, we have:
\begin{equation*}
T^{2H} - r ( T + t_1, T + s_1 )
= T^{2H} - \frac{1}{2} \left(
( T + t_1 )^{2H} + ( T + s_1 )^{2H}
-| t_1 - s_1 |^{2H}
\right).
\end{equation*}
Here is the expression for the top right cell:
\begin{multline*}
r ( T + t_1, t_{ * } + s_2 ) - r ( T, t_{ * } )
=
\frac{1}{2} \left( ( T + t_1 )^{2H} - T^{2H} \right)
+\frac{1}{2} \left( ( t_{ * } + s_2 )^{2H} - t_{ * }^{2H} \right)
\\[7pt]
-\frac{1}{2} \left( | T - t_{ * } + t_1 - s_2 |^{2H} - | T - t_{ * } |^{2H} \right)
\end{multline*}
and similarly for the remaining two. Let us compute the first order coefficients
of different contributions. Jumping ahead, we will be giving these coefficients
names corresponding to their roles within Assumption~\ref{A2}. Recall that
the first index \(i\) in \(A_{i, j}\) corresponds to the order (first or
second) of the contribution, while the second indicates the variable \(t_j\).
The coefficients of the corresponding \(\bm{s}\)-terms can be expressed as
transpositions of these.
\smallskip

First, the only two terms which depend on the difference are the following:
\begin{equation*}
\left| t_1 - s_1 \right|^{2H} \colon \quad
A_{5, 1} \coloneqq
\frac{1}{2}
\begin{pmatrix}
1 & 0 \\
0 & 0
\end{pmatrix},
\qquad
\left| t_2 - s_2 \right|^{2H} \colon \quad
A_{5, 2} \coloneqq
\frac{1}{2}
\begin{pmatrix}
0 & 0 \\
0 & 1
\end{pmatrix},
\qquad
\alpha_1 = \alpha_2 = 2H.
\end{equation*}
Next, we proceed to the power-type contributions of the leading order:
\begin{align*}
& t_1 \colon \quad
A_{2, 1} \coloneqq
H
\begin{pmatrix}
-T^{2H-1} & T^{2H-1} - | T - t_{ * } |^{2H-1} \\
0 & 0
\end{pmatrix},
\\[7pt]
& t_2 \colon \quad
A_{1, 2} \coloneqq
H
\begin{pmatrix}
0 & 0 \\
t_{ * }^{2H-1} + | T - t_{ * } |^{2H-1} & -t_{ * }^{2H-1}
\end{pmatrix}.
\end{align*}
We will show below that
\begin{equation}\label{fbm-assumptions}
\bm{w}^\top A_{2, 1} \, \bm{w} > 0
\quad \text{and} \quad
A_{1, 2} \, \bm{w} ( \bm{t} ) \sim 0,
\end{equation}
whence the names \(A_{2, 1}\) and \(A_{1, 2}\). It also explains why we do
not need to compute the second order in \(t_1\). However, we do need to find
the second order in the second coordinate:
\begin{equation}\label{fbm-2-order}
t_2^2 \colon \quad
A_{2, 2} \coloneqq
H \left( H - \frac{1}{2} \right)
\begin{pmatrix}
0 & 0 \\
t_{ * }^{2H-2} + | T - t_{ * } |^{2H-2} & -t_{ * }^{2H-2}
\end{pmatrix},
\end{equation}
and the coefficient \(A_{6, 2, 2}\) of \(t_2 \, s_2\) is zero.
To show~\eqref{fbm-assumptions}, we need the inverse of \(\Sigma\)
\begin{equation*}
\Sigma^{-1} ( \bm{t} )
=
\frac{1}{t_1^{2H} t_2^{2H} - r^2 ( t_1, t_2 )}
\begin{pmatrix}
t_2^{2H} & r ( t_1, t_2 ) \\
r ( t_1, t_2 ) & t_1^{2H}
\end{pmatrix}
\end{equation*}
and its action on the vector \(\bm{w} ( \bm{t} )\):
\begin{equation*}
\bm{w} ( \bm{t} ) = \Sigma^{-1} ( \bm{t} ) \, \bm{b} =
\frac{1}{t_1^{2H} t_2^{2H} - r^2 ( t_1, t_2 )}
\begin{pmatrix}
t_2^{2H} \, a + r ( t_1, t_2 ) \, b
\\[7pt]
r ( t_1, t_2 ) \, a + t_1^{2H} \, b
\end{pmatrix}.
\end{equation*}
Using the following identity~\eqref{eq:identity_from_op_lemma} from
Lemma~\ref{lemma:double-crossing-minimizer-fBm}
\begin{equation*}
b \Bigg[
r ( s_{ * }, 1 )
\Big( s_{ * }^{2H - 1} + ( 1 - s_{ * } )^{2H - 1} \Big)
-s_{ * }^{2H - 1}
\Bigg]
+a \Bigg[
s_{ * }^{2H}
\Big( s_{ * }^{2H - 1} + ( 1 - s_{ * } )^{2H - 1} \Big)
-r ( s_{ * }, 1 ) s_{ * }^{2H - 1}
\Bigg]
= 0
\end{equation*}
we can show that Assumptions~\ref{A2.3} to~\ref{A2.5} are satisfied
with
\begin{equation*}
A_{1, 2} \, \bm{w} ( \bm{t} ) \sim 0, \quad
\bm{w}^\top A_{2, 1} \, \bm{w} > 0, \quad
\bm{w}^\top \, A_{5, 2} \, \bm{w} > 0, \quad
\beta_1 = 1, \quad
\beta_2 = 2, \quad
\mathcal{F} = \{ 2 \},
\end{equation*}
which gives
\begin{multline*}
\Sigma - R ( \bm{t}_{ * } + \bm{t}, \bm{t}_{ * } + \bm{s} )
=
\Big[
A_{2, 1} \, t_1
+A_{1, 2} \, t_2
+A_{2, 2} \, t_2^2
\Big]
+\Big[
A_{2, 1}^\top \, s_1
+A_{1, 2}^\top \, s_2
+A_{2, 2}^\top \, s_2^2
\Big]
\\[7pt]
+A_{5, 1} \, | t_1 - s_1 |^{2H}
+A_{5, 2} \, | t_2 - s_2 |^{2H}
+o \left( \sum_{i = 1}^n \Big[
t_1 + t_2^2 + s_1 + s_2^2
\Big] \right),
\end{multline*}
which verifies~\ref{A2.1} and~\ref{A2.2}. Finally, \(\mathcal{F} \times \mathcal{F}\)
contains one element \((2, 2)\), and by~\eqref{fbm-2-order} we have that
\begin{equation*}
A_{6, 2, 2} + A_{1, 2} \Sigma^{-1} A_{1, 2}^\top
= A_{1, 2} \Sigma^{-1} A_{1, 2}
= C_2 \, C_2^\top
\quad \text{with} \quad
C_2 = A_{1, 2} \, \Sigma^{-1/2},
\end{equation*}
so~\ref{A2.6} is satisfied.
\end{proof}
\section*{Acknowledgements}
\label{sec:org659376a}
The authors kindly acknowledge the financial support by SNSF Grant
200021-196888.
\section*{References}
\label{sec:orgcc09156}
\printbibliography[heading=none]
\end{document}